\documentclass[12pt]{article}
\usepackage[margin=2.2cm]{geometry}                

\usepackage{hyperref}
\hypersetup{
    colorlinks=true,
    allcolors=violet
}

\usepackage{lmodern}

\usepackage{lscape}
\usepackage{graphicx}
\usepackage{amssymb}
\usepackage{mathtools}
\usepackage{amsmath}
\usepackage{amsthm}
\usepackage{epstopdf}
\usepackage{pifont}
\usepackage{youngtab}
\usepackage{bm}
\usepackage{tikz}
\usetikzlibrary{calc}
\usetikzlibrary{arrows}
\usetikzlibrary{arrows.meta}
\usetikzlibrary{shapes.arrows}
\usetikzlibrary{decorations.pathmorphing}
\usepackage{url}
\usepackage{array}
\begingroup
    \makeatletter
    \@for\theoremstyle:=definition,remark,plain\do{%
        \expandafter\g@addto@macro\csname th@\theoremstyle\endcsname{%
            \addtolength\thm@preskip\parskip
            }%
        }
\endgroup

\allowdisplaybreaks

\newtheorem{thm}{Theorem}[section]
\newtheorem{prop}[thm]{Proposition}

\newtheorem{lem}[thm]{Lemma}

\newtheorem{conj}[thm]{Conjecture}

\theoremstyle{definition}
\newtheorem{defn}[thm]{Definition}
\newtheorem{eg}[thm]{Example}
\newtheorem{rmk}[thm]{Remark}

\newcommand{\Aa}{\mathbb{A}}
\newcommand{\CC}{\mathbb{C}}

\newcommand{\PP}{\mathbb{P}}
\newcommand{\QQ}{\mathbb{Q}}
\newcommand{\RR}{\mathbb{R}}
\newcommand{\TT}{\mathbb{T}}
\newcommand{\VV}{\mathbb{V}}
\newcommand{\ZZ}{\mathbb{Z}}

\newcommand{\pB}{\mathsf{B}}

\newcommand{\pD}{\mathsf{D}}

\newcommand{\sD}{\mathcal{D}}
\newcommand{\sE}{\mathcal{E}}

\newcommand{\sO}{\mathcal{O}}
\newcommand{\sU}{\mathcal{U}}

\DeclareMathOperator{\Proj}{Proj}
\DeclareMathOperator{\Spec}{Spec}

\DeclareMathOperator{\Hom}{Hom}
\DeclareMathOperator{\Gr}{Gr}

\DeclareMathOperator{\OGr}{OGr}

\DeclareMathOperator{\SO}{SO}

\DeclareMathOperator{\Dih}{Dih}

\DeclareMathOperator{\Newt}{Newt}

\DeclareMathOperator{\SL}{SL}
\DeclareMathOperator{\Pf}{Pf}
\DeclareMathOperator{\Cl}{Cl}

\DeclareMathOperator{\Bir}{Bir}

\DeclareMathOperator{\MM}{MM}
\DeclareMathOperator{\NEbar}{\overline{NE}}
\DeclareMathOperator{\ord}{ord}

\DeclareMathOperator{\dP}{dP}
\DeclareMathOperator{\vol}{vol}

\DeclareMathOperator{\divv}{div}
\DeclareMathOperator{\Int}{int}

\DeclareMathOperator{\trop}{trop}
\DeclareMathOperator{\conv}{conv}

\title{The 3-dimensional Lyness map and a \\ self-mirror log Calabi--Yau 3-fold}
\author{Tom Ducat}
\date{May 17, 2021}                                           

\begin{document}

\maketitle

\begin{abstract}
The 2-dimensional Lyness map is a 5-periodic birational map of the plane which may famously be resolved to give an automorphism of a log Calabi--Yau surface, given by the complement of an anticanonical pentagon of $(-1)$-curves in a del Pezzo surface of degree 5. This surface has many remarkable properties and, in particular, it is mirror to itself. We construct the 3-dimensional big brother of this surface by considering the 3-dimensional Lyness map, which is an 8-periodic birational map. The variety we obtain is a special (non-$\mathbb Q$-factorial) affine Fano 3-fold of type $V_{12}$, and we show that it is a self-mirror log Calabi--Yau 3-fold.
\end{abstract}

\setcounter{tocdepth}{1}
\tableofcontents

\section{Introduction}

\subsection{The Lyness map}

\begin{defn}
The \emph{$d$-dimensional Lyness map} $\sigma_d\in\Bir(\CC^d)$ is the birational map
\[ \sigma_d\left(x_1,x_2,\ldots,x_{d-1},x_d\right) = \left(x_2,x_3,\ldots,x_d,\frac{1+x_2+\cdots+x_d}{x_1}\right). \]
By iterating $\sigma_d^{\pm1}$, we can consider an associated sequence of rational functions $(x_i : i \in \ZZ)$, where $x_i\in\CC(x_1,\ldots,x_d)$ is defined inductively for $i>d$ and $i<1$ by using the \emph{$d$-dimensional Lyness recurrence relation}
\begin{equation}\tag{LR$_d$}\label{eqn!rr-d}
x_ix_{d+i} = 1 + x_{i+1} + \cdots + x_{d+i-1}, \qquad \forall i\in\ZZ.
\end{equation}
\end{defn}

\paragraph{Low dimensional behaviour.}
In dimensions $d\leq3$ the recurrence relation (LR$_d$) is very well behaved, with two very nice and surprising properties: it is \emph{periodic}, and the sequence of rational functions $(x_i:i\in\ZZ)$ possesses a \emph{Laurent phenomenon}. In other words, this sequence is actually a sequence of Laurent polynomials $x_i\in\CC[x_1^{\pm1},\ldots,x_d^{\pm1}]$. 

\paragraph{Higher dimensional behaviour.}
Unfortunately, for all $d\geq 4$ the recurrence relation (LR$_d$) is much more badly behaved, something that is also reflected in the dynamical behaviour of~$\sigma_d$. It is no longer periodic; nor does it have a Laurent phenomenon. Because of this there seems to be no straightforward way of generalising the very attractive log Calabi--Yau varieties (or their scattering diagrams) described in this paper when $d=2$ or~$3$.

\subsection{A tale of two log Calabi--Yau varieties}

The two recurrences (LR$_2$) and (LR$_3$) can be used to construct two affine log Calabi--Yau varieties. The first is the famous affine del Pezzo surface of degree 5; the second is an affine Fano 3-fold of type $V_{12}$.\footnote{Recall that a Fano 3-fold of type $V_{12}$ is obtained by taking the intersection of the orthogonal Grassmannian $\OGr(5,10)\subset\PP^{15}$ in its spinor embedding with a linear subspace of codimension 7.} As we will recall in \S\ref{sect!mirror}, affine log Calabi--Yau varieties (satisfying some suitable technical hypotheses) are expected to enjoy some remarkable properties coming from mirror symmetry. In particular, there is a conjectural involution on the set of such varieties which, for a given log Calabi--Yau variety $U$, associates a \emph{mirror} $U^\star$. According to a conjecture of Gross, Hacking \& Keel (cf.\ Conjecture~\ref{conj!GHK}), one feature of the relationship between $U$ and $U^\star$ is that there is expected to be a special additive basis of the coordinate ring of $U^\star$, called the basis of \emph{theta functions}, whose elements correspond (roughly speaking) to boundary divisors in compactifications of $U$, and vice-versa.

\subsubsection{Dimension 2 and the del Pezzo surface $\dP_5$} \label{sect!LR2-terms}
In dimension 2 the recurrence relation (LR$_2$) generates a 5-periodic sequence $(x_i:i\in \ZZ/5\ZZ)$. In terms of $x_1$ and $x_2$, the solution is given by
\[ x_3 = \frac{x_2+1}{x_1}, \quad x_4 = \frac{x_1 + x_2 + 1}{x_1x_2} \quad \text{and} \quad x_5 = \frac{x_1+1}{x_2}. \]
These are well-known as the five cluster variables appearing in the simplest nontrivial cluster algebra: the cluster algebra of type $A_2$. As we recall in \S\ref{sect!dP5}, the associated \emph{cluster variety} is an affine surface $U\subset\Aa^5$ whose coordinate ring is generated by $x_1,\ldots,x_5$ and can be realised as the interior of the log Calabi-Yau pair $(X,D)$, where $X\subset\PP^5$ is a del Pezzo surface of degree 5 and $D=\sum_{i=1}^5D_i$ is an anticanonical pentagon of $(-1)$-curves. The variety $U$ is known to be self-mirror \cite[Example 5.9]{hk}, i.e.\ the mirror $U^\star$ is isomorphic to $U$. In terms of the conjecture mentioned above, this mirror correspondence places the five boundary components $D_i$ in a natural one-to-one correspondence with the five cluster variables $x_i$. 

The mirror of a complex projective Fano variety is expected to be a \emph{Landau--Ginzburg model} $(V,w)$, i.e.\ a quasiprojective algebraic variety $V$ over $\CC$ equipped with a holomorphic surjective function $w\colon V\to \CC$, called a \emph{Landau--Ginzburg potential}. Therefore, for our log Calabi--Yau variety $U$, adding the boundary divisor $D$ to obtain the del Pezzo surface $X$ corresponds to decorating the mirror $U^\star\cong U$ with a holomorphic function. In this case, summing the five cluster variables that correspond to the components of $D$ defines a potential $w=x_1+\cdots+x_5$ for a Landau--Ginzburg model $w\colon U\to\CC$ which is mirror to $X$.

\begin{figure}[htbp]
\begin{center}
\begin{tikzpicture}[scale=2]
   \draw[thick] ({sin(0)},{cos(0)}) -- ({sin(72)},{cos(72)}) -- ({sin(144)},{cos(144)}) -- ({sin(216)},{cos(216)}) -- ({sin(288)},{cos(288)}) -- cycle;
   \node at ({sin(0+36)},{cos(0+36)}) {$x_1$};
   \node at ({sin(72+36)},{cos(72+36)}) {$x_2$};
   \node at ({sin(144+36)},{cos(144+36)}) {$x_3$};
   \node at ({sin(216+36)},{cos(216+36)}) {$x_4$};
   \node at ({sin(288+36)},{cos(288+36)}) {$x_5$};
   
   \begin{scope}[xshift = 3cm]
   \draw[gray,dashed] ({cos(45)},{sin(45)}) -- ({cos(135)},{sin(135)}) -- ({cos(225)},{sin(225)}) -- ({cos(315)},{sin(315)}) -- cycle;
   \draw[thick] (1,0) -- (0,1) -- (-1,0) -- (0,-1) -- cycle;
   \draw[thick] ({cos(0)},{sin(0)}) -- ({cos(45)},{sin(45)}) -- ({cos(90)},{sin(90)}) -- ({cos(135)},{sin(135)}) -- ({cos(180)},{sin(180)}) -- ({cos(225)},{sin(225)}) -- ({cos(270)},{sin(270)}) -- ({cos(315)},{sin(315)}) -- cycle;
   \node at (-0.1,-0.1) {\small $q_1$};
   \node at (0.58,0.58) {\small $x_3$};
   \node at (0.58,-0.58) {\small $x_5$};
   \node at (-0.58,-0.58) {\small $x_7$};
   \node at (-0.58,0.58) {\small $x_1$};
   \node at (0.1,0.1) {\small \textcolor{gray}{$q_2$}};
   \node at (0,0.82) {\small \textcolor{gray}{$x_2$}};
   \node at (0.82,0) {\small \textcolor{gray}{$x_4$}};
   \node at (0,-0.82) {\small \textcolor{gray}{$x_6$}};
   \node at (-0.82,0) {\small \textcolor{gray}{$x_8$}};
   \node at (-1,1) {$P$};
   \end{scope}
   
   \begin{scope}[xshift = 6cm]
   \draw[gray,dashed] ({cos(45)},{sin(45)}) -- ({cos(225)},{sin(225)}) ({cos(135)},{sin(135)}) -- ({cos(315)},{sin(315)});
   \draw[thick] (1,0) -- (-1,0) (0,1) -- (0,-1);
   \draw[thick] ({cos(0)},{sin(0)}) -- ({cos(45)},{sin(45)}) -- ({cos(90)},{sin(90)}) -- ({cos(135)},{sin(135)}) -- ({cos(180)},{sin(180)}) -- ({cos(225)},{sin(225)}) -- ({cos(270)},{sin(270)}) -- ({cos(315)},{sin(315)}) -- cycle;
   \node at ({cos(45+22.5)*2/3},{sin(45+22.5)*2/3}) {\small $x_3$};
   \node at ({cos(45+90+22.5)*2/3},{sin(45+90+22.5)*2/3}) {\small $x_1$};
   \node at ({cos(45+180+22.5)*2/3},{sin(45+180+22.5)*2/3}) {\small $x_7$};
   \node at ({cos(45-90+22.5)*2/3},{sin(45-90+22.5)*2/3}) {\small $x_5$};
   \node at ({cos(0+22.5)*2/3},{sin(0+22.5)*2/3}) {\small \textcolor{gray}{$x_4$}};
   \node at ({cos(90+22.5)*2/3},{sin(90+22.5)*2/3}) {\small \textcolor{gray}{$x_2$}};
   \node at ({cos(180+22.5)*2/3},{sin(180+22.5)*2/3}) {\small \textcolor{gray}{$x_8$}};
   \node at ({cos(270+22.5)*2/3},{sin(270+22.5)*2/3}) {\small \textcolor{gray}{$x_6$}};
   \node at (-1,1) {$Q$};
   \end{scope}
\end{tikzpicture}
\caption{The correspondence between boundary divisors and theta functions on the mirror for (a) del Pezzo surface $X$ of degree~5, (b) the Fano 3-fold $X_P$ of type $V_{12}$ and (c) the Fano 3-fold $X_Q$ of type $V_{16}$.}
\label{fig!bdries}
\end{center}
\end{figure}
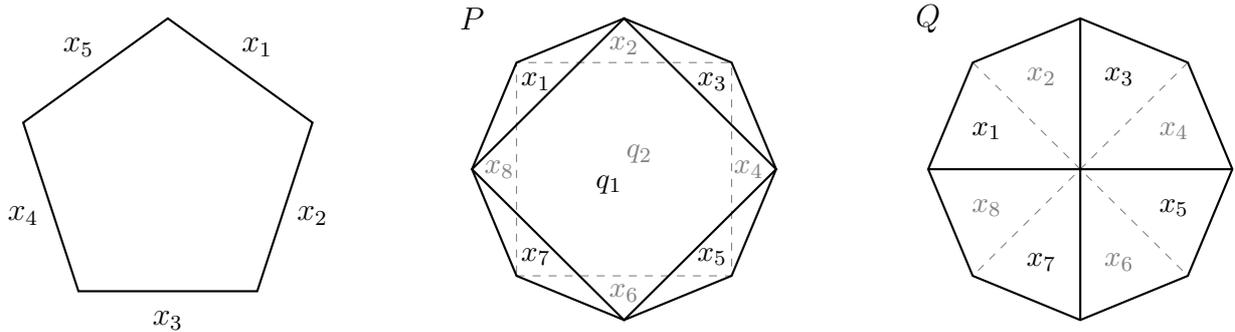

\subsubsection{Dimension 3 and the Fano 3-fold $V_{12}$} \label{sect!LR3-terms}
In dimension 3 the recurrence relation (LR$_3$) generates an 8-perodic sequence $\left(x_i :i\in \ZZ/8\ZZ \right)$. Expanding in terms of $x_1$, $x_2$, $x_3$ we have
\[ x_4 = \frac{x_2+x_3+1}{x_1} \qquad x_5 = \frac{(x_1 + 1)(x_3 + 1) + x_2}{x_1x_2} \qquad x_6 = \frac{(x_1 + x_2 + 1)(x_2 + x_3 + 1)}{x_1x_2x_3} \]
\[ x_7 = \frac{(x_1 + 1)(x_3 + 1) + x_2}{x_2x_3} \qquad \text{and} \qquad x_8 = \frac{x_1+x_2+1}{x_3}. \]
These rational functions satisfy eight relations of the form $x_{i-1}x_{i+2}=1+x_i+x_{i+1}$ and we observe that $x_1,\ldots,x_8$ satisfy two further relations: $x_1x_5=x_3x_7$ and $x_2x_6=x_4x_8$. In analogy to the dimension 2 case, we will see in \S\ref{sect!OGr} that the ring generated by the eight variables $x_1,\ldots,x_8$ defines an algebraic variety $U\subset\Aa^8$ which is also a kind of cluster variety (in the sense of Definition~\ref{defn!cluster-variety}(3)). This variety $U$ is the interior of a log Calabi--Yau pair $(X_P,D_P)$, where $X_P$ is a special (non-$\QQ$-factorial) Fano 3-fold of type $V_{12}$ and $D_P\in|{-K_X}|$ is a boundary divisor with ten components. Eight components are isomorphic to $\PP^2$, two components are isomorphic to $\PP^1\times\PP^1$ and they meet along their toric boundary strata according to the polytope $P$ shown in Figure~\ref{fig!bdries}(b). 

Like the affine del Pezzo surface before it, we will see that this 3-fold $U$ is also self-mirror. However, in contrast to the two dimensional case, there is now a problem: we have ten boundary divisors but only eight cluster variables. Moreover the Landau--Ginzburg model $w_P\colon U\to \CC$ defined by the potential $w_P:=x_1+\ldots+x_8$ does not have the right period sequence to be a mirror for a Fano 3-fold of type $V_{12}$. Instead the period sequence of $w_P$ suggests that it is a mirror to a Fano 3-fold of type $V_{16}$ (cf. \S\ref{sec!LG-V12-V16}).

The explanation for this difference is that the eight $\PP^2$ components of $D_P$ correspond to the eight cluster variables $x_i$, but the two $\PP^1\times\PP^1$ components correspond to two `missing' cluster variables 
\[ q_1 = x_1x_5 - 1 = \frac{(x_1 + 1)(x_3 + 1)}{x_2} \quad \text{and} \quad q_2 = x_2x_6 - 1 =\frac{(x_1 + x_2 + x_3 + 1)(x_2 + 1)}{x_1x_3}, \]
which are invariants for the squared Lyness map $\sigma_3^2$. The new variables $q_1$ and $q_2$ appear in the derivation of a cluster algebra-like exchange graph for $U$ (cf.\ \S\ref{sec!U-exchange-graph}). Adjoining $q_1$ and $q_2$ to the ring $\CC[X_P]$ corresponds to a birational map (an \emph{unprojection} in the languange of Miles Reid) mapping $X_P$ onto a 3-fold $X_Q$
\[ \psi\colon X_P\subset \PP^8_{(x_0:x_1:\cdots:x_8)} \dashrightarrow X_Q\subset \PP^{10}_{(x_0:x_1:\cdots:x_8:q_1:q_2)}  \]
which blows up and then contracts the two (non-Cartier) Weil divisors in the boundary of $X_P$ corresponding to $q_1$ and $q_2$. This extends to a birational map of pairs $\psi\colon (X_P,D_P)\dashrightarrow (X_Q,D_Q)$, where $X_Q$ is a degenerate Fano 3-fold of type $V_{16}$ and $D_Q$ is an anticanonical boundary divisor with eight $\PP^1\times\PP^1$ components meeting according to the polytope $Q$ displayed in Figure~\ref{fig!bdries}(c). 

Finally, summing the ten cluster variables that generate the coordinate ring of $X_Q$ gives a potential $w_Q:=x_1+\cdots+x_8+q_1+q_2$ for a Landau--Ginzburg model $w_Q\colon U\to \CC$ with the correct period sequence to be mirror to a Fano 3-fold of type $V_{12}$. Thus in this example the Landau--Ginzburg model $(U,w_P)$ is mirror to the Fano variety $X_Q$ and the Landau--Ginzburg model $(U,w_Q)$ is mirror to the Fano variety $X_P$.

\subsubsection{Generalising Borisov--Batyrev duality}

We note that the two polytopes $P$ and $Q$ appearing in Figure~\ref{fig!bdries} are combinatorially dual to one another. We will go one step further and interpret them as dual reflexive polytopes in the \emph{tropicalisation of $U$}, a self-dual integral affine manifold with singularities, which we denote by $N_U$. This duality can then be generalised for any pair of dual reflexive polytopes in $N_U$ and is the analogue of Borisov--Batyrev duality in the ordinary toric setting.

\subsection{Summary}

\subsubsection{Contents}

In \S\ref{sect!mirror} we give a brief overview of mirror symmetry for log Calabi--Yau varieties, which is intended as motivational context for the rest of the paper. In particular we explain how to construct the \emph{tropicalisation} $N_U$ of a log Calabi--Yau variety $U$ which is an integral affine manifold with singularities. We explain how to construct polytopes inside this space $N_U$ and how they can be dualised to give polytopes in $N_V$, a dual integral affine manifold with singularities corresponding to the mirror log Calabi--Yau variety $V=U^\star$. 

In \S\ref{sect!dP5} we discuss the well-known affine del Pezzo surface $U$ of degree 5 related to the 2-dimensional Lyness recurrence (LR$_2$). In particular we describe the cluster structure on $U$, the tropicalisation $N_U$ and we show that $N_U$ is self-dual as an integral affine manifold with singularities. We then give some examples of mirror symmetry for pairs of reflexive polygons in $N_U$. In \S\ref{sect!OGr} we repeat all of this to do a similar analysis for the affine 3-fold of type $V_{12}$ related to the 3-dimensional Lyness recurrence (LR$_3$).

\subsubsection{Acknowledgements} 
I would like to thank Alessio Corti, Alastair King, Andy Hone and Thomas Prince, amongst a number of others, for helpful discussions on the subject of this paper. In particular, Alessio Corti first proposed the geometric definition of a seed described in \S\ref{sec!mutation}, Alastair King suggested formulating the material in \S\ref{sect!OGr-rep-thry} using the spin representation of $\pB_4$ (rather than the half-spin representation of $\pD_5$) and Thomas Prince helped derive the scattering diagram described in \S\ref{sec!V12-scattering}. 

\section{Mirror symmetry and log Calabi--Yau varieties} \label{sect!mirror}

This section is intended as a motivational context for the rest of the paper and as such we refer the reader to \cite{bir-geom-cluster,hk} for more detailed accounts.

\subsection{Log Calabi--Yau varieties}

\begin{defn}\label{defn!log-CY}
A \emph{log Calabi--Yau pair} is a pair $(X,D)$ consisting of a smooth complex projective variety $X$ and a reduced effective integral anticanonical divisor $D\in |{-K_X}|$ with simple normal crossings.\footnote{The smoothness and simple normal crossings assumptions are made purely for simplicity. More generally one could consider $(X,D)$ to be a pair with $\QQ$-factorial divisorial log terminal singularities (although we note that the `natural' compactification of the 3-fold cluster variety appearing in \S\ref{sect!OGr} is not $\QQ$-factorial).} We call the interior of a log Calabi--Yau pair $U=X\setminus D$ a \emph{log Calabi--Yau variety}. 
\end{defn}

A log Calabi--Yau variety $U$ is naturally equipped with a nonvanishing holomorphic volume form $\Omega_U$ in the following way. The boundary divisor $D = \divv s_X$ is cut out by a section $s_X\in H^0(X,{-K_X})$ so, after restricting to $U$, we get the volume form $\Omega_U = \left({s_X}|_U\right)^{-1}\in H^0(U,K_U)$ which has simple poles along the components of $D$. Since $\Omega_U$ is only determined up to rescaling by a constant, we may assume that $\Omega_U$ is normalised so that $\int_U\Omega_U=1$. The natural framework within which we would now like to work is the category of log Calabi--Yau varieties and \emph{volume preserving birational maps}, i.e.\ birational maps $\phi\colon U \dashrightarrow V$ such that $\phi^*\Omega_V=\Omega_U$.

\begin{eg}
The prototypical example of a log Calabi--Yau variety is the algebraic torus $\TT_N = \CC^\times\otimes_\ZZ N$ associated to a lattice $N\cong \ZZ^d$. The volume form on $\TT_N$ is given by $\Omega_{\TT_N} = (2\pi i)^{-d}\frac{dz_1}{z_1}\wedge\cdots\wedge \frac{dz_d}{z_d}$. There are two dual lattices associated to $\TT_N$: the \emph{cocharacter lattice} $N$, whose points correspond to divisorial valuations along toric boundary divisors, and the \emph{character lattice} $M=\Hom(N,\ZZ)$, whose points $m\in M$ correspond to monomial functions $z^m=z_1^{m_1}\ldots z_d^{m_d}$ on $\TT_N=\Spec \CC[M]$. These monomials form a natural additive basis for the coordinate ring of $\TT_N$. 
\end{eg}

Note that the roles of the two lattices $N$ and $M$ are interchanged when we replace the torus $\TT_N$ by its dual (or `mirror') torus $\TT_M$. A key idea of Gross, Hacking \& Keel is that one can introduce an object that serves as a generalisation of the cocharacter lattice for an arbitrary log Calabi--Yau variety.

\begin{defn}
Given a log Calabi--Yau variety $U$ with volume form $\Omega_U$, the set of \emph{integral tropical points} of $U$ is given by
\[ N_U(\ZZ) = \left\{ \text{divisorial valuations } \nu\colon \CC(U)\setminus \{0\} \to \ZZ \text{ such that } \nu(\Omega_U) <0 \right\} \cup \{0\}. \]
\end{defn}

The set $N_U(\ZZ)$ is more commonly referred to as $U^{\trop}(\ZZ)$ in the literature, but our notation is chosen to emphasise the fact that $N_U(\ZZ)$ is supposed to be a generalisation of the cocharacter lattice $N$ when $U=\TT_N$. Indeed, as a set we have that $N_{\TT_N}(\ZZ)=N$. 

\subsubsection{The conjecture of Gross, Hacking \& Keel}

Gross, Hacking \& Keel have conjectured that an analogue for the character lattice also exists for $U$, which generalises the duality enjoyed by tori to pairs of mirror log Calabi--Yau varieties. Before stating their conjecture we need to introduce two technical hypotheses.

\begin{defn}
Let $(X,D)$ be a $d$-dimensional log Calabi--Yau pair with interior $U=X\setminus D$. We call $U$ \emph{positive} if $D$ supports an ample divisor and we say that $U$ has \emph{maximal boundary} if $D$ contains a 0-stratum (i.e.\ a point at which $d$ components of $D$ meet transversely).
\end{defn}

From now on we will assume that all log Calabi--Yau varieties are positive and have maximal boundary. The assumption that $U$ is positive is useful because it implies that $U$ is affine. The maximal boundary condition is introduced to ensure that the set of integral tropical points $N_U(\ZZ)$ is as big as possible. In \S\ref{sec!trop} we will realise $N_U(\ZZ)$ as the set of integral points in a real affine manifold with singularities $N_U$. In general, the dimension of $N_U$ will be equal to the codimension of the smallest stratum of $D$.

\begin{conj}[{\cite[Conjecture 0.6]{ghk}, cf.\ \cite{bir-geom-cluster,hk}}] \label{conj!GHK}
Given a positive log Calabi--Yau variety $U$ with maximal boundary, then there exists a mirror positive log Calabi--Yau variety $V$ with maximal boundary whose coordinate ring $\CC[V]$ has an additive basis 
\[ \mathcal B_V = \{ \vartheta_n \in \CC[V] : n \in N_U(\ZZ) \} \]
parameterised by the integral tropical points of $U$. The elements of this basis are called the \emph{theta functions} of $V$ and are canonically determined up to multiplication by scalars. The multiplication rule for theta functions $\vartheta_a\vartheta_b=\sum_{c\in N_U} \alpha_{abc}\vartheta_c$ has structure constants $\alpha_{abc}$ which can be obtained as certain counts of rational curves in $U$.
\end{conj}

The mirror variety $V$ is not expected to be unique; indeed $V$ is expected to appear as a general fibre in a whole family of mirrors to $U$. Mirror symmetry is then an involution in the sense that the mirror of $V$ is a family of log Calabi--Yau varieties deformation equivalent to $U$. If $U$ and $V$ are a pair of mirror log Calabi--Yau varieties as in the statement of the conjecture then let us write $M_U(\ZZ):=N_V(\ZZ)$ and $M_V(\ZZ):=N_U(\ZZ)$ (indicating that $N_V(\ZZ)$ is supposed to be a generalisation of the character lattice for $U$, and vice-versa). 

\subsection{Cluster varieties} \label{sect!cluster-varieties}

\subsubsection{Toric models}

We can create new examples of log Calabi--Yau varieties from a given log Calabi--Yau pair $(X,D)$ by considering volume preserving blowups of $X$. 

\begin{defn} \label{defn!cluster-variety}
Consider a log Calabi--Yau pair $(X,D)$ with boundary divisor $D=\sum_{i=1}^kD_i$ and interior $U=X\setminus D$.
\begin{enumerate}
\item A \emph{toric blowup} of $(X,D)$ is a blowup of $X$ along a stratum of $D$. A \emph{nontoric blowup} of $(X,D)$ is a blowup of $X$ along a smooth subvariety $Z\subset X$ of codimension 2, where $Z$ is contained in one of the boundary components $Z\subset D_i$ and meets the other components of $D$ transversely. 
\item We call a log Calabi--Yau pair $(X,D)$ a \emph{toric model} for $U$ if there exists a map $\pi\colon (X,D) \to (T,B)$ to a toric pair $(T,B)$ where $\pi$ is a composition of nontoric blowups.
\item We call $U$ a \emph{(generalised) cluster variety} if it has a toric model.
\end{enumerate}
\end{defn}

We use the terminology `\emph{generalised} cluster variety' since the more commonly accepted definition of a cluster variety (cf.\ \cite[Definition 3.1]{hk}) imposes a second condition: in addition to the toric model $\pi\colon (X,D)\to (T,B)$, the variety $U$ is required to have a nondegenerate holomorphic 2-form $\sigma\in H^0\left(X,\Omega^2_X(\log D)\right)$. The existence of this 2-form $\sigma$ imposes a very strong condition on types of centre $Z\subset T$ that can be blown up by $\pi$. In particular it forces $Z=\bigcup_{i=1}^k Z_i$ to be a union of \emph{hypertori}. If we let $N$ and $M$ be the (co)character lattices of the torus $\TT=T\setminus B$, then a hypertorus is a subvariety of $T$ of the form $Z_i = \VV(z^{m_i} + \lambda_i)\subset B_{n_i}$, where $\lambda_i\in\CC^\times$ is some coefficient (usually assumed to be chosen generically), $n_i\in N$ determines the boundary divisor $B_{n_i}\subset T$ containing $Z_i$ and $m_i\in n_i^\perp\subset M$. Thus the centres $Z_i$ are determined by pairs $(n_i,m_i)\in N\times M$. The advantage of this more restrictive situation is that there is a very simple candidate for the mirror to $U$, called the \emph{Fock--Goncharov dual} of $U$: for each $i=1,\ldots,k$ we simply swap the roles of $n_i\in N$ and $m_i\in M$ to obtain a `mirror' toric model.

Nevertheless, as we will soon see, the cluster algebra-like combinatorics of seeds and mutation can be used to describe any log Calabi--Yau variety with a toric model, and do not rely on the 2-form $\sigma$. This fact has led Corti to suggest that this should be the `true' definition of a cluster variety.

\begin{rmk}
The name `generalised cluster variety' is not ideal, since there are many other proposals for a `generalised cluster variety' appearing in the literature. Of these, it is perhaps closest in spirit to the definition of a \emph{Laurent phenomenon algebra} given by Lam \& Pylyavskyy \cite{lp}. Translated into our language, a Laurent phenomenon algebra is essentially the coordinate ring of a $d$-dimensional log Calabi--Yau variety with a toric model given by blowing up $d$ centres $Z_i\subset H_i \subset \Aa^d$, one in each of the coordinate hyperplanes $H_i$ of $\Aa^d$.  
\end{rmk}

\subsubsection{Mutations} \label{sec!mutation}

A toric model $\pi\colon (X,D)\to(T,B)$ for $U$ is determined by a set of centres 
\[ S=\{ Z_i \subset T : i = 1,\ldots,k \} \]
which comprise the locus $Z=\bigcup_{i=1}^kZ_i$ blown up by $\pi$. Note that the choice of toric model for $U$ is not uniquely determined by the set $S$, since we can modify our chosen pairs $(X,D)$ and $(T,B)$ by compatible choices of toric blowups or blowdowns. However the set $S$ does specify a unique torus $\TT_S:=T\setminus B$ and a volume preserving embedding $j_S:=(\pi^{-1})|_{\TT_S} \colon \TT_S \hookrightarrow U$. (This inclusion of the torus $\TT_S$ into $U$ is the geometric manifestation of the Laurent phenomenon for $U$.)

Similarly to the (strict) cluster case described above, we can represent each component of $Z$ in the form $Z_i = Z(n_i,f_i) := \VV(f_i) \subset B_{n_i}$, where $n_i\in N$ is a primitive vector\footnote{If we wanted to, we could extend the definition of $Z(n_i,f_i)$ to allow arbitrary $n_i\in N$ by considering nonreduced centres $Z(n_i,f_i)=\VV(f_i^r)\subset B_{n_i'}$, where $r\geq1$ and $n_i = rn_i'$ for a primitive vector $n_i'\in N$.} and $f_i\in \CC[n_i^\perp\cap M]$ is a Laurent polynomial. The difference is now that $f_i$ is not constrained to be simply binomial. 

\begin{defn} \leavevmode
\begin{enumerate}
\item We call the set $S$ a \emph{seed} for $U$ and the torus embedding $j_S \colon \TT_S \hookrightarrow U$ a \emph{cluster torus chart} of $U$. 

\item Given a pair $(n,f)$, which represents a centre $Z=Z(n,f)$ in the boundary of a toric compactification of a torus $\TT$, we define the \emph{mutation along $Z$} to be the birational map $\mu_{(n,f)}\colon \TT \dashrightarrow \TT$ of the form $\mu_{(n,f)}^*(z^m) = f^{-\langle n,m\rangle}z^m$.
\end{enumerate}
\end{defn}

\begin{eg}\label{eg!one-mutation}
The geometry of a mutation $\mu=\mu_{(n,f)}$ is described by Gross, Hacking \& Keel \cite[\S3.1]{bir-geom-cluster}. In particular, they show that $\mu$ can be viewed as a kind of \emph{elementary transformation of $\PP^1$-bundles}. Indeed, by blowing up if necessary, we can arrange for $T$ to contain the two boundary divisors $B_+:=B_n$ and $B_-:=B_{-n}$ and then consider the toric subvariety $T_0\cong \left(\PP^1\times (\CC^\times)^{d-1}\right)\subset T$ consisting of the big open torus $\TT$ and the relative interior of $B_+$ and $B_-$. Let $Z_+=Z(n,f)$ and let $E_-:=\VV(f)\subset T_0$. Then they show that the extension of $\mu$ to a birational map $\mu\colon T_0\dashrightarrow T_0$ is resolved by blowing up the locus $Z_+\subset B_+$ and contracting the strict transform of the divisor $E_-$ to the locus $Z_-=Z({-n},f)\subset B_-$.
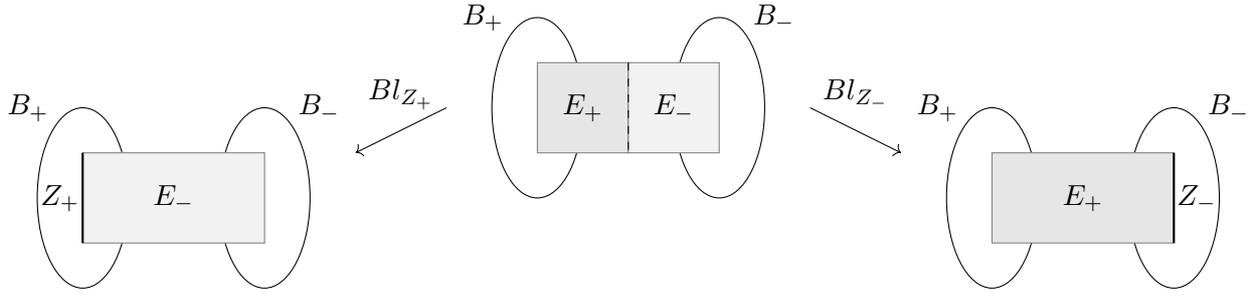
\begin{center}
\resizebox{\textwidth}{!}{\begin{tikzpicture}[scale=0.6,font=\small]
   \draw (0,0) ellipse (1cm and 2cm);
   \draw (4,0) ellipse (1cm and 2cm);
   \draw[gray,fill=gray!10] (0,-1) -- (0,1) -- (4,1) -- (4,-1) -- cycle;
   \draw[thick] (0,1) -- (0,-1);
   \node at (-1.2,2) {$B_+$};
   \node at (5.2,2) {$B_-$};
   \node at (-0.5,0) {$Z_+$};
   \node at (2,0) {$E_-$};
   
   \begin{scope}[xshift=10cm,yshift=2cm]
   \draw (0,0) ellipse (1cm and 2cm);
   \draw (4,0) ellipse (1cm and 2cm);
   \draw[gray,fill=gray!20] (0,-1) -- (0,1) -- (2,1) -- (2,-1) -- cycle;
   \draw[gray,fill=gray!10] (4,-1) -- (4,1) -- (2,1) -- (2,-1) -- cycle;
   \draw[dashed] (2,1) -- (2,-1);
   \node at (-1.2,2) {$B_+$};
   \node at (5.2,2) {$B_-$};
   \node at (1,0) {$E_+$};
   \node at (3,0) {$E_-$};
   \end{scope}
   
   \begin{scope}[xshift=20cm]
   \draw (0,0) ellipse (1cm and 2cm);
   \draw (4,0) ellipse (1cm and 2cm);
   \draw[gray,fill=gray!20] (0,-1) -- (0,1) -- (4,1) -- (4,-1) -- cycle;
   \draw[thick] (4,1) -- (4,-1);
   \node at (-1.2,2) {$B_+$};
   \node at (5.2,2) {$B_-$};
   \node at (4.5,0) {$Z_-$};
   \node at (2,0) {$E_+$};
   \end{scope}
   
   \draw[->] (8,2) -- (6,1); \node at (7,2.3) {$Bl_{Z_+}$};
   \draw[->] (16,2) -- (18,1); \node at (17,2.3) {$Bl_{Z_-}$};
\end{tikzpicture}}\end{center}
Changing coordinates so that $n=(1,0,\ldots,0)$, then $\mu$ can be assumed to be of the form
\[ \mu^*(z_1,z_2,\ldots,z_d) = \left( f(z_2,\ldots,z_d)^{-1}z_1, z_2,\ldots,z_d \right). \]
It is convenient to introduce $z_1' = z_1^{-1}f(z_2,\ldots,z_d)$ and to distinguish the domain and codomain of $\mu\colon \TT\dashrightarrow \TT'$ so that $\TT=(\CC^\times)^d_{z_1,z_2,\ldots,z_d}$, $\TT'=(\CC^\times)^d_{z_1',z_2,\ldots,z_d}$ and $\mu^*(z_1',z_2,\ldots,z_d) = \left(z_1^{-1},z_2,\ldots,z_d\right)$. However note that this map is \emph{volume negating}, in the sense that $\mu^*\Omega_{\TT'}=-\Omega_{\TT}$, so the torus $\TT'$ should be considered to come equipped with the negative of its standard volume form. Now $U = Bl_{Z_+}T_0\setminus \left(B_+\cup B_-\right)$ is the affine variety
\[ U = \VV\big(z_1z_1' - f(z_2,\ldots,z_d)\big)\subset \Aa^2_{z_1,z_1'} \times (\CC^\times)^{d-1}_{z_2,\ldots,z_d} \]
and this is covered, up to the complement of a subset $\Sigma=\VV(z_1,z_2,f)\subset U$ of codimension two in $U$, by the two torus charts $\TT,\TT'\hookrightarrow U$ which are glued together by $\mu$. The locus $\Sigma \subset U$ not covered by the two torus charts is the intersection of the two divisors $\Sigma=E_+\cap E_-$, although since $U$ is normal and affine we have that $U=\Spec H^0(U\setminus \Sigma,\mathcal{O}_{U\setminus \Sigma})$, so these two charts provide enough information to recover $U$.
\end{eg}

\paragraph{Mutating a seed.}
Example~\ref{eg!one-mutation} gives a complete description of the cluster structure for a log Calabi--Yau variety $U$ which is determined by a seed $S=\{Z(n,f)\}$ with only one centre. In particular there are two cluster torus charts $\TT_S=\TT$ and $\TT_{S'}=\TT'$ where $S'$ is the seed $S=\{Z(-n,f)\}$, and the mutation $\mu\colon \TT_S\dashrightarrow \TT_{S'}$ provides the transition map between them.

In general, for a seed with several centres $S = \{ Z_i=Z(n_i,f_i) : i \in 1,\ldots,k \}$ we can consider applying a mutation $\mu_i=\mu_{(n_i,f_i)}$ for each $i\in\{1,\ldots,k\}$. For a given $i$, we can define a new seed $\mu_i S$, such that the mutation becomes a map of cluster torus charts $\mu_i\colon \TT_S \dashrightarrow \TT_{\mu_iS}$. To do this, extend $\mu$ to a birational map of projective toric varieties $\mu\colon (T,B)\dashrightarrow (T',B')$. The \emph{mutated seed} is the set of centres $\mu_iS=\{\mu_iZ_j : j \in 1,\ldots k\}$, such that
\[ \mu_iZ_j = \begin{cases}
Z(-n_i,f_i) & j=i \\
C_{Z_j}T' & j\neq i
\end{cases} \]
where $C_{Z_j}T' = \overline{\mu_i(\eta_{Z_j})}$ denotes the centre of $Z_j$ on $T'$ (the closure of the image of the generic point $\eta_{Z_j}$ under $\mu_i$). As $\mu_i$ is volume preserving, it follows that $C_{Z_j}T'$ is contained in the boundary of $T'$ and is guaranteed to have the form $\mu_iZ_j=Z(n_j',f_j')$ for some $n_j'\in N$ and $f_j'\in \CC[n_j'\cap M]$.  

\paragraph{The exchange graph of $U$.}
The \emph{exchange graph} for $U$ is the $k$-regular graph whose set of vertices is the set of seeds for $U$ and whose set of edges is given by mutations.

\paragraph{An atlas of torus charts for $U$.}
We can think of the set of seeds for $U$ as giving an atlas of cluster torus charts $j_S\colon \TT_S\hookrightarrow U$ which can be glued together by transition maps which are the mutations. By \cite[Proposition~2.4]{bir-geom-cluster}, from any initial seed $S$ we can consider the scheme 
\[ U^0 = \TT_S \cup \bigcup_{i=1}^k \TT_{\mu_i S}\]
obtained by gluing all the cluster tori which are one mutation away from $S$. In general there are some issues with this construction. In particular, $U^0$ may not be separated if two centres have nonempty intersection $Z_i\cap Z_j\neq\emptyset$. However, if all of the centres $Z_i$ are disjoint then the maps $j_{\mu_iS}\colon \TT_{\mu_iS}\hookrightarrow U$ glue to give an embedding $j\colon U^0\hookrightarrow U$, which covers $U$ up to a set $\Sigma=U\setminus U^0$ of codimension at least $2$ \cite[Lemma~3.5]{bir-geom-cluster}. If $U$ is positive (and hence affine) then we have $U=\Spec H^0(U^0,\sO_{U^0})$ as in Example~\ref{eg!one-mutation}.

\subsubsection{Examples} \label{sec!mutation-examples}

It is natural to wonder whether there is a more explicit combinatorial formula describing the effect of a mutation $\mu_i$ on the set of pairs $(n_j,f_j)$ that define the other centres $Z_j=Z(n_j,f_j)$ in the seed, as there is in the ordinary cluster case \cite[Equation 2.3]{bir-geom-cluster}. Unfortunately, at this level of generality the mutations are somewhat more complicated to keep track of, and the location of the centre $\mu_i Z_j$ depends crucially on the exact form of $f_i$ and $f_j$. It is natural to hope that if $Z_j\subset B_{n_j}$, then $\mu_i {Z_j}\subset \mu_i B_{n_j}$, i.e.\ that $Z_j$ remains in the same boundary component after the mutation. However this is not necessarily the case. We give a couple of examples to see the kind of difficulties that can occur.

\begin{defn}
Consider a seed $S=\{Z_1,\ldots,Z_k\}$, a choice of mutation $\mu_i$ and a centre $Z_j$ with $j\neq i$. We say that $Z_j$ makes an \emph{unexpected jump} if $\mu_i {Z_j}\not\subset \mu_i B_{n_j}$.
\end{defn}

\begin{eg} \label{eg!bad-mutation}
We consider a toric model $\pi\colon(X,D)\to (\Aa^3,B)$ obtained by blowing up three centres which are contained in the coordinate hyperplanes of $\Aa^3$: 
\[ Z_1 = Z\big( (1,0,0), \; 1+z_2 \big), \quad Z_2 = Z\big((0,0,1), \; 1+z_2\big), \quad Z_3 = Z\big((0,0,1), \; 1+z_1\big). \] 
Consider the mutation along $Z_1$. From the geometric description given above, $\mu_1$ blows up $Z_1$ and contracts the strict transform of the divisor $E_1=\VV(1+z_2)\subset \Aa^3$ to the centre $\mu_1Z_1 = Z\left((-1,0,0), \; 1+z_2\right)$ at infinity. Since $\mu_1^*(1+z_1^{-1}+z_2)=z_1^{-1}(1+z_2)$, it is also straightforward to see that the mutation sends $Z_3$ to $\mu_1 Z_3 = Z\left((0,0,1), \; 1+z_1^{-1}+z_2\right)$. However, since $Z_2\subset E_1$, the contraction of the strict transform of $E_1$ moves the centre of $Z_2$ out of the divisor $B_{(0,0,1)}$. We get $\mu_1 Z_2 = Z\left((-1,0,1), \; 1+z_2\right)$ and thus $Z_2$ makes an unexpected jump. We illustrate the effect of the mutation on the centres in the following diagram, where the centres $Z_i$ have been drawn `tropically'.
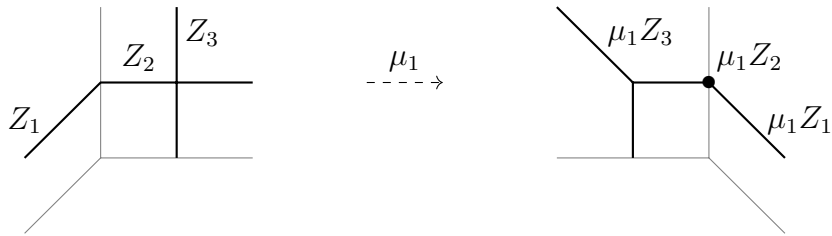
\begin{center}\begin{tikzpicture}[scale=1]
   \begin{scope}[xshift = 0cm]
   \draw[gray] (-1,-1) -- (0,0) -- (0,2) (0,0) -- (2,0);
   \draw[thick] (-1,0) -- (0,1) -- (2,1) (1,0) -- (1,2);   
   \node at (-1,1/2) {$Z_1$};
   \node at (1/2,4/3) {$Z_2$};
   \node at (4/3,5/3) {$Z_3$};
   \end{scope}
   
   \begin{scope}[xshift = 6cm]
   \draw[gray] (0,0) -- (2,0) -- (2,2) (2,0) -- (3,-1);
   \draw[thick] (3,0) -- (2,1) -- (1,1) -- (1,0) (1,1) -- (0,2);
   \node at (2,1) {$\bullet$};
   \node at (3+0.2,1/2) {$\mu_1Z_1$};
   \node at (7/3+0.2,4/3) {$\mu_1Z_2$};
   \node at (1+0.1,5/3) {$\mu_1Z_3$};
   \end{scope}
   
   \draw[dashed,->] (3.5,1) to node[above,midway] {$\mu_1$} (4.5,1);
\end{tikzpicture}\end{center}
\end{eg}

The bad behaviour demonstrated in Example~\ref{eg!bad-mutation} looks like it can be fixed by introducing a generic coefficient $\lambda\in\CC^\times$ and deforming the centre $Z_2$ to $Z\left((0,0,1),\lambda+z_2\right)$. Then we will no longer have the inclusion $Z_2\subset E_1$ and the mutation $\mu_1$ will keep $Z_2$ and $Z_3$ inside the same boundary component. Unfortunately, as we show in the next example, it is not always possible to avoid unexpected jumps, even if we start from a seed where all of the centres have generic coefficients.

\begin{eg}
Consider the toric model $\pi\colon(X,D)\to (\Aa^3,B)$ obtained by blowing up three centres 
\begin{align*} 
Z_1 &= Z\big((1,0,0),\: a_1 + a_2z_2 + a_3z_3\big), \\
Z_2 &= Z\big((0,1,0),\: b_1z_1 + b_2 + b_3z_3\big), \\
Z_3 &= Z\big((0,0,1),\: c_1z_1 + c_2z_2 + c_3\big)  
\end{align*}
where $a_i,b_i,c_i\in\CC^\times$ are generic coefficients. This seed consists of three general lines, one in each of the three coordinate hyperplanes of $\Aa^3$. We can apply the mutation $\mu_1$ along $Z_1$ to obtain a new seed:
\begin{align*} 
Z_1' := \mu_1Z_1 &= Z\big((-1,0,0),\: a_1 + a_2z_2 + a_3z_3 \big), \\
Z_2' := \mu_1Z_2 &= Z\big((0,1,0),\: b_1(a_1 + a_3z_3) + z_1^{-1}(b_2 + b_3z_3)\big), \\
Z_3' := \mu_1Z_3 &= Z\big((0,0,1),\: c_1(a_1 + a_2z_2) + z_1^{-1}(c_2z_2 + c_3)\big).
\end{align*}
We note that $Z_1'$ and $Z_2'$ intersect in a point $p=\VV(z_1^{-1},z_2,a_1+a_3z_3)$ belonging to the intersection of their boundary components. If we now apply the mutation $\mu_2$ along $Z_2'$ then we compute that the total transform of the centre $Z_1'$ is reducible and given by
\[ Z\big((-1,0,0),\: z_2^{-1} + a_2b_1 \big) \cup Z\big((-1,0,0),\: a_1+a_3z_3 \big) \]
where the first component is the centre $\mu_2Z_1'$ and the second component is the exceptional line $p'$ over the point $p$ (represented by the dashed vertical line in the third diagram below). Therefore mutating the seed at $Z_2'$ gives 
\begin{align*} 
Z_1'' := \mu_2Z_1' &= Z\big((-1,0,0),\: z_2^{-1} + a_2b_1 \big), \\
Z_2'' := \mu_2Z_2' &= Z\big((0,-1,0),\: b_1(a_1 + a_3z_3) + z_1^{-1}(b_2 + b_3z_3)\big), \\
Z_3'' := \mu_2Z_3' &= Z\big((0,0,1),\: (a_2c_1 + c_2z_1^{-1})(a_1b_1 + b_2z_1^{-1}) + z_2^{-1}(a_1c_1 + c_3z_1^{-1}) \big).
\end{align*}  
We draw the sequence of mutations with the centres represented tropically, as before.
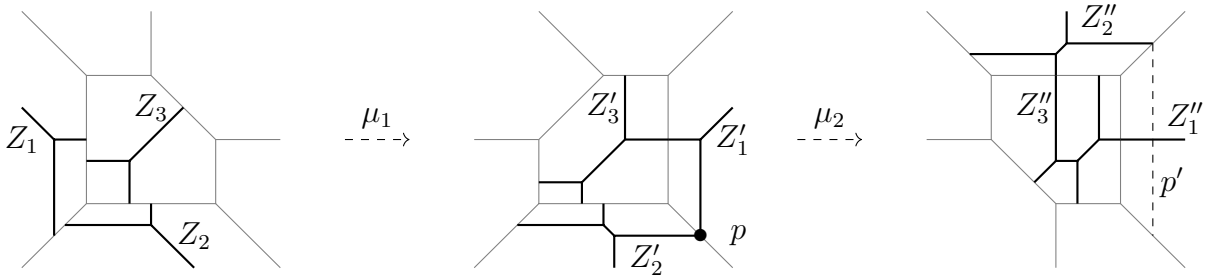
\begin{center}\begin{tikzpicture}[scale=0.85]
   \begin{scope}[xshift = -1cm]
   \draw[gray] (-1,-1) -- (0,0) -- (2,0) -- (2,1) -- (1,2) -- (0,2) -- (0,0);
   \draw[gray] (2,0) -- (3,-1) (2,1) -- (3,1) (1,2) -- (1,3) (0,2) -- (-1,3);
   
   \draw[thick] (-1/2,-1/2) -- (-1/2,1) -- (0,1) (-1/2,1) -- (-1,3/2); 
   \draw[thick] (-1/3,-1/3) -- (1,-1/3) -- (1,0) (1,-1/3) -- (5/3,-1); 
   \draw[thick] (0,2/3) -- (2/3,2/3) -- (2/3,0) (2/3,2/3) -- (3/2,3/2); 
   
   \node at (-1,1) {$Z_1$};
   \node at (5/3,-1/2) {$Z_2$};
   \node at (1,3/2) {$Z_3$};
   \end{scope}
   
   \begin{scope}[xshift = 6cm]
   \draw[gray] (-1,-1) -- (0,0) -- (2,0) -- (2,2) -- (1,2) -- (0,1) -- (0,0);
   \draw[gray] (2,0) -- (3,-1) (2,2) -- (3,3) (1,2) -- (0,3) (0,1) -- (-1,1);
   
   \draw[thick] (5/2,-1/2) -- (5/2,1) -- (2,1) (5/2,1) -- (3,3/2); 
   \draw[thick] (-1/3,-1/3) -- (1,-1/3) -- (1,0) (1,-1/3) -- (7/6,-1/2) -- (5/2,-1/2) (7/6,-1/2) -- (7/6,-1); 
   \draw[thick] (0,1/3) -- (2/3,1/3) -- (2/3,0) (2/3,1/3) -- (4/3,1) -- (2,1) (4/3,1) -- (4/3,2); 
   
   \node at (3,1) {$Z_1'$};
   \node at (5/3,-7/8) {$Z_2'$};
   \node at (1,3/2) {$Z_3'$};
   
   \node at (5/2,-1/2) [label={0:$p$}] {$\bullet$};
   \end{scope}
   
   \begin{scope}[xshift = 13cm]
   \draw[gray] (-1,1) -- (0,1) -- (1,0) -- (2,0) -- (2,2) -- (0,2) -- (0,1);
   \draw[gray] (2,0) -- (3,-1) (2,2) -- (3,3) (0,2) -- (-1,3) (1,0) -- (1,-1);
   
   \draw[thick] (2,1) -- (3,1); 
   \draw[thick] (-1/3,2+1/3) -- (1,2+1/3) -- (1,2) (1,2+1/3) -- (7/6,2+1/2) -- (5/2,2+1/2) (7/6,2+1/2) -- (7/6,3); 
   \draw[thick] (2/3,1/3) -- (1,2/3) -- (4/3,2/3) -- (4/3,0) (1,2/3) -- (1,2) (4/3,2/3) -- (5/3,1) -- (2,1) (5/3,1) -- (5/3,2); 
   
   \draw[dashed] (5/2,5/2) -- (5/2,-1/2);
   
   \node at (3,4/3) {$Z_1''$};
   \node at (5/3,2+7/8) {$Z_2''$};
   \node at (2/3,3/2) {$Z_3''$};
   \node at (2.8,1/3) {$p'$};
   
   \end{scope}
   
   \draw[dashed,->] (3,1) to node[above, midway] {$\mu_1$} (4,1);
   \draw[dashed,->] (10,1) to node[above, midway] {$\mu_2$} (11,1);
	
\end{tikzpicture}\end{center}
We now see that $Z_1''$ is contained in the hypersurface determined by the equation of $Z_3''$ and so, in exactly the same style as Example~\ref{eg!bad-mutation}, if we mutate this last seed at $Z_3''$ then $Z_1''$ will make an unexpected jump.
\end{eg}

\subsection{The tropicalisation of a log Calabi--Yau variety}

\subsubsection{The tropicalisation of $U$} \label{sec!trop}

The cocharacter lattice of a torus comes with some extra structure that we would like to generalise to $N_U(\ZZ)$. There is no group structure on $N_U(\ZZ)$ in general. Instead (as hinted above) the right way to proceed is to try and realise $N_U(\ZZ)$ as the integral points of a real affine manifold $N_U$, in the same way that the lattice $N$ can be realised as the set of integral points of the real vector space $N_\RR := N\otimes_\ZZ \RR$. This can be done, but only by introducing singularities into the affine structure on $N_U$. The space $N_U$ thus obtained is an \emph{integral affine manifold with singularities} and is referred to as the \emph{tropicalisation} of $U$.

One can build $N_U$ directly, by choosing a compactification\footnote{Despite the potentially misleading notation, in dimension $d\geq3$ the construction of $N_U$ is actually dependent on the choice of a compactification $(X,D)$ of $U$. Different choices of compactification lead to affine structures on $N_U$ which differ up to integral piecewise-linear homeomorphism, cf.\ \cite[Definition~3.8]{hk}.} $(X,D)$ of $U$ and defining $N_U$ to be the cone over the dual intersection complex of the boundary divisor $D$ \cite{hk}. However, for log Calabi--Yau varieties with a toric model $\pi\colon (X,D)\to (T,B)$ then, by generalising the 2-dimensional case covered in \cite[\S1.2]{ghk}, there is a natural way to build the space $N_U$ by altering the affine structure on the cocharacter space $N_\RR$ of the torus $\TT_S=T\setminus B$. In particular, if $U$ has a toric model then, as a real manifold, $N_U$ is homeomorphic to $N_\RR\cong \RR^d$.

For simplicity (and because it applies in the case we study later on) we assume that $(X,D)$ and $(T,B)$ are smooth, and the centres $Z_i\subset T$ blown up by $\pi$ are smooth and disjoint. Let $\mathcal F$ be the fan in $N_\RR$ that defines $(T,B)$ and consider a pair of maximal smooth cones $\sigma_1,\sigma_2$ of $\mathcal F$ that meet along a codimension 2 face $\tau = \sigma_1 \cap \sigma_2$. We can write $\sigma_1 = \langle v_1,\ldots, v_d\rangle$, $\sigma_2 = \langle v_2,\ldots,v_{d+1}\rangle$ and $\tau = \langle v_2,\ldots, v_d\rangle$ for some choice of primitive vectors $v_1,\ldots,v_{d+1}\in N$. Now the cone $\tau$ corresponds to a torus invariant curve $\overline C_\tau\cong \PP^1\subset T$ and therefore also a curve $C_\tau=\pi^{-1}(\overline C_\tau)$ in the boundary of $(X,D)$. Similarly the vectors $v_i$ correspond to the set of boundary divisors $D_i=D_{v_i}\subset D$ meeting $C_\tau$. Consider the linear map $\psi \colon \sigma_1\cup \sigma_2 \to \RR^d$ defined by 
\[ \psi(v_i) = e_i \text{ for } i=1,\ldots,d, \quad \text{and} \quad \psi(v_{d+1}) = -e_1-\sum_{i=2}^d (D_i\cdot C_\tau)e_i, \]
where $e_1,\ldots, e_d$ are the standard basis vectors of $\RR^d$. This sends $\sigma_1$ onto the positive orthant $\sigma_1'\subset \RR^d$ and $\sigma_2$ onto an appropriately chosen cone $\sigma_2'\subset \RR^d$, which has been cooked up so that the toric variety defined by the fan with the two maximal cones $\sigma_1',\sigma_2'$ contains a projective curve, $C_{\tau'}$ for $\tau'=\sigma_1'\cap \sigma_2'$, which has identical intersection numbers with boundary divisors as the curve $C_\tau\subset X$.

To define the tropicalisation $N_U$ we let $N_U^{\text{sing}}$ be the union of cones of $\mathcal F$ of codimension~$\geq2$ and $N^0_U=N_U\setminus N_U^{\text{sing}}$. To define the affine structure on $N_U^0$, for any pair of maximal cones $\sigma_1,\sigma_2$ in $\mathcal F$ as above we consider the integral affine structure on $\operatorname{int} (\sigma_1\cup \sigma_2)$ given by pulling back the integral affine structure on $\RR^d$ by the map $\psi|_{\operatorname{int} (\sigma_1\cup \sigma_2)}$. In particular, we have the following condition, in terms of the intersection theory on $X$, to tell when a piecewise linear function is actually linear along some codimension $2$ cone of $\mathcal F$ \cite[\S1.3]{mandel}. Suppose that $\phi\colon \operatorname{int} (\sigma_1\cup \sigma_2)\to \RR$ is a piecewise-linear function which is linear on each of the cones $\sigma_1, \sigma_2$. This determines a Weil divisor $\Xi_\phi = \sum_{i=1}^{d+1} \phi(v_i)D_{v_i}$ on $X$. Then $\phi$ is linear along the interior of $\tau=\sigma_1\cap\sigma_2$ if $\Xi_\phi\cdot C_\tau=0$.

The sets of the form $\operatorname{int} (\sigma_1\cup \sigma_2)$ cover $N_U^0$, and these glue together to give an affine structure on $N_U^0$. It may or may not be possible to extend this over some, or all, of the cones of $N_U^{\text{sing}}$, so at this point it is customary to redefine $N_U^0$ to be the maximal subset of $N_U$ on which this affine structure extends and then set $N_U^{\text{sing}} = N_U\setminus N_U^0$. We call $N_U^{\text{sing}}$ the \emph{singular locus} of $N_U$. We note that the subset of integral tropical points $N_U(\ZZ)\subset N_U$ is identified with the cocharacter lattice $N\cong \ZZ^d$ of $N_\RR$ by this construction. 

\subsubsection{Scattering diagrams} \label{sec!scattering}

Given a log Calabi--Yau variety $U$, the approach pioneered in \cite{ghk} to proving Conjecture~\ref{conj!GHK} in the 2-dimensional setting is to construct the ring $\CC[V]$ directly from the tropicalisation $N_U$, by equipping $N_U$ with the structure of a \emph{consistent scattering diagram}. 

We begin by working in $N_\RR$, where $N=\ZZ^d$ is a lattice with dual lattice $M=\Hom(N,\ZZ)$. Let $\mathfrak m=(z_1,\ldots,z_d)\subset \CC[\![N]\!]$ denote the maximal ideal in the ring of formal power series. A \emph{wall} in $N_\RR$ is a rational polyhedral cone $\mathfrak d_i\subset N_\RR$ of codimension 1. Roughly speaking, a \emph{scattering diagram} is then a collection $\mathfrak D=\{(\mathfrak d_i, f_i) : i \in I\}$ of walls $\mathfrak d_i$ decorated with \emph{wall functions} $f_i$. If $u_i\in M$ is a primitive vector such that $\mathfrak d_i\subset u_i^\perp$, then the wall function $f_i\in \CC[\![u_i^\perp\cap N]\!]$ is a monic power series, i.e.\ $f_i \equiv 1 \mod \mathfrak m$. The collection of walls is usually not finite and may accumulate in certain regions of $N_\RR$, or even all of $N_\RR$. Moreover the wall functions are almost always infinite power series, rather than polynomials, in which case there is a further finiteness condition specifying that only finitely many $f_i\not\equiv 1 \mod \mathfrak m^k$ for each $k\geq1$. However, the scattering diagrams constructed for the examples discussed in this paper are always finite: the set of walls form a finite complete fan $\mathcal F$ in $N_\RR$ and the wall functions are all polynomials.

Given any $(\mathfrak d, f)\in \mathfrak D$, crossing the wall $\mathfrak d$ in the direction of the normal vector $u$ specifies a \emph{wall crossing automorphism}
\[ \theta_{(\mathfrak d, f)} \colon \CC[\![N]\!] \to \CC[\![N]\!] \qquad  \theta_{(\mathfrak d, f)}(z^n) = z^n f^{-\langle n, u \rangle}.\] 
The scattering diagram is then called \emph{consistent} if, for any loop $\gamma\colon[0,1]\to N_\RR$ that begins and ends in a chamber of $\mathfrak D$ and crosses all walls of $\mathfrak D$ transversely, the composition of all the wall crossing automorphisms is the identity.\footnote{This composition has to be defining inductively, working modulo successive powers of $\mathfrak m^k$, when $\gamma$ crosses infinitely many walls of $\mathfrak D$.}

\paragraph{Log Calabi--Yau varieties with a toric model.}
Suppose that $U$ is log Calabi--Yau variety with a toric model $\pi\colon (X,D)\to (T,B)$ determined by a seed $S = \{Z_i : i = 1,\ldots,k \}$, where each $Z_i$ is a general smooth hypersurface in a component of $B$. Let $N$ be the cocharacter lattice of the torus $T\setminus B$. Under these assumptions Arg\"uz \& Gross \cite{ag} give a general inductive method to construct a consistent scattering diagram $\mathfrak D_S$ in $N_\RR$. They define an initial scattering diagram $\mathfrak D_{S,\text{in}}$ which is supported on the walls of the fan of $T$ which are affected by the nontoric blowup $\pi$. Then they give an inductive procedure which can be used to compute a consistent scattering diagram $\mathfrak D_S$ from $\mathfrak D_{S,\text{in}}$ \cite[Theorem 1.1]{ag}.

\paragraph{Broken lines.}
In order to use $\mathfrak D=\mathfrak D_S$ to construct the coordinate ring $\CC[V]$ of the mirror $V=U^\star$, it is important to fix the integral affine structure on $N_\RR$ by reconsidering $\mathfrak D$ as being a scattering diagram in $N_U$. Now the basis of theta functions $\mathcal B_V=\{ \vartheta_n \in \CC[V] : n \in N_U(\ZZ) \}$ of Conjecture~\ref{conj!GHK} are computed by counting \emph{broken lines} in $N_U$, which are tropicalisations of certain rational curves in $U$. Roughly speaking, \emph{a broken line for $n\in N_U(\ZZ)$ which starts at $q\in N_U$} is a parameterised piecewise linear path $\ell\colon[0,\infty)\to N_U$ such that 
\begin{enumerate}
\item $\ell(0)=q$, 
\item $\ell$ is allowed to bend finitely many times as it crosses walls of $\mathfrak D$ and, 
\item after it makes its last bend, $\ell$ exits $N_U$ with tangent vector $\ell'(t) = n$. 
\end{enumerate}
If we let $t_0=0$, $t_{k+1}=\infty$ and $t_1,\ldots,t_k\in\RR_{>0}$ be the times at which $\ell$ bends, and $(\mathfrak d_i,f_i)$ be the corresponding walls of $\mathfrak D$. Then we can associate a monomial $m_i\in \CC[N]$ to each domain of linearity $(t_i,t_{i+1})$ of $\ell$. This is done inductively by labelling the last domain of linearity $(t_k,\infty)$ with $m_k := z^n$, and then labelling $(t_{i-1},t_i)$ with a monomial $m_{i-1}$ appearing in the expansion of $\theta_{(\mathfrak d_i,f_i)}(m_i)$ which is dictated to us by the bend that $\ell$ makes along the $i$th wall. 

\paragraph{Theta functions.}
Fix an initial point $q\in N_U$, which is assumed to be a suitably generic (i.e.\ irrational) point of $N_U$ to ensure that $\ell$ crosses all walls of $\mathfrak D$ transversely. The theta function $\vartheta_n$ can be expanded as a Laurent power series\footnote{In a general scattering diagram this expansion is only possible as a formal Laurent power series, corresponding to counts that involve infinitely many broken lines.} $\vartheta_n = \sum_{\ell} m_\ell \in \CC[\![N]\!]$ obtained by taking the sum over all broken lines for $n$ which start at $q$, where $m_\ell := m_0$ is the monomial attached to the first domain of linearity of $\ell$ obtained by the process described above. These theta functions then form an additive basis for a ring which is expected to be the coordinate ring of the mirror variety $V=U^\star$.

\subsubsection{Generalising toric geometry}

Once we have constructed the mirror log Calabi--Yau variety $V=U^\star$, we can treat the tropicalisations $N_U$ and $M_U := N_V$ as an analogue of the cocharacter space and character space for $U$ respectively. These integral affine manifolds with singularities are equipped with a dual intersection pairing and Mandel \cite{mandel} has used this to generalise many of the traditional techniques of toric geometry to this setting, particularly in the 2-dimensional case. 

\paragraph{The intersection pairing.}
The dual intersection pairing is given by
\[ \langle{\cdot},{\cdot}\rangle \colon N_U(\ZZ)\times M_U(\ZZ)\to \ZZ \qquad \langle n,m \rangle = \nu_{D_n}(\vartheta_m)\] 
which is given by evaluating the order of vanishing of a theta function $\vartheta_m\in \CC[U]$ along a boundary component $D_n$ appearing in a compactification of $U$. This can be extended to a pairing $\langle{\cdot},{\cdot}\rangle \colon N_U\times M_U\to \RR$ by first extending to the rational points of $N_U$ and $M_U$ by bilinearity, and then extending to the real valued points of $N_U$ and $M_U$ by continuity. At least in the 2-dimensional setting, the intersection pairing obtained this way is the same the intersection pairing given by switching the roles of $U$ and $V$ \cite[Theorem~1.5]{mandel}. 

\paragraph{Convexity in $N_U$.}
Given the intersection pairing, Mandel now defines the \emph{(strong) convex hull} of a subset $S\subset N_U$ to be
\[ \conv(S) := \left\{ n\in N_U : \langle n,s \rangle \geq \inf_{s\in S} \langle s, m\rangle, \: \forall m\in M_U \right\} \]
and similarly for subsets of $M_U$. Then a \emph{(strongly) convex} subset $S\subset N_U$ is a subset for which $S=\conv(S)$. A \emph{polytope} $P\subset N_U$ is the convex hull of a finite set $S$, and we call $P$ \emph{integral} if $P=\conv(P\cap N_U(\ZZ))$. Moreover we can define the \emph{Newton polytope} $\Newt(\vartheta)\subset M_U$ for a function $\vartheta\in \CC[U]$ to be
\[ \Newt(\vartheta) = \conv\left\{ m\in M_U(\ZZ) : \vartheta=\sum_{m\in M_U} \alpha_m\vartheta_m, \: \alpha_m\neq 0 \right\}. \]

For any $c\in \RR$ and any $m\in M_U$ we let $(m)^{\geq c}$ denote the `halfspace' of $N_U$ given by
\[ (m)^{\geq c} := \{ n \in N_U : \langle n, m \rangle \geq c \} \]
which, in contrast to the classical setting, may be a bounded subset of $N_U$, or even empty.
Then the \emph{polar polytope} of a set $S\subset M_U$ is defined to be $S^\circ := \bigcap_{s\in S}  (\vartheta_s)^{\geq -1}$, and these are precisely the convex polytopes of $N_U$ that contain the origin \cite[Corollary 5.9]{mandel}. Note that $(P^\circ)^\circ)=P$. Finally, as in the toric setting, we can define an integral polytope $P\subset N_U$ to \emph{reflexive} if $P^\circ\subset M_U$ is an integral polytope in $M_U$.

\subsection{Mirror symmetry}

\subsubsection{Mirror symmetry for Fano varieties}

Let $X$ be a smooth projective $d$-dimensional Fano variety over $\CC$. As mentioned in the introduction \S\ref{sect!LR2-terms}, mirror symmetry predicts the existence of a mirror Landau--Ginzburg model $(V,w)$ to $X$. Following \cite[\S2.1]{kkp}, to make the correspondence more precise we must decorate the two sides of the mirror with some extra data
\[ \big(X, \; s_X, \; \omega_X\big) \quad \xleftrightarrow{\:\text{mirror}\:} \quad \big( (V,w), \; \Omega_V, \; \omega_V\big) \]
where $\omega_X$ and $\omega_V$ are symplectic forms, $s_X\in H^0(X,{-K_X})$ is an anticanonical section and $\Omega_V\in H^0(V,{K_V})$ is a volume form. In particular, the section $s_X$ specifies a boundary divisor $D=\divv s_X\in |{-K_X}|$. In the case that $(X,D)$ is a Fano compactification of a positive log Calabi--Yau variety $U=X\setminus D$ with maximal boundary, the mirror Landau Ginzburg model to $X$ will be defined on the mirror to $U$, i.e.\ $V=U^\star$. In terms of the mirror correspondence above, deleting the boundary divisor of $X$ corresponds to forgetting the potential on $V$.

From the point of view of homological mirror symmetry, there are now various flavours of the Fukaya category and derived category that one can associate to either side of this correspondence which are conjectured to be equivalent. However we take a slightly more low-tech point of view championed by the Fanosearch program \cite{fano-mirror}, as we now describe.

\subsubsection{Landau--Ginzburg mirrors} \label{sect!period}

There is a test one can apply to check whether a given Landau--Ginzburg model $(V,w)$ is a possible mirror to a given Fano variety $X$, which is to compute the \emph{(classical) period} of~$w$, 
\[ \pi_{w}(t) = \int_{V}\frac{\Omega_V}{1-tw} \in \CC[\![t]\!]. \] 
This calculation is particularly simple in the case that $V$ has a toric model, corresponding to the inclusion of a cluster torus chart $j\colon \TT \hookrightarrow V$. After restricting $w$ to $\TT$, we obtain a Laurent polynomial $w|_\TT\in\CC[x_1^{\pm1},\ldots,x_d^{\pm1}]$ and it follows that 
\[ \pi_w(t) = \pi_{w|_\TT}(t) = \sum_{n=0}^\infty \operatorname{const}{\left(\left(w|_\TT\right)^n\right)}t^n, \]
which can be seen by expanding $\left(1-tw|_\TT\right)^{-1}$ as a power series in $t$ and repeatedly applying Cauchy's residue theorem. Under mirror symmetry, $\pi_{w}(t)$ is expected to equal the \emph{regularised quantum period} $\widehat{G}_X(t)$ of $X$, which is a power series whose coefficients encode certain Gromov--Witten invariants for $X$ (see \cite[{\S}A]{ccgk} for details). For all of the 105 smooth 3-dimensional Fano varieties $X$, the regularised quantum period $\widehat{G}_X(t)$ has been computed by Coates, Corti, Galkin \& Kasprzyk~\cite{ccgk}.

\subsubsection{A generalisation of Batyrev--Borisov duality} \label{sec!BB-duality}

Let $U$ and $V=U^\star$ be two mirror log Calabi--Yau varieties and let $\mathcal B_U=\{ \vartheta_m : m\in M_U \}$ and $\mathcal B_V=\{ \varphi_m : m\in N_U \}$ be the bases of theta functions on $U$ and $V$. Given a pair of dual reflexive polygons $P\subset N_U$ and $Q=P^\star\subset M_U$ we can now make the following constructions, which generalise Batyrev--Borisov mirror symmetry for toric Fano varieties.\footnote{Traditionally Batyrev duality (resp.\ Batyrev--Borisov duality) refers to the mirror duality between the resolution of a general Calabi--Yau hypersurface (resp.\ complete intersection) inside two dual toric Fano varieties.} 
\begin{enumerate}
\item We define a \emph{Landau--Ginzburg model} $(V,w_P)$, where the potential $w_{P} = \sum_{p\in P} a_p\varphi_p$ is obtained by summing the theta functions on $V$ that correspond to the integral points of $P$ with a specific choice of coefficients $a_p\in\ZZ_{\geq0}$ (see Remark~\ref{rmk!binomial-edge-coeffs}). 

\item The polytope $Q\subset M_U$ determines a grading on the ring $\CC[U]$, where $\deg \vartheta_{m}$ is the least integer $k$ such that $m\in kQ$, for each $m\in M_U(\ZZ)$. We let $R_Q$ be the homogenisation of $\CC[U]$ with respect to this grading, with homogenising variable $\vartheta_0$. Let $(X_Q,D_Q)$ be the pair $X_Q=\Proj R_Q$ with boundary divisor $D_Q=\VV(\vartheta_0)$. 
\end{enumerate}

Note that $X_Q$ is an anticanonically embedded, possibly degenerate (i.e.\ non-$\QQ$-factorial) Fano variety compactifying $U=X_Q\setminus D_Q$, where the sections of $|-K_{X_Q}|$ correspond to the integral points of $Q$. The expectation is that $(X_Q,D_Q)$ and $(V,w_P)$ are mirror in the following sense.

\begin{conj}\label{conj!my-conj}
If $(X_Q,D_Q)$ admits a $\QQ$-Gorenstein deformation to a pair $(X,D)$, where $X$ is a $\QQ$-factorial Fano variety, then there exists a choice of positive integral coefficients on the lattice points of $P$ such that the Landau--Ginzburg model $(V,w_P)$ is mirror to $(X,D)$.
\end{conj}

Note that $P$ may admit more than one (or even no) such choice of coefficients if $X_P$ admits more than one (or no) deformation to a Fano variety $X$. In general $P$ is expected to support a Landau--Ginzburg potential corresponding to each deformation component of $X_Q$ (cf.\ Example~\ref{eg!last-one}).

\begin{rmk}\label{rmk!binomial-edge-coeffs}
We should explain how to choose the coefficients $a_p$ in the construction of the Landau--Ginzburg model $(V,w_P)$. In dimension $d=2$, the \emph{Minkowski ansatz} of the Fanosearch program \cite[\S6]{fano-mirror} suggests that the correct choice of coefficients on a reflexive polygon is to label the origin with coefficient $a_0=0$ and to label all of the other lattice points with \emph{binomial edge coefficients}, i.e.\ we label the $i$th lattice point on each edge of lattice length $k$ with the coefficient $\binom{k}{i}$. In higher dimensions we expect that the correct formulation will be given by a generalisation to this setting of the \emph{0-mutable polynomials}, introduced by Corti, Filip \& Petracci \cite{cfp}. In the toric case, the 0-mutable polynomials supported on a face $F\subset P$ are special labellings of $F$ with positive integral coefficients which are conjecturally in bijection with the smoothing components of the corresponding strata $D_F\subset X_Q$. The expectation is that the potential $w_P$ obtained by labelling the faces of $P$ with a compatible system of $0$-mutable polynomials will then specify the deformation of $X_Q$ described in Conjecture~\ref{conj!my-conj}.
\end{rmk}

\section{Dimension 2: the del Pezzo surface $\dP_5$} \label{sect!dP5}

As an illustration before we tackle the 3-dimensional case, we begin by briefly recalling the famous tale of the del Pezzo surface $\dP_5$. 

\subsection{The affine del Pezzo surface $U$}
Let $x_1,\ldots,x_5$ be the five terms of the 2-dimensional Lyness recurrence (LR$_2$) given in \S\ref{sect!LR2-terms}. Geometrically, $x_1,\ldots,x_5$ are regular functions on an affine surface 
\[ U = \Spec{\CC[x_1,\ldots,x_5]/\left( x_{i-1}x_{i+1} - x_i - 1 : i \in \ZZ/5\ZZ\right)}, \]
the vanishing locus of the five relations obtained from (LR$_2$). We can homogenise the equations with respect to a new variable $x_0$ to obtain the projective closure $X=\overline{U}\subset\PP^5_{x_0,x_1,\ldots,x_5}$ which is a smooth projective surface. The projection map $p_{12}\colon X\to \PP^2_{x_0,x_1,x_2}$ is birational, and by resolving the baselocus of $p_{12}^{-1}$ we find that $X$ is a blowup of $\PP^2$ in four points 
\[ e_1 = (0:1:-1), \: e_2 = (1:0:-1), \: d_3 = (1:0:0), \: d_5 = (0:1:0) \]
making $X$ a del Pezzo surface of degree~5. As is classically known, there are ten $(-1)$-curves on $X$. We call them $D_1,\ldots,D_5,E_1,\ldots, E_5$, where $E_1,E_2,D_3,D_5$ are the exceptional curves over the points $e_1,e_2,d_3,d_5$ and the remaining six curves are the strict transform of the lines passing through any two of these four points, labelled according to Figure~\ref{fig!dP5-blowup}. These ten curves have dual intersection diagram given by the Petersen graph. 
\begin{figure}[htbp]
\begin{center}
\begin{tikzpicture}[scale=1]
	\begin{scope}[scale=1.2,yshift=-0.3cm]
	\draw (-1.5,-1) -- (2,-1);
	\draw (-1,-1.5) -- (-1,2);
	\draw (2,-1.5) -- (-1.5,2);
	\node at (0.25,-1) [label={below:\small $e_2$}]{$\bullet$};
	\node at (1.5,-1) [label={below:\small $d_3$}]{$\bullet$};
	\node at (-1,1.5) [label={left:\small $d_5$}]{$\bullet$};
	\node at (-1,0.25) [label={left:\small $e_1$}]{$\bullet$};
	\node at (-1,2.3) {\small $D_1$};
	\node at (0.6,0.6) {\small $D_4$};
	\node at (2.3,-1) {\small $D_2$};
	\node at (0.6,-0.7) {\small $E_3$};
	\node at (-0.6,-0.6) {\small $E_4$};
	\node at (-0.7,0.6) {\small $E_5$};
	\draw (0.35,-1.2) -- (-1.2,1.9); 
	\draw (-1.2,0.35) -- (1.9,-1.2); 
	\draw (-1.2,0.45) -- (0.45,-1.2);
	\end{scope}
	
	\begin{scope}[xshift=8cm]
	\foreach \i in {1,...,5} {
		\node at ({2*sin(72*\i-72)},{2*cos(72*\i-72)}) [label={[label distance=-5pt]90-72*\i+72:\small $D_\i$}]{$\bullet$};
		\draw ({sin(\i*72)},{cos(\i*72)}) -- ({2*sin(\i*72)},{2*cos(\i*72)});
		\draw ({sin(\i*72)},{cos(\i*72)}) -- ({sin(\i*72+144)},{cos(\i*72+144)});
		\draw ({2*sin(\i*72)},{2*cos(\i*72)}) -- ({2*sin(\i*72+72)},{2*cos(\i*72+72)});
	}
		\node at ({sin(0)},{cos(0)}) [label={[label distance=-5pt]0:\small $E_1$}]{$\bullet$};
		\node at ({sin(72)},{cos(72)}) [label={[label distance=-5pt]-90:\small $E_2$}]{$\bullet$};
		\node at ({sin(144)},{cos(144)}) [label={[label distance=-10pt]-135:\small $E_3$}]{$\bullet$};
		\node at ({sin(216)},{cos(216)}) [label={[label distance=-5pt]-180:\small $E_4$}]{$\bullet$};
		\node at ({sin(288)},{cos(288)}) [label={[label distance=-5pt]90:\small $E_5$}]{$\bullet$};
		\end{scope}
\end{tikzpicture}
\caption{(a) A realisation of $X\subset \PP^5$ as a blowup of $\PP^2$. (b) The dual intersection graph of the ten $(-1)$-curves in $X$.}
\label{fig!dP5-blowup}
\end{center}
\end{figure}
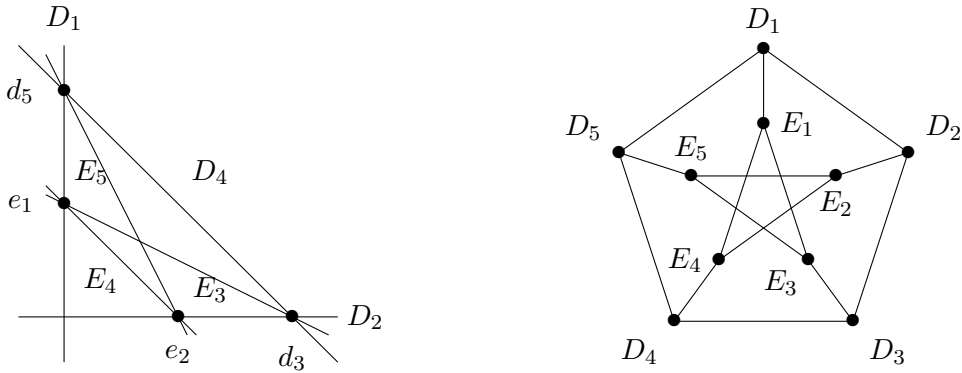

\paragraph{A log Calabi--Yau pair $(X,D)$.}
The complement to $U$ is the divisor $D := X\setminus U$, which is an anticanonical cycle $D = \bigcup_{i=1}^5D_i \in |{-K_X}|$ of five $(-1)$-curves, corresponding to the outside ring of the Petersen graph. In other words, $U$ is the interior of a log Calabi--Yau pair $(X,D)$. 

\paragraph{The interior $(-1)$-curves.}
The other five $(-1)$-curves $E = \bigcup_{i=1}^5E_i$ form a complementary anticanonical pentagram, and are called \emph{interior} $(-1)$-curves. They are given by the locus in $U$ where a cluster variable vanishes, i.e.\ $E_i|_U = \VV(x_i)\subset U$ for $i=1,\ldots,5$.

\subsection{The Grassmannian $\Gr(2,5)$}

The five equations defining the affine surface $U$ can be written as the five maximal Pfaffians of the following skewsymmetric $5\times 5$ matrix (where, for simplicity, we have omitted the diagonal of zeroes and the antisymmetry)
\[ \Pf_4\begin{pmatrix}
1 & x_1 & x_2 & 1 \\
 & 1 & x_3 & x_4 \\
 & & 1 & x_5 \\
 & & & 1
\end{pmatrix} = 0 \qquad \implies \qquad
x_{i-1}x_{i+1} = x_i + 1 \qquad \forall i\in\ZZ/5\ZZ.  \]
We can homogenise (LR$_2$) by introducing five parameters $y_1,\ldots,y_5$ in a particularly nice and symmetric way:
\begin{equation}\label{eq!hgs-LR2}\tag{$*$}
\Pf_4\begin{pmatrix}
y_5 & x_1 & x_2 & y_3 \\
 & y_2 & x_3 & x_4 \\
 & & y_4 & x_5 \\
 & & & y_1 
\end{pmatrix} = 0 \qquad \implies \qquad
x_{i-1}x_{i+1} = x_iy_i + y_{i-2}y_{i+2} \qquad \forall i\in\ZZ/5\ZZ. 
\end{equation}
After setting $x_i=-1$ for all $i$, we see that $\left(y_{2i} : i \in \ZZ/5\ZZ\right)$ is also a solution to (LR$_2$). 

The homogenised equations define a 7-dimensional variety $\sU\subset \Aa^{10}$, which is the affine cone over $\Gr(2,5)$ in the Pl\"ucker embedding. The projection $p \colon \sU \to \Aa^5_{y_1,\ldots,y_5}$ realises $\sU$ as a flat family of affine del Pezzo surfaces. For any point in the big open torus $y\in (\CC^\times)^5\subset \Aa^5$, the fibre $U_y = p^{-1}(y)$ is isomorphic to $U=X\setminus D$ after making the change of variables $z_i=\frac{x_iy_i}{y_{i+2}y_{i-2}}$ for all $i\in\ZZ/5\ZZ$. However, above the coordinate strata of $\Aa^5$ the fibres begin to degenerate. The most degenerate fibre, appearing over $0\in\Aa^5$, is the \emph{5-vertex} 
\[ p^{-1}(0) = \Aa^2_{x_1,x_2} \cup \Aa^2_{x_2,x_3} \cup \Aa^2_{x_3,x_4} \cup \Aa^2_{x_4,x_5} \cup \Aa^2_{x_5,x_1}, \]
a cycle of 5 coordinate planes glued together along their toric boundary strata. This fibration $p\colon \sU\to\Aa^5$ is (very nearly\footnote{The mirror family constructed in \cite{ghk} is fibred over the toric base scheme $\Spec\CC[\NEbar(X)]$. In terms of their construction, our parameters correspond to $y_i=z^{D_i}$, so our fibration is over $\Aa^5=\Spec\CC[z^{D_1},\ldots, z^{D_5}]$. In other words, it is obtained by restricting the coefficients to the subcone of $\NEbar(X)$ generated by the boundary components $D_1,\ldots,D_5$.}) the mirror family to $(X,D)$ constructed by Gross, Hacking \& Keel \cite{ghk}, in which the variables $x_1,\ldots,x_5$, the compactifying parameters $y_1,\ldots,y_5$ and the equations \eqref{eq!hgs-LR2} are interpreted in terms of the Gromov--Witten theory of $(X,D)$. In general, the fibres of the family $\mathcal U$ are mirror to the interior of the pair $(X,D)$, so the fact that $U$ appears in both roles here due to the fact that it is self-mirror.

\subsection{The tropicalisation of $U$}

\subsubsection{A toric model for $U$}

Consider the toric pair $(T,B)$ obtained by blowing up the two points $d_5,d_3\in\PP^2$ in Figure~\ref{fig!dP5-blowup}(a) above, so that $T$ is a smooth projective toric surface whose boundary divisor $B$ consists of a cycle of five rational curves, with self-intersection numbers $(0,0,-1,-1,-1)$. This gives a toric model $\pi\colon (X,D)\to (T,B)$ for $U$ which is a composition of two nontoric blowups of $T$ (the blowup of the image of the two points $e_1,e_2\in T$). In particular, the induced map $\pi|_D\colon D\to B$ on boundary divisors is an isomorphism and so we can label the components of $B=\bigcup_{i=1}^5B_i$ by $i\in\ZZ/5\ZZ$, where $B_i=\pi(D_i)$.

The seed for this toric model is the set of points $S=\{e_1, e_2\}$, so let us denote the corresponding inclusion of a cluster torus by $j_{12}\colon\TT_{12}=\Spec\CC[x_1^{\pm1},x_2^{\pm1}]\hookrightarrow U$. Now the Lyness map and its inverse are given by
\[ \sigma_2(x_1,x_2) = \left(x_2,\frac{1+x_2}{x_1}\right) = \left(x_2,x_3\right), \qquad \sigma_2^{-1}(x_1,x_2) = \left(\frac{1+x_1}{x_2},x_1\right) = \left(x_5,x_1\right)\] 
and these are (up to a permutation of the coordinates) the mutation at $e_1\in B_1$ and the mutation at $e_2\in B_2$ respectively. Consider the mutation at $e_1$. It is given by blowing up $e_1$ and contracting the strict transform of the curve $E_3$ to a point $e_3\in B_3$. This gives a toric model $\pi'\colon (X,D) \to (T',B')$ which is isomorphic to the original one, except that now all the labels have been shifted by one, $i\mapsto i+1\in\ZZ/5\ZZ$. In particular the dense open torus is $T'\setminus B'=\TT_{23}=\Spec\CC[x_2^{\pm1},x_3^{\pm1}]$. Thus the Lyness map can be interpreted as a mutation of toric models.
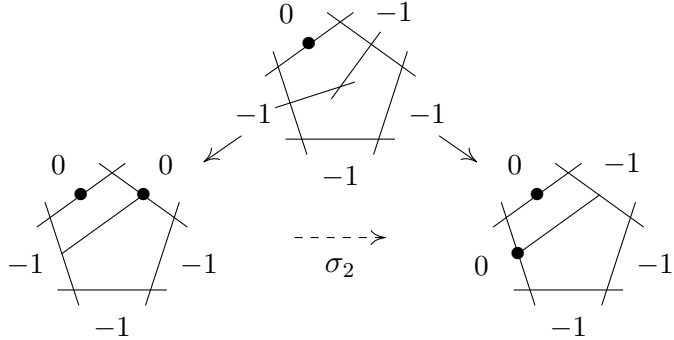
\begin{center}
  \begin{tikzpicture}[scale=1]
     \draw[-{>[length=2mm, width=2mm]}] (-1.3,-0.65) -- (-1.8,-1);
     \draw[-{>[length=2mm, width=2mm]}] (1.3,-0.65) -- (1.8,-1);
     \draw[dashed,-{>[length=2mm, width=2mm]}] (-0.6,-2) -- (0.6,-2);
     \node at (0,-2.4) {$\sigma_2$};
     
    \begin{scope}
     \foreach \i in {0,...,4} {
       \draw ({sin(72*\i-10)},{cos(72*\i-10)}) -- ({sin(72*(\i+1)+10)},{cos(72*(\i+1)+10)});
     }
     
     \node at ({0.7*sin(72*3)},{-0.7*cos(72*3)}) {$\bullet$};
     \node at ({1.2*sin(72*3)},{-1.2*cos(72*3)}) {\small $0$};
     \node at ({1.2*sin(72*4)},{-1.2*cos(72*4)}) {\small $-1$};
     \node at ({1.2*sin(72*5)},{-1.2*cos(72*5)}) {\small $-1$};
     \node at ({1.2*sin(72*1)},{-1.2*cos(72*1)}) {\small $-1$};
     \node at ({1.2*sin(72*2)},{-1.2*cos(72*2)}) {\small $-1$};
       
     \foreach \i in {4,2} {
       \draw ({0.9*sin(72*\i)},{-0.9*cos(72*\i)}) -- ({-0.2*sin(72*\i)},{0.2*cos(72*\i)});
     } 
     \end{scope}

    \begin{scope}[xshift=-3cm,yshift=-2cm]
     
     \draw ({0.7*sin(72*2)},{-0.7*cos(72*2)}) -- ({0.7*sin(72*4)},{-0.7*cos(72*4)});
     
     \node at ({1.2*sin(72*3)},{-1.2*cos(72*3)}) {\small $0$};
     \node at ({1.2*sin(72*4)},{-1.2*cos(72*4)}) {\small $-1$};
     \node at ({1.2*sin(72*5)},{-1.2*cos(72*5)}) {\small $-1$};
     \node at ({1.2*sin(72*1)},{-1.2*cos(72*1)}) {\small $-1$};
     \node at ({1.2*sin(72*2)},{-1.2*cos(72*2)}) {\small $0$};
     \foreach \i in {0,...,4} { \draw ({sin(72*\i-10)},{cos(72*\i-10)}) -- ({sin(72*(\i+1)+10)},{cos(72*(\i+1)+10)}); }
     \node at ({0.7*sin(72*3)},{-0.7*cos(72*3)}) {$\bullet$};
     \node at ({0.7*sin(72*2)},{-0.7*cos(72*2)}) {$\bullet$};
     \end{scope}
     
    \begin{scope}[xshift=3cm,yshift=-2cm]
    
     \draw ({0.7*sin(72*2)},{-0.7*cos(72*2)}) -- ({0.7*sin(72*4)},{-0.7*cos(72*4)});
     
     \node at ({1.2*sin(72*3)},{-1.2*cos(72*3)}) {\small $0$};
     \node at ({1.2*sin(72*4)},{-1.2*cos(72*4)}) {\small $0$};
     \node at ({1.2*sin(72*5)},{-1.2*cos(72*5)}) {\small $-1$};
     \node at ({1.2*sin(72*1)},{-1.2*cos(72*1)}) {\small $-1$};
     \node at ({1.2*sin(72*2)},{-1.2*cos(72*2)}) {\small $-1$};
     \foreach \i in {0,...,4} { \draw ({sin(72*\i-10)},{cos(72*\i-10)}) -- ({sin(72*(\i+1)+10)},{cos(72*(\i+1)+10)}); }
     \node at ({0.7*sin(72*3)},{-0.7*cos(72*3)}) {$\bullet$};
     \node at ({0.7*sin(72*4)},{-0.7*cos(72*4)}) {$\bullet$};
     \end{scope}
     
  \end{tikzpicture}
\end{center}
By the five-periodicity of the Lyness map, the surface $U$ contains five cluster torus charts $\TT_{i,i+1} :=\Spec\CC[x_i^{\pm1},x_{i+1}^{\pm1}]$ for $i\in \ZZ/5\ZZ$ and $U$ is given by the union of all five of them, identified according to the mutations.

\paragraph{The exchange graph of $U$.}
Each toric model for $U$ has two nontoric centres to blow up in, so the graph is 2-valent. The 5-periodicity implies that $G$ is a pentagon.

\subsubsection{The tropicalisation $N_U$}

Consider the complete fan $\mathcal F$ in $\RR^2$ determined by the five rays $\rho_i=\RR_{\geq0} v_i$ for $i\in \ZZ/5\ZZ$, where the $v_i$ are the five points
\[ v_1 = (-1,0), \quad v_2 = (0,-1), \quad v_3 = (1,0), \quad v_4 = (1,1), \quad v_5 = (0,1). \]
This is the fan of the toric pair $(T,B)$ which appears in our toric model. We can now construct $N_U$, the tropicalisation of $U$, by taking this fan $\mathcal F$ and altering the integral affine structure in $\RR^2$ as described in \S\ref{sec!trop}. This changes the affine structure along the rays $\rho_1$ and $\rho_2$, corresponding to the two components of $(T,B)$ along which a nontoric blowup occurs, and has the effect of bending lines that pass through $\rho_1$ and $\rho_2$ towards the origin. As a result it introduces a singularity at the origin $0\in N_U$. 

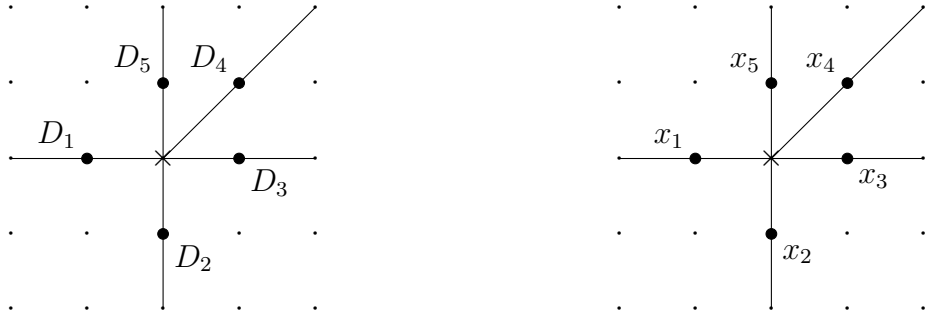
\begin{figure}[htbp]
\begin{center}
  \begin{tikzpicture}[scale=1]
  
	\foreach \x in {-2,...,2} {
	\foreach \y in {-2,...,2} {
    	\node at ({\x},{\y}) {$\cdot$};
	}}
	
    \draw (-2,0) -- (2,0) (0,-2) -- (0,2) (0,0) -- (2,2);
    \node at (0,0) {$\times$};

	\node[anchor=south east] at (-1,0) {$D_1$};
	\node[anchor=north west] at (0,-1) {$D_2$};
	\node[anchor=north west] at (1,0) {$D_3$};
	\node[anchor=south east] at (1,1) {$D_4$};
	\node[anchor=south east] at (0,1) {$D_5$}; 
	
	\draw[fill=black] (-1,0) circle (2pt);
	\draw[fill=black] (1,0) circle (2pt);
	\draw[fill=black] (0,1) circle (2pt);
	\draw[fill=black] (1,1) circle (2pt);
	\draw[fill=black] (0,-1) circle (2pt);

	\begin{scope}[xshift = 8cm]
  
	\foreach \x in {-2,...,2} {
	\foreach \y in {-2,...,2} {
    	\node at ({\x},{\y}) {$\cdot$};
	}}
	
    \draw (-2,0) -- (2,0) (0,-2) -- (0,2) (0,0) -- (2,2);
    \node at (0,0) {$\times$};

	\node[anchor=south east] at (-1,0) {$x_1$};
	\node[anchor=north west] at (0,-1) {$x_2$};
	\node[anchor=north west] at (1,0) {$x_3$};
	\node[anchor=south east] at (1,1) {$x_4$};
	\node[anchor=south east] at (0,1) {$x_5$}; 
	
	\draw[fill=black] (-1,0) circle (2pt);
	\draw[fill=black] (1,0) circle (2pt);
	\draw[fill=black] (0,1) circle (2pt);
	\draw[fill=black] (1,1) circle (2pt);
	\draw[fill=black] (0,-1) circle (2pt);
	\end{scope}
  \end{tikzpicture}
\caption{(a) The fan for $(X,D)$ in $N_U$. (b) The scattering diagram for $U$ in $M_U$.}
\label{fig!dP5-exchange-graph}
\end{center}
\end{figure}

\subsubsection{A scattering diagram for $U$}\label{sec!dp5-scattering}

The same fan $\mathcal F$ in $N_U$ also supports the structure of a scattering diagram $\mathfrak D$ which can be used to construct the coordinate ring of (the mirror to) $U$. The construction of $\mathfrak D$ and its relationship to $\CC[U]$ is described in \cite[Example 3.7]{ghk}. To obtain $\mathfrak D$, we attach the following five scattering functions $f_i\in\CC[z_1^{\pm1},z_2^{\pm1}]$ to the rays $\rho_i$ for $i\in\ZZ/5\ZZ$:
\[ f_1=1 + z_1, \quad f_2=1 + z_2, \quad f_3=1 + z_1, \quad f_4=1 + z_1z_2, \quad f_5=1 + z_2. \]
It is now straightforward to check that the collection of walls $\mathfrak D= \{ \left(\rho_i,f_i\right) : i \in \ZZ/5\ZZ\}$ defines a consistent scattering diagram, i.e.\ that starting from any point in the interior of a chamber of $\mathcal F$ and composing the five wall crossing automorphisms corresponding to a loop around $0\in N_U$ yields the identity. 

To get from the scattering diagram back to $\CC[U]$ we can consider the cluster monomials $x_i=\vartheta_{v_i}$ which are the theta functions corresponding to the points $v_i\in M_U$ for $i\in \ZZ/5\ZZ$. As described in \S\ref{sec!scattering} these can be expanded as Laurent polynomials $x_i=\sum_\ell c_\ell z^{m_\ell}$ where $\ell$ ranges over all of the broken lines for $v_i$ starting at some given point $q\in N_U$. In our case the coefficients of all the scattering functions are all equal to $1$ so all $c_\ell=1$. Suppose that $q$ is chosen near to the point $(-1+\varepsilon,1+\varepsilon)$ for some small irrational $\varepsilon>0$, which lies in the chamber $\langle v_5,v_1 \rangle$. Then the five expansions of the cluster variables are shown as in Figure~\ref{fig!dP5-broken-lines}, and these five terms generate the coordinate ring $\CC[U]$.
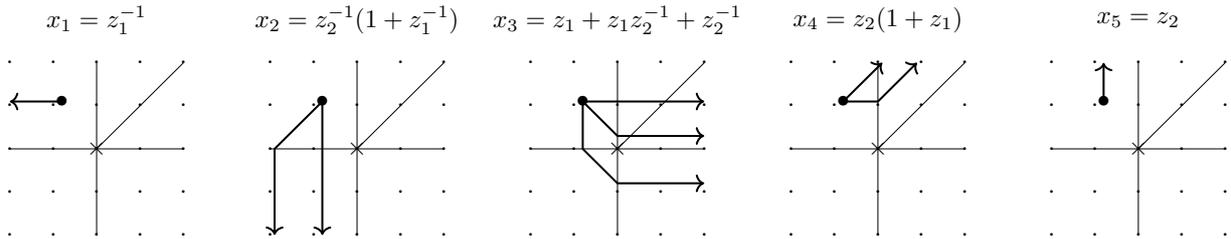
\begin{figure}[htbp]
\begin{center}
\resizebox{\textwidth}{!}{
\begin{tikzpicture}[scale = 0.6,font=\footnotesize]

   \begin{scope}[xshift = 0cm]
   \node at (0,3) {$x_1=z_1^{-1}$};
   \node at (0,0) {$\times$};
   \draw (-2,0) -- (2,0);
   \draw (0,-2) -- (0,2);
   \draw (0,0) -- (2,2);
	
	\foreach \x in {-2,...,2} {
	\foreach \y in {-2,...,2} {
    	\node at ({\x},{\y}) {$\cdot$};
	}}
	
	\node at (-0.8,1.1) {$\bullet$};
	\draw[thick,->] (-0.8,1.1) -- (-2,1.1);
	\end{scope}

   \begin{scope}[xshift = 6cm]
   \node at (0,3) {$x_2=z_2^{-1}(1+z_1^{-1})$};
   \node at (0,0) {$\times$};
   \draw (-2,0) -- (2,0);
   \draw (0,-2) -- (0,2);
   \draw (0,0) -- (2,2);
	
	\foreach \x in {-2,...,2} {
	\foreach \y in {-2,...,2} {
    	\node at ({\x},{\y}) {$\cdot$};
	}}
	
	\node at (-0.8,1.1) {$\bullet$};
	\draw[thick,->] (-0.8,1.1) -- (-0.8,-2);
	\draw[thick,->] (-0.8,1.1) -- (-1.9,0) -- (-1.9,-2);
	\end{scope}

   \begin{scope}[xshift = 12cm]
   \node at (0,3) {$x_3=z_1+z_1z_2^{-1} + z_2^{-1}$};
   \node at (0,0) {$\times$};
   \draw (-2,0) -- (2,0);
   \draw (0,-2) -- (0,2);
   \draw (0,0) -- (2,2);
	
	\foreach \x in {-2,...,2} {
	\foreach \y in {-2,...,2} {
    	\node at ({\x},{\y}) {$\cdot$};
	}}
	
	\node at (-0.8,1.1) {$\bullet$};
	\draw[thick,->] (-0.8,1.1) -- (2,1.1);
	\draw[thick,->] (-0.8,1.1) -- (0,0.3) -- (2,0.3);
	\draw[thick,->] (-0.8,1.1) -- (-0.8,0) -- (0,-0.8) -- (2,-0.8);
	\end{scope}

   \begin{scope}[xshift = 18cm]
   \node at (0,3) {$x_4=z_2(1+z_1)$};
   \node at (0,0) {$\times$};
   \draw (-2,0) -- (2,0);
   \draw (0,-2) -- (0,2);
   \draw (0,0) -- (2,2);
	
	\foreach \x in {-2,...,2} {
	\foreach \y in {-2,...,2} {
    	\node at ({\x},{\y}) {$\cdot$};
	}}
	
	\node at (-0.8,1.1) {$\bullet$};
	\draw[thick,->] (-0.8,1.1) -- (0.1,2);
	\draw[thick,->] (-0.8,1.1) -- (0,1.1) -- (0.9,2);
	\end{scope}

   \begin{scope}[xshift = 24cm]
   \node at (0,3) {$x_5=z_2$};
   \node at (0,0) {$\times$};
   \draw (-2,0) -- (2,0);
   \draw (0,-2) -- (0,2);
   \draw (0,0) -- (2,2);
	
	\foreach \x in {-2,...,2} {
	\foreach \y in {-2,...,2} {
    	\node at ({\x},{\y}) {$\cdot$};
	}}
	
	\node at (-0.8,1.1) {$\bullet$};
	\draw[thick,->] (-0.8,1.1) -- (-0.8,2);
	\end{scope}
\end{tikzpicture} }
\caption{The monomials $x_i$ obtained by counting broken lines.}
\label{fig!dP5-broken-lines}
\end{center}
\end{figure}

For any $a,b\in\ZZ_{\geq0}$ now consider the point $m=av_i + bv_{i+1}$ in the chamber $\langle v_i,v_{i+1} \rangle$. Since any broken line that leaves a chamber of $\mathfrak D$ can never return to it (cf.\ \cite[Example 3.7]{ghk}), it is easy to see that there is only one broken line for $m$ that starts at a point $q$ very near to $m$, namely the straight line that leaves $q$ in the direction of $m$. Thus if we choose to expand the theta function $\vartheta_m$ by counting broken lines starting at $q$ we obtain $\vartheta_m=\vartheta_i^a\vartheta_{i+1}^b$. This completely determines all of the theta functions associated to the points of $M_U(\ZZ)$.

\subsubsection{The intersection pairing}

Since the mirror of $U$ constructed from $\mathfrak D$ is isomorphic to $U$, it follows that the tropicalisation $N_U$ is self-dual, i.e.\ that $N_U\cong M_U$. The duality between the two affine manifolds $N_U$ and $M_U$ is given by extending the intersection pairing 
\[ \langle {\cdot},{\cdot} \rangle \colon N_U(\ZZ)\times M_U(\ZZ) \to \ZZ \qquad \langle n,m \rangle := \nu_{D_n}(\vartheta_m) \]
to an intersection pairing $\langle {\cdot},{\cdot} \rangle \colon N_U\times M_U \to \RR$ which can be calculated $\RR$-linearly in each cone of the fan $\mathcal F$. Given that, as a rational function on $X$, the divisor of $x_i$ is 
\[ \operatorname{div}x_i = E_i + D_i - D_{i-2} - D_{i+2}, \]
the pairing $N_U\times M_U\to \RR$ can be determined by Table~\ref{table!dP5}. 
\begin{table}[htp]
\caption{Table of intersection numbers $\langle D_i,x_j \rangle$ for $i,j\in\ZZ/5\ZZ$.}
\begin{center}
\renewcommand{\arraystretch}{1.2}
\begin{tabular}{|c|ccccc|} \hline
 & $D_1$ & $D_2$ & $D_3$ & $D_4$ & $D_5$ \\ \hline
$x_1$ & $ 1$ & $ 0$ & $-1$ & $-1$ & $ 0$ \\
$x_2$ & $ 0$ & $ 1$ & $ 0$ & $-1$ & $-1$ \\
$x_3$ & $-1$ & $ 0$ & $ 1$ & $ 0$ & $-1$ \\
$x_4$ & $-1$ & $-1$ & $ 0$ & $ 1$ & $ 0$ \\
$x_5$ & $ 0$ & $-1$ & $-1$ & $ 0$ & $ 1$ \\ \hline
\end{tabular}
\end{center}
\label{table!dP5}
\end{table}%

For example, we can now compute the five halfspaces $(x_i)^{\geq-1}\subset N_U$, which are shown in Figure~\ref{fig!halfspaces}. Equivalently, exactly the same diagrams show the five halfspaces $(D_i)^{\geq-1}\subset M_U$.
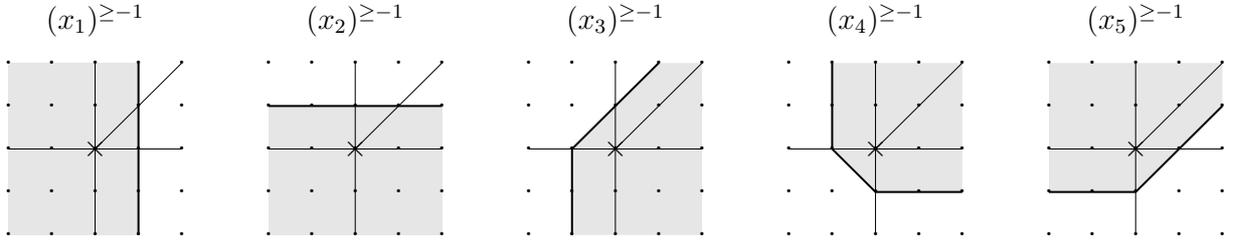
\begin{figure}[htbp]
\begin{center}
\resizebox{\textwidth}{!}{
\begin{tikzpicture}[scale = 0.6]

   \begin{scope}[xshift = 0cm]
   \fill[gray!20] (1,2) -- (1,-2) -- (-2,-2) -- (-2,2) -- cycle;
   \node at (0,3) {$(x_1)^{\geq-1}$};
   \node at (0,0) {$\times$};
   \draw (-2,0) -- (2,0);
   \draw (0,-2) -- (0,2);
   \draw (0,0) -- (2,2);
	
	\foreach \x in {-2,...,2} {
	\foreach \y in {-2,...,2} {
    	\node at ({\x},{\y}) {$\cdot$};
	}}
	
	\draw[thick] (1,-2) -- (1,2);
	\end{scope}

   \begin{scope}[xshift = 6cm]
   \fill[gray!20] (2,1) -- (-2,1) -- (-2,-2) -- (2,-2) -- cycle;
   \node at (0,3) {$(x_2)^{\geq-1}$};
   \node at (0,0) {$\times$};
   \draw (-2,0) -- (2,0);
   \draw (0,-2) -- (0,2);
   \draw (0,0) -- (2,2);
	
	\foreach \x in {-2,...,2} {
	\foreach \y in {-2,...,2} {
    	\node at ({\x},{\y}) {$\cdot$};
	}}
	
	\draw[thick] (-2,1) -- (2,1);
	\end{scope}

   \begin{scope}[xshift = 12cm]
   \fill[gray!20] (-1,-2) -- (-1,0) -- (1,2) -- (2,2) -- (2,-2) -- cycle;
   \node at (0,3) {$(x_3)^{\geq-1}$};
   \node at (0,0) {$\times$};
   \draw (-2,0) -- (2,0);
   \draw (0,-2) -- (0,2);
   \draw (0,0) -- (2,2);
	
	\foreach \x in {-2,...,2} {
	\foreach \y in {-2,...,2} {
    	\node at ({\x},{\y}) {$\cdot$};
	}}
	
	\draw[thick] (1,2) -- (-1,0) -- (-1,-2);
	\end{scope}

   \begin{scope}[xshift = 18cm]
   \fill[gray!20] (-1,2) -- (-1,0) -- (0,-1) -- (2,-1) -- (2,2) -- cycle;
   \node at (0,3) {$(x_4)^{\geq-1}$};
   \node at (0,0) {$\times$};
   \draw (-2,0) -- (2,0);
   \draw (0,-2) -- (0,2);
   \draw (0,0) -- (2,2);
	
	\foreach \x in {-2,...,2} {
	\foreach \y in {-2,...,2} {
    	\node at ({\x},{\y}) {$\cdot$};
	}}
	
	\draw[thick] (-1,2) -- (-1,0) -- (0,-1) -- (2,-1);
	\end{scope}

   \begin{scope}[xshift = 24cm]
   \fill[gray!20] (-2,-1) -- (0,-1) -- (2,1) -- (2,2) -- (-2,2) -- cycle;
   \node at (0,3) {$(x_5)^{\geq-1}$};
   \node at (0,0) {$\times$};
   \draw (-2,0) -- (2,0);
   \draw (0,-2) -- (0,2);
   \draw (0,0) -- (2,2);
	
	\foreach \x in {-2,...,2} {
	\foreach \y in {-2,...,2} {
    	\node at ({\x},{\y}) {$\cdot$};
	}}
	
	\draw[thick] (2,1) -- (0,-1) -- (-2,-1);
	\end{scope}
\end{tikzpicture} }
\caption{The halfspaces $(x_i)^{\geq-1}\subset N_U$ for $i\in \ZZ/5\ZZ$.}
\label{fig!halfspaces}
\end{center}
\end{figure}

\subsection{Applications to mirror symmetry} \label{sec!dP5-mirror-constructions}

\subsubsection{Reflexive polygons in $N_U$}

Now that we have built the spaces $N_U$ and $M_U$ as integral affine manifolds with singularitiesand their their dual intersection pairing $\langle{\cdot},{\cdot}\rangle \colon N_U\times M_U\to \RR$, we can use them to construct examples of del Pezzo pairs and their mirror Landau--Ginzburg models, as in \S\ref{sec!BB-duality}.

\begin{thm}\label{thm!dP5-refl}
Up to automorphism there are 23 reflexive polygons in $N_U$, which are displayed in Figure~\ref{fig!refl-dP5}. For each reflexive polytope $P$, the Landau--Ginzburg model $w_Q\colon U\to \CC$ obtained by labelling $Q=P^\star$ with binomial edge coefficients has the right period to be a mirror to the del Pezzo pair $(X_P,D_P)$ (where the del Pezzo surface of degree 8 represented by the unique polygon with eight vertices is $\mathbb F_1$).
\end{thm}

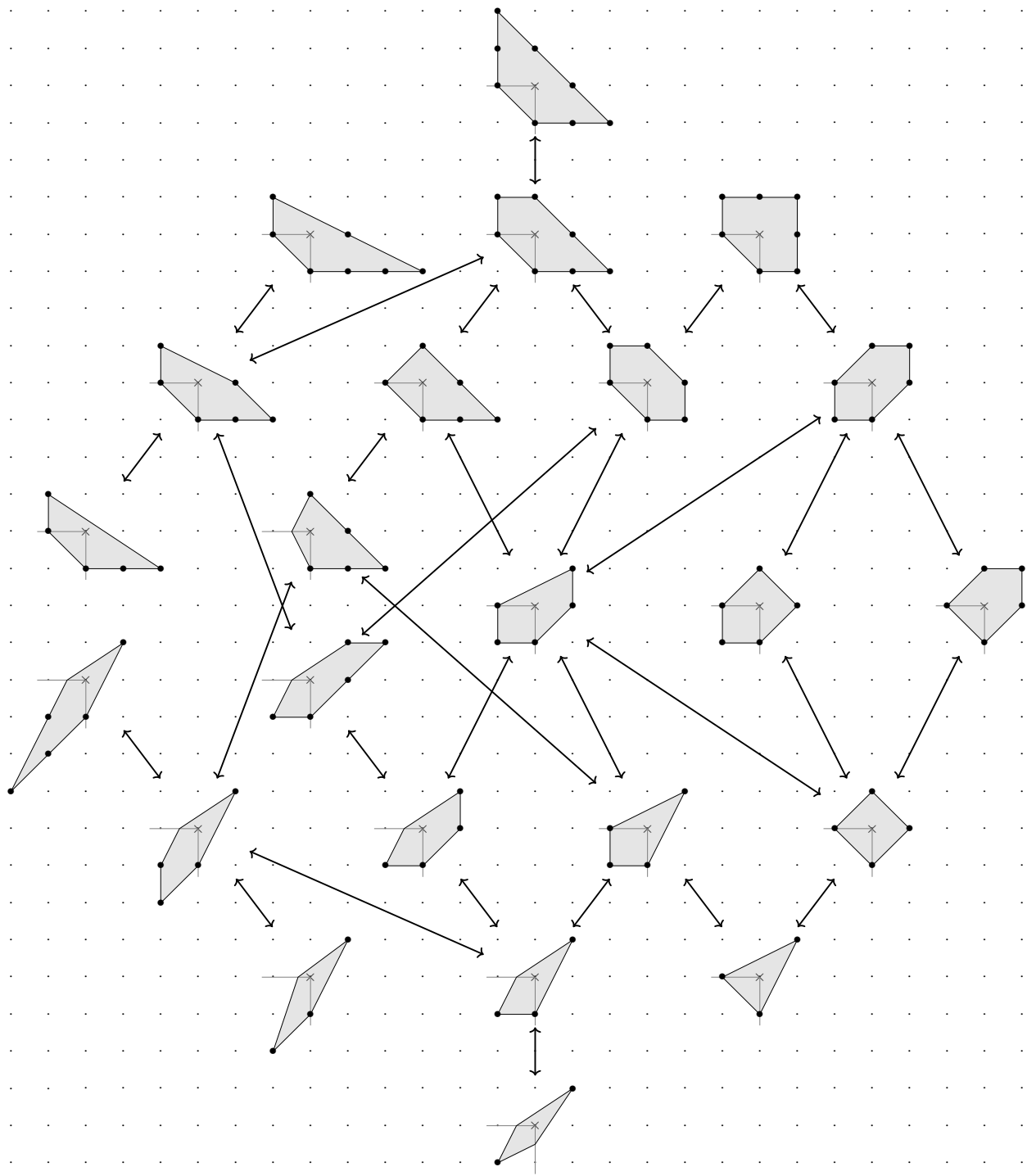
\begin{figure}[htbp]
\begin{center}
\resizebox{\textwidth}{!}{\begin{tikzpicture}

\newlength\myspace
\setlength{\myspace}{1.2cm}

   \foreach \i in {-14,...,13}
   \foreach \j in {-2,...,29}{
       \node at (\i,-\j) {$\cdot$};
   }
   \begin{scope}[xshift = -6cm, yshift = -24cm]
   \node[inner sep=\myspace] (a1) at (0,0) {$\times$};
   \draw[gray] (-1.3,0) -- (0,0) -- (0,-1.3);
   \draw[fill=gray,fill opacity=0.2] (1,1) -- (0,-1) -- (-1,-2) -- (-1/3,0) -- cycle;
   \node at (1,1) {$\bullet$};
   \node at (0,-1) {$\bullet$};
   \node at (-1,-2) {$\bullet$};
   \end{scope}
   
   \begin{scope}[xshift = -12cm, yshift = -16cm]
   \node[inner sep=\myspace] (a2) at (0,0) {$\times$};
   \draw[gray] (-1.3,0) -- (0,0) -- (0,-1.3);
   \draw[fill=gray,fill opacity=0.2] (1,1) -- (-1/2,0) -- (-2,-3) -- (0,-1) -- cycle;
   \node at (1,1) {$\bullet$};
   \node at (0,-1) {$\bullet$};
   \node at (-1,-2) {$\bullet$};
   \node at (-2,-3) {$\bullet$};
   \node at (-1,-1) {$\bullet$};
   \end{scope}
   
   \begin{scope}[xshift = -12cm, yshift = -12cm]
   \node[inner sep=\myspace] (a3) at (0,0) {$\times$};
   \draw[gray] (-1.3,0) -- (0,0) -- (0,-1.3);
   \draw[fill=gray,fill opacity=0.2] (2,-1) -- (-1,1) -- (-1,0) -- (0,-1) -- cycle;
   \node at (2,-1) {$\bullet$};
   \node at (1,-1) {$\bullet$};
   \node at (0,-1) {$\bullet$};
   \node at (-1,0) {$\bullet$};
   \node at (-1,1) {$\bullet$};
   \end{scope}
   
   \begin{scope}[xshift = -6cm, yshift = -4cm]
   \node[inner sep=\myspace] (a4) at (0,0) {$\times$};
   \draw[gray] (-1.3,0) -- (0,0) -- (0,-1.3);
   \draw[fill=gray,fill opacity=0.2] (-1,1) -- (3,-1) -- (0,-1) -- (-1,0) -- cycle;
   \node at (-1,1) {$\bullet$};
   \node at (-1,0) {$\bullet$};
   \node at (0,-1) {$\bullet$};
   \node at (1,-1) {$\bullet$};
   \node at (2,-1) {$\bullet$};
   \node at (3,-1) {$\bullet$};
   \node at (1,0) {$\bullet$};
   \end{scope}
   
   \begin{scope}[xshift = 0cm, yshift = -28cm]
   \node[inner sep=\myspace] (a5) at (0,0) {$\times$};
   \draw[gray] (-1.3,0) -- (0,0) -- (0,-1.3);
   \draw[fill=gray,fill opacity=0.2] (-1,-1) -- (-1/2,0) -- (1,1) -- (0,-1/2) -- cycle;
   \node at (1,1) {$\bullet$};
   \node at (-1,-1) {$\bullet$};
   \end{scope}
   
   \begin{scope}[xshift = -9cm, yshift = -20cm]
   \node[inner sep=\myspace] (a6) at (0,0) {$\times$};
   \draw[gray] (-1.3,0) -- (0,0) -- (0,-1.3);
   \draw[fill=gray,fill opacity=0.2] (-1,-2) -- (-1,-1) -- (-1/2,0) -- (1,1) -- (0,-1) -- cycle;
   \node at (1,1) {$\bullet$};
   \node at (-1,-1) {$\bullet$};
   \node at (0,-1) {$\bullet$};
   \node at (-1,-2) {$\bullet$};
   \end{scope}
   
   \begin{scope}[xshift = -9cm, yshift = -8cm]
   \node[inner sep=\myspace] (a7) at (0,0) {$\times$};
   \draw[gray] (-1.3,0) -- (0,0) -- (0,-1.3);
   \draw[fill=gray,fill opacity=0.2] (-1,1) -- (1,0) -- (2,-1) -- (0,-1) -- (-1,0) -- cycle;
   \node at (-1,1) {$\bullet$};
   \node at (1,0) {$\bullet$};
   \node at (2,-1) {$\bullet$};
   \node at (1,-1) {$\bullet$};
   \node at (0,-1) {$\bullet$};
   \node at (-1,0) {$\bullet$};
   \end{scope}
   
   \begin{scope}[xshift = 0cm, yshift = -0cm]
   \node[inner sep=\myspace] (a8) at (0,0) {$\times$};
   \draw[gray] (-1.3,0) -- (0,0) -- (0,-1.3);
   \draw[fill=gray,fill opacity=0.2] (-1,2) -- (-1,0) -- (0,-1) -- (2,-1) -- cycle;
   \node at (-1,2) {$\bullet$};
   \node at (-1,1) {$\bullet$};
   \node at (-1,0) {$\bullet$};
   \node at (0,-1) {$\bullet$};
   \node at (1,-1) {$\bullet$};
   \node at (2,-1) {$\bullet$};
   \node at (1,0) {$\bullet$};
   \node at (0,1) {$\bullet$};
   \end{scope}
   
   \begin{scope}[xshift = 0cm, yshift = -24cm]
   \node[inner sep=\myspace] (a9) at (0,0) {$\times$};
   \draw[gray] (-1.3,0) -- (0,0) -- (0,-1.3);
   \draw[fill=gray,fill opacity=0.2] (-1,-1) -- (0,-1) -- (1,1) -- (-1/2,0) -- cycle;
   \node at (-1,-1) {$\bullet$};
   \node at (0,-1) {$\bullet$};
   \node at (1,1) {$\bullet$};
   \end{scope}
   
   \begin{scope}[xshift = -3cm, yshift = -20cm]
   \node[inner sep=\myspace] (a10) at (0,0) {$\times$};
   \draw[gray] (-1.3,0) -- (0,0) -- (0,-1.3);
   \draw[fill=gray,fill opacity=0.2] (-1,-1) -- (0,-1) -- (1,0) -- (1,1) -- (-1/2,0) -- cycle;
   \node at (-1,-1) {$\bullet$};
   \node at (1,0) {$\bullet$};
   \node at (0,-1) {$\bullet$};
   \node at (1,1) {$\bullet$};
   \end{scope}
   
   \begin{scope}[xshift = -6cm, yshift = -16cm]
   \node[inner sep=\myspace] (a11) at (0,0) {$\times$};
   \draw[gray] (-1.3,0) -- (0,0) -- (0,-1.3);
   \draw[fill=gray,fill opacity=0.2] (-1,-1) -- (0,-1) -- (2,1) -- (1,1) -- (-1/2,0) -- cycle;
   \node at (-1,-1) {$\bullet$};
   \node at (1,0) {$\bullet$};
   \node at (0,-1) {$\bullet$};
   \node at (1,1) {$\bullet$};
   \node at (2,1) {$\bullet$};
   \end{scope}
   
   \begin{scope}[xshift = -6cm, yshift = -12cm]
   \node[inner sep=\myspace] (a12) at (0,0) {$\times$};
   \draw[gray] (-1.3,0) -- (0,0) -- (0,-1.3);
   \draw[fill=gray,fill opacity=0.2] (0,-1) -- (-1/2,0) -- (0,1) -- (2,-1) -- cycle;
   \node at (0,-1) {$\bullet$};
   \node at (0,1) {$\bullet$};
   \node at (1,0) {$\bullet$};
   \node at (1,-1) {$\bullet$};
   \node at (2,-1) {$\bullet$};
   \end{scope}
   
   \begin{scope}[xshift = -3cm, yshift = -8cm]
   \node[inner sep=\myspace] (a13) at (0,0) {$\times$};
   \draw[gray] (-1.3,0) -- (0,0) -- (0,-1.3);
   \draw[fill=gray,fill opacity=0.2] (0,-1) -- (-1,0) -- (0,1) -- (2,-1) -- cycle;
   \node at (-1,0) {$\bullet$};
   \node at (0,-1) {$\bullet$};
   \node at (0,1) {$\bullet$};
   \node at (1,0) {$\bullet$};
   \node at (1,-1) {$\bullet$};
   \node at (2,-1) {$\bullet$};
   \end{scope}
   
   \begin{scope}[xshift = 0cm, yshift = -4cm]
   \node[inner sep=\myspace] (a14) at (0,0) {$\times$};
   \draw[gray] (-1.3,0) -- (0,0) -- (0,-1.3);
   \draw[fill=gray,fill opacity=0.2] (-1,1) -- (-1,0) -- (0,-1) -- (2,-1) -- (0,1) -- cycle;
   \node at (-1,0) {$\bullet$};
   \node at (-1,1) {$\bullet$};
   \node at (0,-1) {$\bullet$};
   \node at (0,1) {$\bullet$};
   \node at (1,0) {$\bullet$};
   \node at (1,-1) {$\bullet$};
   \node at (2,-1) {$\bullet$};
   \end{scope}
   
   \begin{scope}[xshift = 3cm, yshift = -20cm]
   \node[inner sep=\myspace] (a15) at (0,0) {$\times$};
   \draw[gray] (-1.3,0) -- (0,0) -- (0,-1.3);
   \draw[fill=gray,fill opacity=0.2] (-1,-1) -- (-1,0) -- (1,1) -- (0,-1) -- cycle;
   \node at (-1,0) {$\bullet$};
   \node at (-1,-1) {$\bullet$};
   \node at (0,-1) {$\bullet$};
   \node at (1,1) {$\bullet$};
   \end{scope}
   
   \begin{scope}[xshift = 0cm, yshift = -14cm]
   \node[inner sep=\myspace] (a16) at (0,0) {$\times$};
   \draw[gray] (-1.3,0) -- (0,0) -- (0,-1.3);
   \draw[fill=gray,fill opacity=0.2] (-1,-1) -- (-1,0) -- (1,1) -- (1,0) -- (0,-1) -- cycle;
   \node at (-1,0) {$\bullet$};
   \node at (-1,-1) {$\bullet$};
   \node at (1,0) {$\bullet$};
   \node at (0,-1) {$\bullet$};
   \node at (1,1) {$\bullet$};
   \end{scope}
   
   \begin{scope}[xshift = 3cm, yshift = -8cm]
   \node[inner sep=\myspace] (a17) at (0,0) {$\times$};
   \draw[gray] (-1.3,0) -- (0,0) -- (0,-1.3);
   \draw[fill=gray,fill opacity=0.2] (-1,0) -- (-1,1) -- (0,1) -- (1,0) -- (1,-1) -- (0,-1) -- cycle;
   \node at (-1,0) {$\bullet$};
   \node at (-1,1) {$\bullet$};
   \node at (1,0) {$\bullet$};
   \node at (0,1) {$\bullet$};
   \node at (0,-1) {$\bullet$};
   \node at (1,-1) {$\bullet$};
   \end{scope}
   
   \begin{scope}[xshift = 6cm, yshift = -24cm]
   \node[inner sep=\myspace] (a18) at (0,0) {$\times$};
   \draw[gray] (-1.3,0) -- (0,0) -- (0,-1.3);
   \draw[fill=gray,fill opacity=0.2] (-1,0) -- (0,-1) -- (1,1) -- cycle;
   \node at (-1,0) {$\bullet$};
   \node at (0,-1) {$\bullet$};
   \node at (1,1) {$\bullet$};
   \end{scope}
   
   \begin{scope}[xshift = 9cm, yshift = -20cm]
   \node[inner sep=\myspace] (a19) at (0,0) {$\times$};
   \draw[gray] (-1.3,0) -- (0,0) -- (0,-1.3);
   \draw[fill=gray,fill opacity=0.2] (-1,0) -- (0,-1) -- (1,0) -- (0,1) -- cycle;
   \node at (-1,0) {$\bullet$};
   \node at (0,-1) {$\bullet$};
   \node at (0,1) {$\bullet$};
   \node at (1,0) {$\bullet$};
   \end{scope}
   
   \begin{scope}[xshift = 6cm, yshift = -14cm]
   \node[inner sep=\myspace] (a20) at (0,0) {$\times$};
   \draw[gray] (-1.3,0) -- (0,0) -- (0,-1.3);
   \draw[fill=gray,fill opacity=0.2] (-1,0) -- (-1,-1) -- (0,-1) -- (1,0) -- (0,1) -- cycle;
   \node at (-1,0) {$\bullet$};
   \node at (0,-1) {$\bullet$};
   \node at (0,1) {$\bullet$};
   \node at (1,0) {$\bullet$};
   \node at (-1,-1) {$\bullet$};
   \end{scope}
   
   \begin{scope}[xshift = 9cm, yshift = -8cm]
   \node[inner sep=\myspace] (a21) at (0,0) {$\times$};
   \draw[gray] (-1.3,0) -- (0,0) -- (0,-1.3);
   \draw[fill=gray,fill opacity=0.2] (-1,0) -- (-1,-1) -- (0,-1) -- (1,0) -- (1,1) -- (0,1) -- cycle;
   \node at (-1,0) {$\bullet$};
   \node at (0,-1) {$\bullet$};
   \node at (0,1) {$\bullet$};
   \node at (1,0) {$\bullet$};
   \node at (-1,-1) {$\bullet$};
   \node at (1,1) {$\bullet$};
   \end{scope}
   
   \begin{scope}[xshift = 6cm, yshift = -4cm]
   \node[inner sep=\myspace] (a22) at (0,0) {$\times$};
   \draw[gray] (-1.3,0) -- (0,0) -- (0,-1.3);
   \draw[fill=gray,fill opacity=0.2] (-1,0) -- (0,-1) -- (1,-1) -- (1,1) -- (-1,1) -- cycle;
   \node at (-1,0) {$\bullet$};
   \node at (0,-1) {$\bullet$};
   \node at (0,1) {$\bullet$};
   \node at (1,0) {$\bullet$};
   \node at (1,1) {$\bullet$};
   \node at (1,-1) {$\bullet$};
   \node at (-1,1) {$\bullet$};
   \end{scope}
   
   \begin{scope}[xshift = 12cm, yshift = -14cm]
   \node[inner sep=\myspace] (a23) at (0,0) {$\times$};
   \draw[gray] (-1.3,0) -- (0,0) -- (0,-1.3);
   \draw[fill=gray,fill opacity=0.2] (-1,0) -- (0,-1) -- (1,0) -- (1,1) -- (0,1) -- cycle;
   \node at (-1,0) {$\bullet$};
   \node at (0,-1) {$\bullet$};
   \node at (0,1) {$\bullet$};
   \node at (1,0) {$\bullet$};
   \node at (1,1) {$\bullet$};
   \end{scope}
   
   \draw[very thick, <->] (a19) -- (a23);
   \draw[very thick, <->] (a21) -- (a23);
   \draw[very thick, <->] (a18) -- (a19);
   \draw[very thick, <->] (a19) -- (a20);
   \draw[very thick, <->] (a20) -- (a21);
   \draw[very thick, <->] (a21) -- (a22);
   \draw[very thick, <->] (a18) -- (a15);
   \draw[very thick, <->] (a19) -- (a16);
   \draw[very thick, <->] (a21) -- (a16);
   \draw[very thick, <->] (a22) -- (a17);
   \draw[very thick, <->] (a15) -- (a16);
   \draw[very thick, <->] (a16) -- (a17);
   \draw[very thick, <->] (a15) -- (a12);
   \draw[very thick, <->] (a16) -- (a10);
   \draw[very thick, <->] (a16) -- (a13);
   \draw[very thick, <->] (a17) -- (a11);   
   \draw[very thick, <->] (a9) -- (a10);
   \draw[very thick, <->] (a10) -- (a11);
   \draw[very thick, <->] (a12) -- (a13);
   \draw[very thick, <->] (a13) -- (a14);   
   \draw[very thick, <->] (a9) -- (a5);
   \draw[very thick, <->] (a9) -- (a6);
   \draw[very thick, <->] (a11) -- (a7);
   \draw[very thick, <->] (a12) -- (a6);
   \draw[very thick, <->] (a14) -- (a7);
   \draw[very thick, <->] (a14) -- (a8);
   \draw[very thick, <->] (a1) -- (a6);
   \draw[very thick, <->] (a2) -- (a6);
   \draw[very thick, <->] (a3) -- (a7);
   \draw[very thick, <->] (a4) -- (a7);
   \draw[very thick, <->] (a9) -- (a15);
   \draw[very thick, <->] (a14) -- (a17);
\end{tikzpicture}}
\caption{Representatives for the 23 classes of reflexive polygons in $N_U$. The gray lines indicate the two rays along which we bend the affine structure of $N_U$ and the arrows between polygons denote the addition or removal of a vertex. Since $N_U$ is isomorphic to its dual space $M_U$, given a polygon $P\subset N_U$ we can identify the dual polygon $P^\star\subset M_U$ with a polygon in $N_U$. Duality between these reflexive polygons is then given by top-to-bottom reflection in the diagram. In particular there are three self-dual pentagons in the central row.}
\label{fig!refl-dP5}
\end{center}
\end{figure}

\begin{proof}
The classification of reflexive polygons in $N_U$ proceeds in an equivalent manner to the classification of reflexive polygons in the ordinary toric setting. Beginning from any one such polygon, such as the polygon $P$ of Example~\ref{eg!dP5-pentagon}, we can add or remove vertices corresponding to blowing up or blowing down $(-1)$-curves in the boundary of the associated Looijenga pair $(X_P,D_P)$. Disregarding polygons with nonzero interior lattice points produces the 23 examples of Figure~\ref{fig!refl-dP5}. 

To verify that the potential $w_Q$ has the right period to be mirror to $(X_P,D_P)$, we can restrict to the cluster torus $\TT_{12}\subset U$ to get a Laurent polynomial $w_0 := w_Q|_{\TT_{12}}\in\CC[x_1^{\pm1},x_2^{\pm1}]$. The only difference between the Landau--Ginzburg model $(U,w_Q)$ and the toric Landau--Ginzburg model $(\TT_{12},w_0)$ is that we have extended the domain of definition of $w_0$ along the two exceptional curves of $U\setminus \TT_{12}$. It is now just a case of verifying that $w_0$ is mutation equivalent to a known Laurent polynomial mirror to $X_Q$ (see \cite[Figure~1]{many-authors}).
\end{proof}

If $d(P)$ denotes the number of boundary points of a reflexive polygon $P\subset N_U$, then for a dual pair of reflexive polygons we have $d(P)+d(P^\star)=10$. This is in contrast to the usual formula $d(P)+d(P^\star)=12$ that holds in the ordinary toric setting, and corresponds to the fact that we have made two nontoric blowups in the boundary.

\subsubsection{Examples}

We illustrate Theorem~\ref{thm!dP5-refl} with a few examples.

\begin{eg} \label{eg!dP5-pentagon}
Suppose we take the standard projective compactification $(X,D)$ of $U$, which corresponds to the pair $(X_P,D_P)$ where $P=P^\star$ is the self-dual polygon
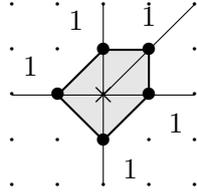
\begin{center}
\begin{tikzpicture}[scale = 0.6]

   \begin{scope}[xshift = 0cm]
   \fill[gray!20] (1,1) -- (1,0) -- (0,-1) -- (-1,0) --(0,1) -- cycle;
   \node at (0,0) {$\times$};
   \draw (-2,0) -- (2,0);
   \draw (0,-2) -- (0,2);
   \draw (0,0) -- (2,2);
	
	\foreach \x in {-2,...,2} {
	\foreach \y in {-2,...,2} {
    	\node at ({\x},{\y}) {$\cdot$};
	}}
	
	\draw[thick] (1,1) -- (1,0) -- (0,-1) -- (-1,0) --(0,1) -- cycle;
	\node at (-1,0) [label={[label distance=-5pt]135:\small $1$}] {$\bullet$};
	\node at (0,-1) [label={[label distance=-5pt]-45:\small $1$}] {$\bullet$};
	\node at (1,0) [label={[label distance=-5pt]-45:\small $1$}] {$\bullet$};
	\node at (0,1) [label={[label distance=-5pt]135:\small $1$}] {$\bullet$};
	\node at (1,1) [label={[label distance=-2pt]90:\small $1$}] {$\bullet$};
	\end{scope}

\end{tikzpicture} 
\end{center}
which, according to Remark~\ref{rmk!binomial-edge-coeffs}, has been labelled with binomial edge coefficients. Note that $P$ is cut out by the five halfspaces of Figure~\ref{fig!halfspaces}. Let $w = w_P$ be the potential on the mirror Landau--Ginzburg model $w\colon U\to \CC$, which is given by
\[ w = x_1 + x_2 + x_3 + x_4 + x_5. \]
This is a $\sigma_2$-invariant function, and we note that this has a nice alternative representation as $w+3 = x_1x_2x_3x_4x_5$, corresponding to the fact that the fibre $w^{-1}(-3)\subset U$ breaks up as the union of the five interior $(-1)$-curves $E = \bigcup_{i=1}^5 E_i$.

\paragraph{The period $\pi_{w}(t)$.}
By restricting $w$ to the cluster torus chart $\TT_{12}$ (i.e.\ expanding it as a Laurent polynomial in terms of $x_1,x_2$), we can compute the period $\pi_{w}(t)$ which we see to be equal to the regularised quantum period for $\dP_5$. The first few terms are given by
\[ \pi_{w}(t) = 1 + 10t^2 + 30t^3 + 270t^4 + 1560t^5 + 11350t^6 + 77700t^7 + \cdots \]
Alternatively, we can compute the period of the shifted potential $w+3$, which is given by the change of variables $\pi_{w+3}(t) = \frac{1}{1-3t}\pi_{w}\left(\frac{t}{1-3t}\right)$. The series is
\[ \pi_{w+3}(t) = 1 + 3t + 19t^2 + 147t^3 + 1251t^4 + 11253t^5 + 104959t^6 + \cdots \]
which can be recognised as one of the famous Ap\'ery series $\pi_{w+3}(t) = \sum_{n=0}^\infty\sum_{k=0}^n \binom{n}{k}^2 \binom{n+k}{k}t^n$ and is well-known as a period for the del Pezzo surface of degree 5. 

\paragraph{The elliptic fibration.}
After extending to a birational map $w\colon X\dashrightarrow \PP^1$, the fibres of $w$ belong to the anticanonical pencil $|D,E|\subset |{-K_X}|$ with baselocus given by the five points $D_i\cap E_i\in X$. Blowing up these five points $\phi\colon \widetilde{X}\to X$ resolves $w$ into an elliptic fibration $\widetilde{w}=w\circ\phi\colon \widetilde{X}\to \PP^1$, which appears in Beauville's classification of rational elliptic surfaces with four singular fibres \cite{Beauville}. There are two $I_5$ fibres over $\infty$ and $-3$, corresponding to $D$ and $E$ respectively, and two further $I_1$ fibres over the values $-\frac{5}{\varphi}$ and $5\varphi$, where $\varphi=\frac{1+\sqrt{5}}2$ is the golden ratio.
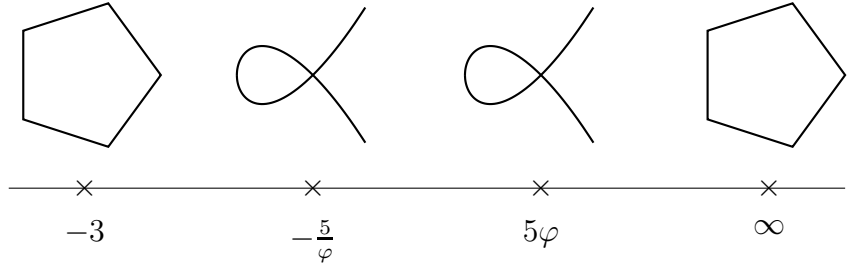
\begin{center}\begin{tikzpicture}
   
   \draw[thick] ({cos(0)},{sin(0)}) -- ({cos(72)},{sin(72)}) -- ({cos(144)},{sin(144)}) -- ({cos(216)},{sin(216)}) -- ({cos(288)},{sin(288)}) -- cycle;
   
   \draw[thick,scale=1,domain=0:1.3,smooth,variable=\x] plot ({3+\x^2-1},{\x^3-\x});
   \draw[thick,scale=1,domain=-1.3:0,smooth,variable=\x] plot ({3-1-\x^2},{\x^3-\x});
   
   \draw[thick,scale=1,domain=0:1.3,smooth,variable=\x] plot ({6+\x^2-1},{\x^3-\x});
   \draw[thick,scale=1,domain=-1.3:0,smooth,variable=\x] plot ({6-1-\x^2},{\x^3-\x});
   
   \draw[thick] ({9+cos(0)},{sin(0)}) -- ({9+cos(72)},{sin(72)}) -- ({9+cos(144)},{sin(144)}) -- ({9+cos(216)},{sin(216)}) -- ({9+cos(288)},{sin(288)}) -- cycle;
   
   \draw (-1,-1.5) -- (10,-1.5);
   \node at (3,-1.5) [label={below:$-\frac{5}\varphi$}]{$\times$};
   \node at (6,-1.5) [label={below:$5\varphi$}]{$\times$};
   \node at (0,-1.5) [label={below:$-3$}] {$\times$};
   \node at (9,-1.5) [label={below:$\infty$}] {$\times$};
\end{tikzpicture}\end{center}

The Landau--Ginzburg model $w\colon U\to \CC$ is obtained from this elliptic fibration $\widetilde w$ by deleting the fibre at $\infty$ and five horizontal sections given by the five $\phi$-exceptional curves. In terms of mirror symmetry, the significance of the five deleted sections is due to the fact that the anticanonical divisor $D\subset X$ has five corners \cite[Remark 2.1]{kkp}. Changing the choice of boundary divisor by smoothing a corner of $D$ corresponds to extending $w$ along one of the missing horizontal sections, so the Landau--Ginzburg model obtained by extending $w$ to $\widetilde{X}\setminus \widetilde{w}^{-1}(\infty)$ is mirror to a del Pezzo surface of degree 5 with a smooth anticanonical divisor, as expected \cite{ako}.

The fibration $\widetilde{w}$ has some very interesting arithmetic properties. The five $\phi$-exceptional curves of $\widetilde w$ are permuted by the $\ZZ/5\ZZ$-symmetry. For each smooth fibre ${\widetilde X}_t$ these sections give rise to the orbit of a rational 5-torsion point under translation. Moreover, the monodromy action on $H^1({\widetilde X}_t,\ZZ)\cong \ZZ^2$ around the four singular fibres generates the congruence subgroup $\Gamma_1(5)\subset \SL(2,\ZZ)$, as shown in \cite{beukers}. 
\end{eg}

\begin{eg} \label{eg!dP2}
To give a slightly more exotic example, consider the following polygon $P\subset N_U$ and its dual $Q=P^\star \subset M_U$, which have also been labelled with binomial edge coefficients.
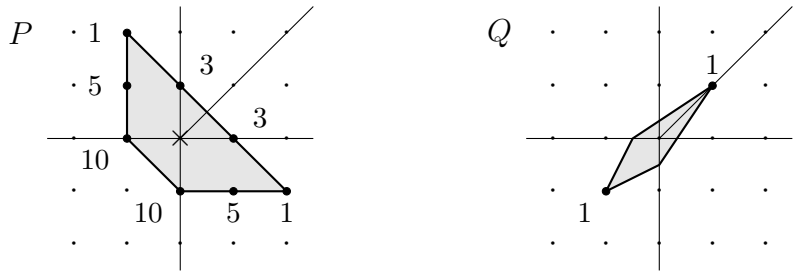
\begin{center} 
\begin{tikzpicture}[scale = 0.35]

   \begin{scope}[xshift = 18cm]
   \fill[gray!20] (-2,-2) -- (-1,0) -- (2,2) -- (0,-1) -- cycle; 
   \draw (0,-5) -- (0,5) (5,0) -- (-5,0) (0,0) -- (5,5);
	
	\node at (-6,4) {$Q$};
	\foreach \x in {-2,...,2} {
	\foreach \y in {-2,...,2} {
    	\node at ({2*\x},{2*\y}) {$\cdot$};
	}}
	
    \draw[thick] (-2,-2) -- (-1,0) -- (2,2) -- (0,-1) -- cycle; 
	
	\draw[fill=black] (-2,-2) circle (4pt);
	\draw[fill=black] (2,2) circle (4pt);
	\node at (-2.8,-2.8) {\small $1$};
	\node at (2,2.8) {\small $1$};
	\end{scope}

	\node at (-6,4) {$P$};
	
    \fill[gray!20] (-2,4) -- (4,-2) -- (0,-2) -- (-2,0) -- cycle; 
   \draw (0,-5) -- (0,5) (5,0) -- (-5,0) (0,0) -- (5,5);
    
	\foreach \x in {-2,...,2} {
	\foreach \y in {-2,...,2} {
    	\node at ({2*\x},{2*\y}) {$\cdot$};
	}}
	
    \draw[thick] (-2,4) -- (4,-2) -- (0,-2) -- (-2,0) -- cycle; 
	
	\draw[fill=black] (-2,4) circle (4pt);
	\draw[fill=black] (-2,2) circle (4pt);
	\draw[fill=black] (-2,0) circle (4pt);
	\draw[fill=black] (0,2) circle (4pt);
	\draw[fill=black] (0,-2) circle (4pt);
	\draw[fill=black] (2,0) circle (4pt);
	\draw[fill=black] (2,-2) circle (4pt);
	\draw[fill=black] (4,-2) circle (4pt);
	
	\node at (0,0) {$\times$};
	
	\node at (-3.2,-0.8) {\small $10$};
	\node at (-1.2,-2.8) {\small $10$};
	\node at (2,-2.8) {\small $5$};
	\node at (4,-2.8) {\small $1$};
	\node at (3,0.8) {\small $3$};
	\node at (1,2.8) {\small $3$};
	\node at (-3.2,4) {\small $1$};
	\node at (-3.2,2) {\small $5$};
\end{tikzpicture} 
\end{center}
Note that both $P$ and $Q$ are both bigons; they each have precisely two vertices when considered in the affine structure of $N_U$ or $M_U$.

\paragraph{Graded ring calculation for $(X_Q,D_Q)$.}
Consider the theta functions $\vartheta_1,\ldots,\vartheta_4$ corresponding to the lattice points $(-1,-1),(1,1)\in Q$ and $(-1,0),(0,-1)\in 2Q$ respectively. In terms of the basis of theta functions of $\CC[U]$, these are
\[ \vartheta_1 = x_1x_2, \quad \vartheta_2 = x_4, \quad \vartheta_3 = x_1, \quad \vartheta_4 = x_2 \]
and, after homogenising, they can be used to give a presentation of the graded ring $R_Q = \CC[\vartheta_0,\vartheta_1,\vartheta_2,\vartheta_3,\vartheta_4]/I$ with generators in degrees $1,1,1,2,2$ respectively. The ideal of equations $I$ defining $R_Q$ is generated by two relations
\[ \vartheta_0^3\vartheta_1=\vartheta_3\vartheta_4 \quad \text{and} \quad \vartheta_1\vartheta_2 = \vartheta_0^2 + \vartheta_3 + \vartheta_4. \]
Using the second equation to eliminate $\vartheta_4$, we get a quartic hypersurface
\[ X_Q\cong \VV\left( \vartheta_1\vartheta_2\vartheta_3 - \vartheta_0^3\vartheta_1 - \vartheta_0^2\vartheta_3 - \vartheta_3^2 \right) \subset \PP(1,1,1,2)_{\vartheta_0,\vartheta_1,\vartheta_2,\vartheta_3}. \]
As may be expected from the spanning fan of $P\subset N_U$, this defines a singular del Pezzo surface of degree $2$, with an $A_2$ singularity at the coordinate point $P_1 = (0:1:0:0)$ and an $A_4$ singularity at the coordinate point $P_2 = (0:0:1:0)$. The boundary divisor $D_Q=\VV(\vartheta_0)$ has two components
\[ D_1 = \VV\left(\vartheta_0, \vartheta_3\right), \qquad D_2= \VV\left(\vartheta_0, \vartheta_1\vartheta_2-\vartheta_3\right) \]
which meet at the two singular points $P_1,P_2$. 

\paragraph{The Landau--Ginzburg model $w_P\colon U\to \CC$.} We form the Landau--Ginzburg potential 
\[ w_P = x_2x_3^2 + 3x_3 + 3x_5 + x_5^2x_1 + 5x_5x_1 + 10x_1 + 10x_2 + 5x_2x_3 \]
obtained by considering binomial edge coefficients on $P$. By restricting to the torus chart $\TT_{12}\subset U$ we obtain a Laurent polynomial, which satisfies the following identity
\[ \left(w_P+12\right)|_{\TT_{12}} = \frac{(1+x_1+x_2)^2(x_1+x_2)^3}{x_1^2x_2^2}. \] 
By \S\ref{sect!period}, the $n$th coefficient $\alpha_n$ of the period $\pi_{w+12}(t)=\sum_{n\geq0}\alpha_nt^n$ is equal to the coefficient of $x_1^{2n}x_2^{2n}$ in 
\[ (1+x_1+x_2)^{2n}(x_1+x_2)^{3n} = \sum_{k=0}^{2n} \binom{2n}{k}(x_1+x_2)^{3n+k}, \]
where the righthand side was obtained by expanding the first bracket by using the binomial formula, treating $1+x_1+x_2$ as the sum of $1$ and $x_1+x_2$. The term $x_1^{2n}x_2^{2n}$ can only appear in the righthand side if $k=n$, and then we easily see that
\[ \alpha_n = \binom{2n}{n}\binom{4n}{2n} \implies \pi_{w_P+12}(t) = \sum_{k\geq0} \frac{(4n)!}{n!n!(2n)!}t^n. \]
This function $\pi_{w_P+12}(t)$ is equal to the \emph{regularised $I$-function} $\widehat{I}_{X_4}(t)$ for a hypersurface $X_4\subset \PP(1,1,1,2)$, and it is known to be a shift of the regularised quantum period $\widehat G_{X_4}(t)$ of $X_4$ \cite[Proposition D.9]{ccgk}. In order to recover $\widehat{G}_{X_4}(t)$ from $\widehat{I}_{X_4}(t)$, the appropriate shift is by the unique constant term required to kill the coefficient $\alpha_1$ of $\widehat{I}_{X_4}(t)$. Since $\alpha_1=\binom21\binom42=12$ we see that $\pi_{w_P}(t) = \widehat{G}_{X_4}(t)$.

This shows that $w_P$ has the right period to be mirror to $X_Q$, but, more precisely, we expect that the particular Landau--Ginzburg model that we have constructed is actually a mirror to the (singular) del Pezzo surface $X_Q$ with its boundary divisor $D_Q$. (Note that $w_P$ has a reducible fibre $w_P^{-1}(-12)$ with two disjoint components; one component of multiplicity 2 and one of multiplicity $3$. After extending $w_P$ to the dual compactification $w_P\colon X_P\dashrightarrow \PP^1$ and resolving into an elliptic fibration, one can show that these two components correspond to the following solid black nodes in a reducible fibre of type $\widetilde E_7$.
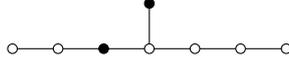
\begin{center}\begin{tikzpicture}[scale=0.6]
   \draw (0,0) -- (6,0) (3,0) -- (3,1);
   \draw[fill=white] (0,0) circle (3pt);
   \draw[fill=white] (1,0) circle (3pt);
   \draw[fill=black] (2,0) circle (3pt);
   \draw[fill=white] (3,0) circle (3pt);
   \draw[fill=white] (4,0) circle (3pt);
   \draw[fill=white] (5,0) circle (3pt);
   \draw[fill=white] (6,0) circle (3pt);
   \draw[fill=black] (3,1) circle (3pt);
\end{tikzpicture}\end{center}
The two white chains of deleted $(-2)$-curves are highly suggestive of the $A_2$ and $A_4$ singularities on $X_Q$.)
\end{eg} 

\begin{eg}
One can do the computations of Example~\ref{eg!dP2} with the roles of $P$ and $Q$ reversed. The graded ring $R_P$ is the homogeneous coordinate ring of a smooth surface $X_P\subset \PP^8$, which is an anticanonically polarised $\mathbb F_1$. The boundary divisor $D_P=D_1\cup D_2$ is the union of two smooth rational curves of self-intersection $D_1^2=1$ and $D_2^2=3$. By restricting to $\TT_{12}\subset U$ we see that the Landau--Ginzburg potential satisfies $w_Q|_{\TT_{12}} = x_1x_2 + \frac{1}{x_1} + \frac{1}{x_2} + \frac{1}{x_1x_2}$ which is a mirror Laurent polynomial for $\mathbb{F}^1$. 
\end{eg}

\section{Dimension 3: the Fano 3-fold $V_{12}$} \label{sect!OGr}

We now describe a parallel story in the 3-dimensional setting which generalises
\begin{enumerate}
\item the cluster structure on $U$ the affine del Pezzo surface of degree 5,
\item the fibration on the affine cone over $\Gr(2,5)$ by such surfaces,
\item the self-dual integral affine manifold with singularities $N_U$ obtained by tropicalising~$U$,
\item the Borisov--Batyrev style mirror symmetry constructions of \S\ref{sec!dP5-mirror-constructions}.
\end{enumerate}
Indeed, the corresponding actors will be the orthogonal Grassmannian $\OGr(5,10)$ and the Fano 3-fold $V_{12}$. We will begin by generalising the second statement by giving a description of $\OGr(5,10)$ which leads to a homogenisation of the 3-dimensional Lyness recurrence (LR$_3$).

\subsection{The orthogonal Grassmannian $\OGr(5,10)$}  \label{sect!OGr-rep-thry}

\subsubsection{$\OGr(5,10)$ as a homogeneous variety}

The orthogonal Grassmannian $\OGr(5,10)$ is one of the two isomorphic irreducible components in the space of $5$-planes $\CC^5 \subset \CC^{10}$ which are isotropic with respect to a given quadratic form. It is the homogeneous variety for the group $\SO(10)$ of type $\pD_5$ with a `half-spinor embedding' $\OGr(5,10)\subset\PP(S^+)$, where $S^+$ is one half of the spin representation $S^+\oplus S^-$ of $\SO(10)$. Given the symmetry broken by choosing $S^+$ over $S^-$, it is perhaps better to think of the isomorphic Grassmannian $\OGr(4,9)\cong \OGr(5,10)$ instead. This is a homogeneous variety for the group $\SO(9)$ of type $\pB_4$, with spinor embedding $\OGr(4,9)\subset \PP(S)$ corresponding to the (irreducible) spin representation $S=\bigoplus_{i=0}^4\bigwedge^i\CC^4$ of $\SO(9)$.
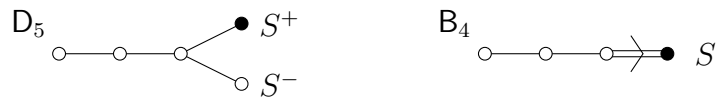
\begin{center}\begin{tikzpicture}[scale=0.8]
	\draw (0,0) -- (2,0) -- (3,0.5) (2,0) -- (3,-0.5);
	\draw (7,0) -- (9,0) (9,0.05) -- (10,0.05) (9,-0.05) -- (10,-0.05);
	\draw (9.4,0.3) -- (9.6,0) -- (9.4,-0.3);
	
	\node at (-1/2,1/2) {$\pD_5$};
	\node at (13/2,1/2) {$\pB_4$};
	
	\draw[fill=white] (0,0) circle (3pt);
	\draw[fill=white] (1,0) circle (3pt);
	\draw[fill=white] (2,0) circle (3pt);
	\draw[fill=black] (3,0.5) circle (3pt) node[xshift = 0.5cm]{$S^+$};
	\draw[fill=white] (3,-0.5) circle (3pt) node[xshift = 0.5cm]{$S^-$};
	
	\draw[fill=white] (7,0) circle (3pt);
	\draw[fill=white] (8,0) circle (3pt);
	\draw[fill=white] (9,0) circle (3pt);
	\draw[fill=black] (10,0) circle (3pt) node[xshift = 0.5cm]{$S$};
\end{tikzpicture}\end{center}
The representation $S$ is 16-dimensional with weights $\tfrac12(\pm1,\pm1,\pm1,\pm1)$, which are the vertices of a 4-dimensional cube $C$ in the weight lattice for $\pB_4$. The Weyl group $W(\pB_4)$ acts as the full symmetry group of $C$ and a Coxeter element in $W(\pB_4)$ acts on $C$ as a rotation of order 8. Taking the orthogonal projection onto the Coxeter plane (shown in Figure~\ref{fig!B4-cube}(a)) we see this 8-fold rotational symmetry of $C$ which splits the vertices into two groups of size~8. 
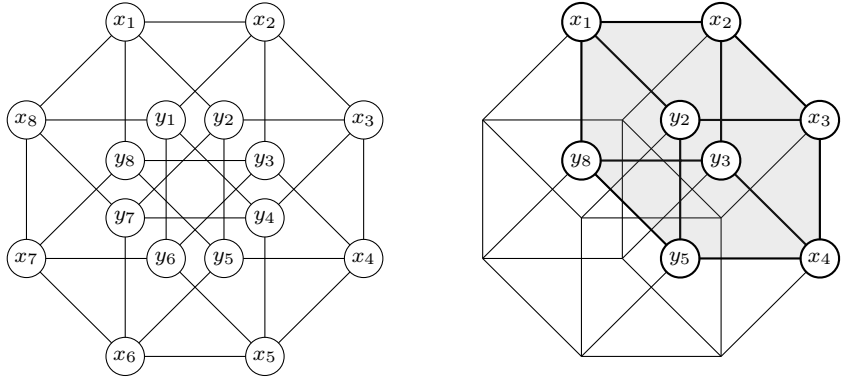
\begin{figure}[htbp]
\begin{center}
\begin{tikzpicture}[scale=1.2]
\begin{scope}
\foreach \i in {1,...,8}
{
   \draw ({0.8284*cos(45*\i+22.5)},{0.8284*sin(45*\i+22.5)}) -- ({2*cos(45*(\i-1)+22.5)},{2*sin(45*(\i-1)+22.5)});
   \draw ({0.8284*cos(45*\i+22.5)},{0.8284*sin(45*\i+22.5)}) -- ({2*cos(45*(\i+1)+22.5)},{2*sin(45*(\i+1)+22.5)});
   \draw ({2*cos(45*\i+22.5)},{2*sin(45*\i+22.5)}) -- ({2*cos(45*(\i+1)+22.5)},{2*sin(45*(\i+1)+22.5)});
   \draw ({0.8284*cos(45*\i+22.5)},{0.8284*sin(45*\i+22.5)}) -- ({0.8284*cos(45*(\i+3)+22.5)},{0.8284*sin(45*(\i+3)+22.5)});
}
	
\foreach \i in {1,...,8}
{
   \draw[fill=white] ({0.8284*cos(45*(3-\i)+22.5)},{0.8284*sin(45*(3-\i)+22.5)}) circle (6pt) node {\scriptsize $y_{\i}$};
   \draw[fill=white] ({2*cos(45*(3-\i)+22.5)},{2*sin(45*(3-\i)+22.5)}) circle (6pt) node {\scriptsize $x_{\i}$};
}
\end{scope}
\begin{scope}[xshift = 5cm]
\foreach \i in {1,...,8}
{
   \draw ({0.8284*cos(45*\i+22.5)},{0.8284*sin(45*\i+22.5)}) -- ({2*cos(45*(\i-1)+22.5)},{2*sin(45*(\i-1)+22.5)});
   \draw ({0.8284*cos(45*\i+22.5)},{0.8284*sin(45*\i+22.5)}) -- ({2*cos(45*(\i+1)+22.5)},{2*sin(45*(\i+1)+22.5)});
   \draw ({2*cos(45*\i+22.5)},{2*sin(45*\i+22.5)}) -- ({2*cos(45*(\i+1)+22.5)},{2*sin(45*(\i+1)+22.5)});
   \draw ({0.8284*cos(45*\i+22.5)},{0.8284*sin(45*\i+22.5)}) -- ({0.8284*cos(45*(\i+3)+22.5)},{0.8284*sin(45*(\i+3)+22.5)});
}
\end{scope}

\begin{scope}[xshift = 5cm, rotate=-22.5]   
   \draw[thick] ({2*cos(45*0)},{2*sin(45*0)}) -- ({0.8284*cos(45*1)},{0.8284*sin(45*1)}) -- ({2*cos(45*2)},{2*sin(45*2)})  ({0.8284*cos(45*1)},{0.8284*sin(45*1)}) -- ({0.8284*cos(45*4)},{0.8284*sin(45*4)})  ({2*cos(45*3)},{2*sin(45*3)}) --  ({0.8284*cos(45*2)},{0.8284*sin(45*2)}) -- ({2*cos(45*1)},{2*sin(45*1)})  ({0.8284*cos(45*2)},{0.8284*sin(45*2)}) -- ({0.8284*cos(45*7)},{0.8284*sin(45*7)});
   \draw[thick,fill=gray, fill opacity=0.15] ({2*cos(45*0)},{2*sin(45*0)}) -- ({2*cos(45*1)},{2*sin(45*1)}) -- ({2*cos(45*2)},{2*sin(45*2)}) -- ({2*cos(45*3)},{2*sin(45*3)}) -- ({0.8284*cos(45*4)},{0.8284*sin(45*4)}) -- ({0.8284*cos(45*7)},{0.8284*sin(45*7)}) -- cycle;
\draw[fill=white,thick] ({2*cos(45*(9-3*2))},{2*sin(45*(9-3*2))}) circle (6pt) node {\textcolor{black}{\scriptsize $x_1$}};
\draw[fill=white,thick] ({2*cos(45*(9-3*5))},{2*sin(45*(9-3*5))}) circle (6pt) node {\textcolor{black}{\scriptsize $x_2$}};
\draw[fill=white,thick] ({2*cos(45*(9-3*8))},{2*sin(45*(9-3*8))}) circle (6pt) node {\textcolor{black}{\scriptsize $x_3$}};
\draw[fill=white,thick] ({2*cos(45*(9-3*3))},{2*sin(45*(9-3*3))}) circle (6pt) node {\textcolor{black}{\scriptsize $x_4$}};
\draw[fill=white,thick] ({0.8284*cos(45*(5-3*3))},{0.8284*sin(45*(5-3*3))}) circle (6pt) node {\textcolor{black}{\scriptsize $y_8$}};
\draw[fill=white,thick] ({0.8284*cos(45*(5-3*4))},{0.8284*sin(45*(5-3*4))}) circle (6pt) node {\textcolor{black}{\scriptsize $y_3$}};
\draw[fill=white,thick] ({0.8284*cos(45*(5-3*1))},{0.8284*sin(45*(5-3*1))}) circle (6pt) node {\textcolor{black}{\scriptsize $y_2$}};
\draw[fill=white,thick] ({0.8284*cos(45*(5-3*2))},{0.8284*sin(45*(5-3*2))}) circle (6pt) node {\textcolor{black}{\scriptsize $y_5$}};
\end{scope}
\end{tikzpicture}
\caption{(a) The sixteen spinor variables labelling the 4-cube $C$. (b) The face of $C$ which corresponds to the equation $x_1x_4=x_2y_5+x_3y_8+y_2y_3$.}
\label{fig!B4-cube}
\end{center}
\end{figure}

\subsubsection{The equations of $\OGr(5,10)$} 
Let $R = \CC[S]$ and name the 16 spinor variables $x_1,\ldots,x_8,y_1,\ldots,y_8\in R$, indexed by $\ZZ/8\ZZ$, according to the labels in Figure~\ref{fig!B4-cube}(a). The corresponding vertices of $C$ are represented by the columns of the following matrix (where the minus signs in front of $y_3$ and $y_7$ are only chosen to make the equations displayed below more beautiful).
\[ \begin{array}{cccccccc|cccccccc}
x_1&x_2&x_3&x_4&x_5&x_6&x_7&x_8&y_1&y_2&-y_3&y_4&y_5&y_6&-y_7&y_8 \\ \hline
+&-&-&-&-&+&+&+& -&+&-&-&+&-&+&+\\
+&+&-&-&-&-&+&+& +&-&+&-&-&+&-&+\\
+&+&+&-&-&-&-&+& +&+&-&+&-&-&+&-\\
+&+&+&+&-&-&-&-& -&+&+&-&+&-&-&+
\end{array} \]
The ten quadratic equations defining $\OGr(5,10)$ are now obtained by swapping minus signs in columns that differ in three or more places. For example,
\[ \left(\begin{smallmatrix}
+ \\ + \\ + \\ + \end{smallmatrix}\right)\left(\begin{smallmatrix}
+ \\ - \\ - \\ - \end{smallmatrix}\right) = \left(\begin{smallmatrix}
+ \\ + \\ + \\ - \end{smallmatrix}\right)\left(\begin{smallmatrix}
+ \\ - \\ - \\ + \end{smallmatrix}\right) - \left(\begin{smallmatrix}
+ \\ - \\ + \\ - \end{smallmatrix}\right)\left(\begin{smallmatrix}
+ \\ + \\ - \\ + \end{smallmatrix}\right) + \left(\begin{smallmatrix}
+ \\ + \\ - \\ - \end{smallmatrix}\right)\left(\begin{smallmatrix}
+ \\ - \\ + \\ + \end{smallmatrix}\right) \]
corresponds to the equation $x_1x_6 = x_8y_5 + y_7y_8 + x_7y_2$. The full ideal $I$ defining $\OGr(5,10)$ is given by
\begin{align*}  
x_1x_4 &=  x_2y_5 + x_3y_8 + y_2y_3 &&& 
x_5x_8 &=  x_6y_1 + x_7y_4 + y_6y_7 &&& x_1x_5 - x_3x_7 &= y_1y_5 - y_3y_7 \\
x_2x_5 &=  x_3y_6 + x_4y_1 + y_3y_4 &&&
x_6x_1 &=  x_7y_2 + x_8y_5 + y_7y_8 &&& x_2x_6 - x_4x_8 &= y_2y_6 - y_4y_8 \\ 
x_3x_6 &=  x_4y_7 + x_5y_2 + y_4y_5 &&&
x_7x_2 &=  x_8y_3 + x_1y_6 + y_8y_1 \\
x_4x_7 &=  x_5y_8 + x_6y_3 + y_5y_6 &&& 
x_8x_3 &=  x_1y_4 + x_2y_7 + y_1y_2	
\end{align*}
where the first two columns give the eight periodic relations
\[ x_ix_{i+3} =  x_{i+1}y_{i+4} + x_{i+2}y_{i-1} + y_{i+1}y_{i+2} \qquad i\in\ZZ/8\ZZ \]
which are a homogenisation of (LR$_3$) by the variables $y_1,\ldots,y_8$. The binomials appearing in these eight equations correspond to the antipodal vertices in the eight 3-cube faces of $C$, as shown in Figure~\ref{fig!B4-cube}(b). The last two equations are implied from the first eight and correspond to the two orthoplexes (4-dimensional octahedra) inscribed on the two bipartite decompositions on the set of vertices of $C$.

\paragraph{The affine cone $\sU$.}
We let $\sU=\Spec(R/I)\subset \Aa^{16}$ be the affine variety defined by these equations, i.e.\ the affine cone over $\OGr(5,10)$. Note that there is a map $i\colon \sU\to \sU$ with $i(x_i)=-y_{3i}$ and $i(y_i)=x_{3i+4}$, which switches the role of the $x$ variables and the $y$ variables. Therefore, after setting $x_i=-1$ for all $i$ we see that $(y_{3i}\colon i\in\ZZ/8\ZZ)$ is also a solution to the Lyness recurrence (LR$_3$). 

\paragraph{Anticanonical class.}
The ideal $I$ is a Gorenstein ideal of codimension 5 and has minimal resolution
\[ R \leftarrow R(-2)^{\oplus 10} \leftarrow R(-3)^{\oplus 16} \leftarrow R(-5)^{\oplus 16} \leftarrow R(-6)^{\oplus 10} \leftarrow R(-8) \leftarrow 0.  \]
From the last module we can read off the adjunction number for $\sU\subset\Aa^{16}$, giving $-K_\sU=\sO_{\Aa^{16}}(16-8)|_\sU=\sO_\sU(8)$. In particular, if we let $\sD_i=\VV(y_i)$ and $\sE_i=\VV(x_i)$ for $i=1,\ldots,8$, then both $\sD=\bigcup_{i=1}^8\sD_i$ and $\sE=\bigcup_{i=1}^8\sE_i$ define anticanonical divisors in $\sU$.

\subsubsection{A fibration of $\sU$ by affine Fano 3-folds}\label{sect!U-lambda-mu}

Consider the projection $p\colon \sU \to \Aa^8_{y_1,\ldots,y_8}$. The fibres of $p$ are a flat family of affine Gorenstein 3-folds which spreads out the components of $\sD$ above the coordinate strata of $\Aa^8$. 

\paragraph{The central fibre $U_0$.}
The central fibre $U_0 := p^{-1}(0)$ is given by the common intersection of all components of $\sD\subset \sU$. It breaks up a reducible affine toric 3-fold with ten components:
\[ \bigcap_{i=1}^8\sD_i = \VV(y_1,\ldots,y_8) = \bigcup_{i=1}^8 \Aa^3_{x_{i-1},x_i,x_{i+1}} \cup Q_{1357} \cup Q_{2468} \]
where $Q_{1357} = \VV(x_1x_5-x_3x_7)\subset \Aa^4_{x_1,x_3,x_5,x_7}$ and $Q_{2468} = \VV(x_2x_6-x_4x_8)\subset \Aa^4_{x_2,x_4,x_6,x_8}$ are both isomorphic to the cone over the Segre embedding of $\PP^1\times \PP^1$. Therefore $U_0$ looks like the cone over a reducible toric surface 
\[ D_P := D_{123}\cup D_{234} \cup \ldots \cup D_{812} \cup D_{1357} \cup D_{2468}, \]
with ten components $D_{1357}=\PP(Q_{1357})$, $D_{2468} = \PP(Q_{2468})$ and $D_{123} = \PP^2_{x_1,x_2,x_3}$, etc. These components intersect along their toric boundary strata like the polytope $P$ of Figure~\ref{fig!polytopes}. Algebraically the coordinate ring of the affine variety $U_0$ is given by the Stanley--Reisner ring associated to the dual polytope $Q=P^\star$.
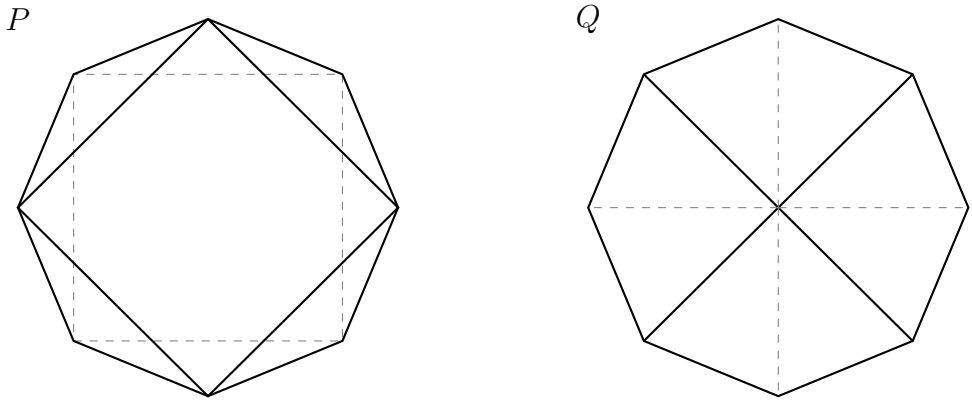
\begin{figure}[htbp]
\begin{center}
\begin{tikzpicture}[scale=2.5]
  
  \draw[gray,dashed] ({cos(45)},{sin(45)}) -- ({cos(135)},{sin(135)}) -- ({cos(225)},{sin(225)}) -- ({cos(315)},{sin(315)}) -- cycle;
  \draw[thick] ({cos(0)},{sin(0)}) -- ({cos(45)},{sin(45)}) -- ({cos(90)},{sin(90)}) -- ({cos(135)},{sin(135)}) -- ({cos(180)},{sin(180)}) -- ({cos(225)},{sin(225)}) -- ({cos(270)},{sin(270)}) -- ({cos(315)},{sin(315)}) -- cycle;
  \draw[thick] ({cos(0)},{sin(0)}) -- ({cos(90)},{sin(90)}) -- ({cos(180)},{sin(180)}) -- ({cos(270)},{sin(270)}) -- cycle;  
  \node at (-1,1) {$P$};
  
  \draw[thick] (3,0) -- ({cos(45)+3},{sin(45)})  (3,0) -- ({cos(135)+3},{sin(135)}) (3,0) -- ({cos(225)+3},{sin(225)}) (3,0) -- ({cos(315)+3},{sin(315)});
  \draw[thick] ({cos(0)+3},{sin(0)}) -- ({cos(45)+3},{sin(45)}) -- ({cos(90)+3},{sin(90)}) -- ({cos(135)+3},{sin(135)}) -- ({cos(180)+3},{sin(180)}) -- ({cos(225)+3},{sin(225)}) -- ({cos(270)+3},{sin(270)}) -- ({cos(315)+3},{sin(315)}) -- cycle;
  \draw[gray,dashed] (3,0) -- ({cos(0)+3},{sin(0)}) (3,0) -- ({cos(90)+3},{sin(90)}) (3,0) -- ({cos(180)+3},{sin(180)}) (3,0) -- ({cos(270)+3},{sin(270)});
  
  \node at (-1+3,1) {$Q$};
\end{tikzpicture}
\caption{The polytope $P$ and the dual polytope $Q=P^*$.}
\label{fig!polytopes}
\end{center}
\end{figure}

\paragraph{The general fibre $U_{\lambda,\mu}$.}
We now look at a general fibre of this fibration $U_y = p^{-1}(y)$ where $y=(y_1,\ldots,y_8)\in(\CC^\times)^8$. We can rescale the coordinates on $U_y$ by 
\[ (x_1,\ldots,x_8)\mapsto \left(\frac{y_1y_2}{y_4}x_1, \; \frac{y_1y_2}{y_7}x_2, \; \frac{y_2y_3}{y_8}x_3, \; \frac{y_4y_5}{y_7}x_4, \; \frac{y_4y_5}{y_2}x_5, \; \frac{y_4y_5y_8}{y_2y_3}x_6, \; \frac{y_7y_8}{y_2}x_7, \; \frac{y_8y_1}{y_3}x_8 \right) \] 
to find that $U_y \cong U_{\lambda,\mu}$, for an affine 3-fold $U_{\lambda,\mu}\subset\Aa^8_{x_1,\ldots,x_8}$ depending on two parameters $\lambda = \frac{y_3y_7}{y_1y_5}$ and $\mu = \frac{y_2y_6}{y_4y_8}$ with $\lambda,\mu\in\CC^\times$. The equations defining $U_{\lambda,\mu}$ become:
\begin{equation*}
\begin{aligned} 
x_1x_4 &=  x_2 + \lambda x_3 + \lambda &&& 
x_5x_8 &=  x_6 + \lambda x_7 + \lambda\mu &&& 
x_1x_5 - \lambda x_3x_7 &= 1 - \lambda \\
x_2x_5 &=  \lambda\mu x_3 + x_4 + \lambda &&& 
x_6x_1 &=  \lambda x_7 + x_8 + \lambda &&& 
x_2x_6 - x_4x_8 &= \lambda\mu - \lambda \\
x_3x_6 &=  x_4 + x_5 + 1 &&&
x_7x_2 &=  x_8 + \mu x_1 + 1 \\ 
x_4x_7 &=  x_5 + x_6 + \mu &&&
x_8x_3 &=  x_1 + x_2 + 1	
\end{aligned}
\end{equation*}
Thus the fibres of $p$ are isomorphic over the intersection of two quadrics in the base given by $\VV(\lambda y_1y_5- y_3y_7,\; y_2y_6- \mu y_4y_8)\cap(\CC^\times)^8$.

\begin{lem}\label{lem!sing}
The projective closure $X_{\lambda,\mu} := \overline{U}_{\lambda,\mu}\subset\PP^8$ has boundary divisor $D_{\lambda,\mu}$ isomorphic to $D_P$, and $X_{\lambda,\mu}$ has 16 ordinary (non-$\QQ$-factorial) nodal singularities along $D_{\lambda,\mu}$, given by the 8 coordinate points $P_{x_i}$ and the 8 points:
\[ (1:-1)\in\PP^1_{x_1,x_2}, \quad (\lambda:-1)\in\PP^1_{x_2,x_3}, \quad (1:-\lambda\mu)\in\PP^1_{x_3,x_4}, \quad (1:-1)\in\PP^1_{x_4,x_5}, \]
\[ (1:-1)\in\PP^1_{x_5,x_6}, \quad (\lambda:-1)\in\PP^1_{x_6,x_7}, \quad (1:-\lambda)\in\PP^1_{x_7,x_8}, \quad (\mu:-1)\in\PP^1_{x_8,x_1}. \]
Moreover, the interior $U_{\lambda,\mu}$
\begin{enumerate}
\item is smooth if both $\lambda,\mu\neq 1$, 
\item acquires one node if either $\lambda=1$ or $\mu=1$, 
\item acquires two nodes if both $\lambda=\mu=1$.
\end{enumerate}
\end{lem}

\begin{proof}
If $\mu=1$ then the point $(-1,0,-1,0,-1,0,-1,0)$ belongs to $U_{\lambda,1}$ and we can eliminate the variables $x_1,x_3,x_5,x_7$ in a neighbourhood of this point to find an ordinary nodal singularity with local equation $(x_2x_6-x_4x_8=0) \subset \Aa^4_{x_2,x_4,x_6,x_8}$. Similarly if $\lambda=1$ the point $(0,-1,0,-1,0,-1,0,-1)$ becomes a node of $U_{1,\mu}$. If both $\lambda=\mu=1$ then this gives nodes at two distinct points of $U_{1,1}$. Smoothness elsewhere and the singular locus along $D_{\lambda,\mu}$ can be checked by  computer algebra, e.g.\ Macaulay2. 
\end{proof}

\subsection{The affine Fano-3-fold $U_{\lambda,\mu}$}

We now study the affine Fano 3-fold $U_{\lambda,\mu} = X_{\lambda,\mu}\setminus D_{\lambda,\mu}$, which has a (generalised) cluster structure and is the 3-dimensional analogue of the affine del Pezzo surface considered in \S\ref{sect!dP5}.

\subsubsection{Blowup description of $U_{\lambda,\mu}$}
One way that the affine del Pezzo surface can be obtained is by blowing up one point on each of the coordinate axis in $\Aa^2$ and deleting the strict transform of the boundary. We now give a similar description for the 3-folds $U_{\lambda,\mu}$. 

\begin{prop}\label{prop!blowup}
Let $H=\VV(x_1x_2x_3)\subset \Aa^3$ denote the union of the coordinate hyperplanes. There is a locus $Z\subset H\subset \Aa^3$ such that $U_{\lambda,\mu} \cong Bl_Z(\Aa^3) \setminus \widetilde H$ -- i.e.\ the blowup of $Z$ minus the strict transform of $H$. According to the cases of Lemma~\ref{lem!sing}:
\begin{enumerate}
\item if $\lambda,\mu\neq 1$ then $Z$ consists of a conic and two lines
\[ Z = \VV(x_1, \lambda + x_2 + \lambda x_3) \cup \VV(x_2, 1 + x_1 + x_3 + \mu x_1x_3 ) \cup \VV(x_3, 1 + x_1 + x_2), \]  
\item if $\lambda\neq1,\mu=1$ then $Z$ consists of four lines (two lines and a reducible conic)
\[ Z = \VV(x_1, \lambda + x_2 + \lambda x_3) \cup \VV(x_2, 1 + x_1 ) \cup \VV(x_2, 1 + x_3 ) \cup \VV(x_3, 1 + x_1 + x_2), \]
\item if $\lambda=1,\mu\neq1$ then $Z$ consists of a conic and two lines meeting at an embedded point
\[ Z = \VV(x_1, 1 + x_2 + x_3) \cup \VV(x_2, 1 + x_1 + x_3 + \mu x_1x_3 ) \cup \VV(x_3, 1 + x_1 + x_2) \cup \VV(x_1, 1 + x_2, x_3), \] 
\item if $\lambda=\mu=1$ then $Z$ consists of a reducible conic and two lines meeting at an embedded point
\[ Z = \VV(x_1, 1 + x_2 + x_3) \cup \VV(x_2, 1 + x_1 ) \cup \VV(x_2, 1 + x_3 ) \cup \VV(x_3, 1 + x_1 + x_2) \cup \VV(x_1, 1 + x_2, x_3). \]
\end{enumerate}
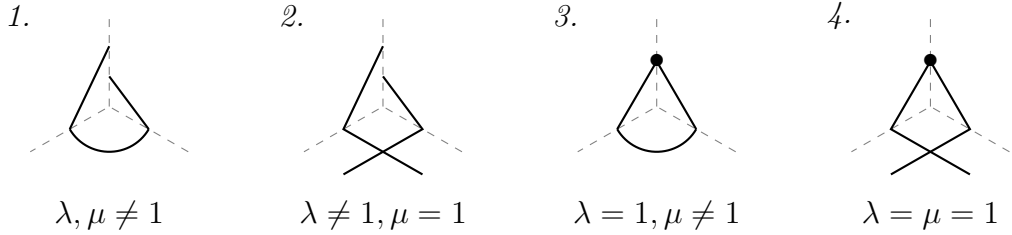
\begin{center}\begin{tikzpicture}[scale=1.2]
	\begin{scope}
		\node at (-1,1) {1.};
		\node at (0,-1.2) {$\lambda,\mu\neq1$};
		\draw[gray,dashed] ({sin(120)},{cos(120)}) -- (0,0) -- ({sin(240)},{cos(240)}) (0,0) -- (0,1);
		\draw[thick] (0,1/3) -- ({sin(120)/2},{cos(120)/2});
		\draw[thick,domain=120:240, samples=30] plot ({sin(\x)/2}, {cos(\x)/2} );
		\draw[thick] (0,2/3) -- ({sin(240)/2},{cos(240)/2});
	\end{scope}
	\begin{scope}[xshift=3 cm]
		\node at (-1,1) {2.};
		\node at (0,-1.2) {$\lambda\neq1,\mu=1$};
		\draw[gray,dashed] ({sin(120)},{cos(120)}) -- (0,0) -- ({sin(240)},{cos(240)}) (0,0) -- (0,1);
		\draw[thick] (0,1/3) -- ({sin(120)/2},{cos(120)/2}) -- ({sin(120)/2+sin(240)},{cos(120)/2+cos(240)});
		\draw[thick] (0,2/3) -- ({sin(240)/2},{cos(240)/2}) -- ({sin(240)/2+sin(120)},{cos(240)/2+cos(120)});
	\end{scope}
	\begin{scope}[xshift=6 cm]
		\node at (-1,1) {3.};
		\node at (0,-1.2) {$\lambda=1,\mu\neq1$};
		\draw[gray,dashed] ({sin(120)},{cos(120)}) -- (0,0) -- ({sin(240)},{cos(240)}) (0,0) -- (0,1);
		\draw[thick] (0,1/2) -- ({sin(120)/2},{cos(120)/2});
		\draw[thick,domain=120:240, samples=30] plot ({sin(\x)/2}, {cos(\x)/2} );
		\draw[thick] (0,1/2) -- ({sin(240)/2},{cos(240)/2});
		\node at (0,1/2) {$\bullet$};
	\end{scope}
	\begin{scope}[xshift=9 cm]
		\node at (-1,1) {4.};
		\node at (0,-1.2) {$\lambda=\mu=1$};
		\draw[gray,dashed] ({sin(120)},{cos(120)}) -- (0,0) -- ({sin(240)},{cos(240)}) (0,0) -- (0,1);
		\draw[thick] (0,1/2) -- ({sin(120)/2},{cos(120)/2}) -- ({sin(120)/2+sin(240)},{cos(120)/2+cos(240)});
		\draw[thick] (0,1/2) -- ({sin(240)/2},{cos(240)/2}) -- ({sin(240)/2+sin(120)},{cos(240)/2+cos(120)});
		\node at (0,1/2) {$\bullet$};
	\end{scope}
\end{tikzpicture}\end{center}
\end{prop}

\begin{rmk}\label{rmk!blowup}
Before proving the proposition, we briefly describe the effect of the blowups we make and explain how $U_{\lambda,\mu}$ ends up with the number of nodes expected from Lemma~\ref{lem!sing}.
\begin{enumerate}
\item The blowup $Bl_Z(\Aa^3)$ obtains an ordinary node in the fibre over any point $P\in Z$ contained in the intersection of two curves. If $P$ is contained in a 1-stratum of the boundary then after the blowup the node is contained in $\widetilde{H}$ and hence does not appear in $U_{\lambda,\mu}$. However, if $P$ is contained in the interior of a 2-stratum (as happens in cases 2.\ and 4.) then the node is not contained in $\widetilde{H}$ and therefore must appear in $U_{\lambda,\mu}$.
\item Blowing up an embedded point in the intersection of two lines also produces a node in the interior $U_{\lambda,\mu}$ (as happens in cases 3.\ and 4.). To see how, we can first blow up the embedded point with exceptional divisor $E$. Let $L\subset E$ be the line that passes through the strict transform of the two lines on either side. Blowing up the lines turns $L$ into a contractible $\sO(-1,-1)$ curve, which is contracted to an ordinary node. 
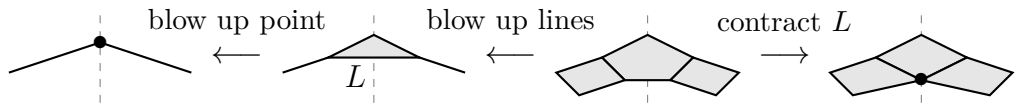
\begin{center}\begin{tikzpicture}[scale=1.2]
	\begin{scope}
		\draw[gray,dashed] (0,0) -- (0,1);
		\draw[thick] (-1,1/3) -- (0,2/3) -- (1,1/3);
		\node at (0,2/3) {$\bullet$};
	\end{scope}
	\begin{scope}[xshift=3cm]
		\draw[gray,dashed] (0,0) -- (0,1);
		\draw[thick, fill = gray!20] (-1/2,1/2) -- (0,3/4) -- (1/2,1/2) -- cycle;
		\node at (-0.2,0.3) {$L$};
		\draw[thick] (-1/2,1/2) -- (-1,1/3) (1/2,1/2) -- (1,1/3);
	\end{scope}
	\begin{scope}[xshift=6cm]
		\draw[gray,dashed] (0,0) -- (0,1);
		\draw[thick, fill = gray!20] (-1/2+1/4,1/2-1/4) -- (-1/2,1/2) -- (0,3/4) -- (1/2,1/2) -- (1/2-1/4,1/2-1/4) -- cycle;
		\draw[thick, fill = gray!20] (-1/2,1/2) -- (-1,1/3) -- (-1+1/4,1/3-1/4) -- (-1/2+1/4,1/2-1/4) -- cycle;
		\draw[thick, fill = gray!20] (1/2,1/2) -- (1,1/3) -- (1-1/4,1/3-1/4) -- (1/2-1/4,1/2-1/4) -- cycle;
	\end{scope}
	\begin{scope}[xshift=9cm]
		\draw[gray,dashed] (0,0) -- (0,1);
		\draw[thick, fill = gray!20] (0,1/4) -- (-1/2,1/2) -- (0,3/4) -- (1/2,1/2) -- cycle;
		\draw[thick, fill = gray!20] (-1/2,1/2) -- (-1,1/3) -- (-1+1/4,1/3-1/4) -- (0,1/4) -- cycle;
		\draw[thick, fill = gray!20] (1/2,1/2) -- (1,1/3) -- (1-1/4,1/3-1/4) -- (0,1/4) -- cycle;
		\node at (0,1/4) {$\bullet$};
	\end{scope}
	\node at (1.5,0.9) {\small blow up point};
	\node at (1.5,0.5) {$\longleftarrow$};
	\node at (4.5,0.9) {\small blow up lines};
	\node at (4.5,0.5) {$\longleftarrow$};
	\node at (7.5,0.9) {\small contract $L$};
	\node at (7.5,0.5) {$\longrightarrow$};
\end{tikzpicture}\end{center}
\item It is interesting that cases 2.\ and 3.\ give two different constructions of varieties that should be isomorphic by exchanging the role of $\lambda$ and $\mu$. Indeed, we can send the configuration 3.\ to the configuration 2.\ by the mutating one of the two line components, e.g.\ we blow up the line $\VV(x_3,1+x_1+x_2)$ and then contract the strict transform of the divisor $\VV(1+x_1+x_2)$ to the plane at infinity.
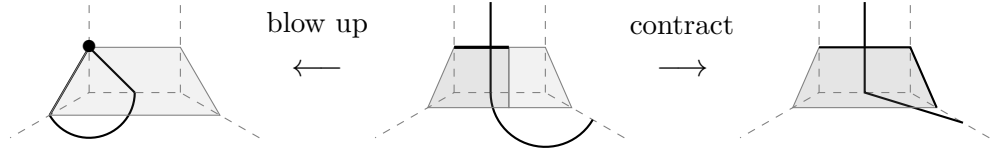
\begin{center}\begin{tikzpicture}[scale=1.2]
	\begin{scope}
		\draw[gray,dashed] ({1+sin(120)},{cos(120)}) -- (1,0) -- (0,0) -- ({sin(240)},{cos(240)}) (0,0) -- (0,1) (1,0) -- (1,1);
		\draw[thick] (0,1/2) -- (1/2,0);
		\draw[thick,domain=90:240, samples=30] plot ({sin(\x)/2}, {cos(\x)/2} );
		\draw[thick] (0,1/2) -- ({sin(240)/2},{cos(240)/2});
		\draw[gray,fill=gray, fill opacity = 0.1] (1,1/2) -- ({1+sin(120)/2},{cos(120)/2}) -- ({sin(240)/2},{cos(240)/2}) -- (0,1/2) -- cycle;
		\node at (0,1/2) {$\bullet$};
	\end{scope}
	
	\node at (2.5,3/4) {\small blow up};
	\node at (2.5,1/4) {$\longleftarrow$};
	
	\begin{scope}[xshift=4cm]
		\draw[gray,dashed] ({1+sin(120)},{cos(120)}) -- (1,0) -- (0,0) -- ({sin(240)},{cos(240)}) (0,0) -- (0,1) (1,0) -- (1,1);
		\draw[thick] (2/5,1) -- (2/5,0);
		\draw[thick,domain=120:270, samples=30] plot ({1+3*sin(\x)/5}, {3*cos(\x)/5} );
		\draw[gray,fill=gray, fill opacity = 0.2] (0,1/2) -- ({sin(240)/3},{cos(240)/3}) -- (3/5,{cos(240)/3}) -- (3/5,1/2) -- cycle;
		\draw[gray,fill=gray, fill opacity = 0.1] (1,1/2) -- ({1+sin(120)/3},{cos(120)/3}) -- (3/5,{cos(240)/3}) -- (3/5,1/2) -- cycle;
		\draw[very thick] (0,1/2) -- (3/5,1/2);
	\end{scope}
	
	\node at (6.5,3/4) {\small contract};
	\node at (6.5,1/4) {$\longrightarrow$};
	
	\begin{scope}[xshift=8cm]
		\draw[gray,dashed] ({1+sin(120)},{cos(120)}) -- (1,0) -- (0,0) -- ({sin(240)},{cos(240)}) (0,0) -- (0,1) (1,0) -- (1,1);
		\draw[thick] (1/2,1) -- (1/2,0) -- ({1+2*sin(120)/3},{2*cos(120)/3});
		\draw[gray,fill=gray, fill opacity = 0.2] (1,1/2) -- ({1+sin(120)/3},{cos(120)/3}) -- ({sin(240)/3},{cos(240)/3}) -- (0,1/2) -- cycle;
		\draw[thick] (0,1/2) -- (1,1/2) -- ({1+sin(120)/3},{cos(120)/3});
	\end{scope}
\end{tikzpicture}\end{center}
\end{enumerate}
\end{rmk}

\begin{proof}
Consider $g \colon U_{\lambda,\mu}\subset \Aa^8 \to \Aa^3_{x_1,x_2,x_3}$, which is a birational projection of $U_{\lambda,\mu}$ onto its image. This image is the constructible set 
\[ g(U_{\lambda,\mu}) = (\CC^\times)^3_{x_1,x_2,x_3} \cup Z^0, \qquad \text{where } Z^0=\begin{cases}
Z\setminus\VV(x_1,x_3) & \text{if }\lambda\neq1 \\
Z & \text{if }\lambda = 1 
\end{cases} \]
and $Z$ is as in the statement of the proposition. (In other words, if $\lambda\neq1$ the image of $U_{\lambda,\mu}$ misses the two points of $Z$ that lie on the $x_2$-axis.) Now $g$ restricts to an isomorphism on the open torus $(\CC^\times)^3_{x_1,x_2,x_3}\subset U_{\lambda,\mu}$ and we can check that the fibres over $z\in Z$ are all affine lines $g^{-1}(z)\cong \Aa^1$, unless $\lambda=1$ and $z=(0,-1,0)$, in which case $g^{-1}(z) \cong \Aa^2$. We can resolve the inverse map $g^{-1}\colon g(U_{\lambda,\mu}) \dashrightarrow U_{\lambda,\mu}$ by blowing up $Z\subset \Aa^3$, as described in Remark~\ref{rmk!blowup}. 
\end{proof}

\subsubsection{The exchange graph of $U_{\lambda,\mu}$} \label{sec!U-exchange-graph}

Recall our fibration $p\colon \mathcal U\to \Aa^8_{y_1,\ldots,y_8}$, which we can think of as a family of affine 3-folds over the coefficient ring $R_y=\CC[y_1,\ldots,y_8]$. We let $\Aa_y:=\Aa^8_{y_1,\ldots,y_8}$ denote the base of this fibration. In the traditional language of cluster algebras, we think of the $x$ variables as \emph{cluster variables} on $\sU$ and the $y$ variables as \emph{frozen variables}. Moreover, we recall the two quadratic terms $q_1,q_2$ which we introduced in \S\ref{sect!LR3-terms}. These are homogenised as follows:
\[ q_1 := x_1x_5 - y_1y_5 = x_3x_7 - y_3y_7, \qquad q_2 := x_2x_6 - y_2y_6 = x_4x_8 - y_4y_8. \]
For convenience of notation, we let $q_1=q_3=q_5=q_7$ and $q_2=q_4=q_6=q_8$. For reasons that will shortly become clear (see Remark~\ref{rmk!why-q1-q2}), we want to add $q_1$ and $q_2$ to our list of cluster variables.

\begin{lem}\label{lem!lp}
For any $i=1,\ldots,8$, we have
\[ x_1,\ldots,x_8,q_1,q_2 \in R_y[x_{i-1}^{\pm1},x_i^{\pm1},x_{i+1}^{\pm1}] \quad \text{and} \quad x_1,\ldots,x_8,q_1,q_2 \in R_y[x_{i}^{\pm1},q_{i}^{\pm1},x_{i+2}^{\pm1}]. \]
\end{lem}

\begin{proof}
This is straightforward to check by simply expanding all of the cluster variables in the corresponding Laurent ring.
\end{proof}

We let $\TT_{i-1,i,i+1} := \Spec R_y[x_{i-1}^{\pm1},x_i^{\pm1},x_{i+1}^{\pm1}]$ and $\TT_{i,q,i+2} = \Spec R_y[x_{i}^{\pm1},q_{i}^{\pm1},x_{i+2}^{\pm1}]$ for $i\in\ZZ/8\ZZ$. Then Lemma~\ref{lem!lp} shows that these give a system of 16 open affine charts in $\sU$ of the form $\mathbb{T}=(\CC^\times)^3\times\Aa_y$, which we can represent at the vertices of the exchange graph $G$ shown in Figure~\ref{fig-OGR510}.
\begin{figure}[htbp]
\begin{center}
\begin{tikzpicture}[scale=2.5]
   \draw[gray,dashed] ({cos(45)/2},{sin(45)/2}) -- ({cos(135)/2},{sin(135)/2}) -- ({cos(225)/2},{sin(225)/2}) -- ({cos(315)/2},{sin(315)/2}) -- cycle;
   \draw[gray,dashed] ({cos(45)/2},{sin(45)/2}) -- ({cos(45)},{sin(45)}) ({cos(135)/2},{sin(135)/2}) -- ({cos(135)},{sin(135)}) ({cos(225)/2},{sin(225)/2}) -- ({cos(225)},{sin(225)}) ({cos(315)/2},{sin(315)/2}) -- ({cos(315)},{sin(315)});
   \draw[thick] (0.5,0) -- (0,0.5) -- (-0.5,0) -- (0,-0.5) -- cycle;
   \draw[thick] ({cos(0)},{sin(0)}) -- ({cos(45)},{sin(45)}) -- ({cos(90)},{sin(90)}) -- ({cos(135)},{sin(135)}) -- ({cos(180)},{sin(180)}) -- ({cos(225)},{sin(225)}) -- ({cos(270)},{sin(270)}) -- ({cos(315)},{sin(315)}) -- cycle;
   \draw[thick] (0.5,0) -- ({cos(0)},{sin(0)}) (0,0.5) -- ({cos(90)},{sin(90)}) (-0.5,0) -- ({cos(180)},{sin(180)}) (0,-0.5) -- ({cos(270)},{sin(270)});
   \node at (0,0) {\small $q_1$};
   \node[fill=white] at (0.5,0.5) {\small $x_3$};
   \node[fill=white] at (0.5,-0.5) {\small $x_5$};
   \node[fill=white] at (-0.5,-0.5) {\small $x_7$};
   \node[fill=white] at (-0.5,0.5) {\small $x_1$};
   \node at (0.15,0.15) {\small \textcolor{gray}{$q_2$}};
   \node at (0.15,0.7) {\small \textcolor{gray}{$x_2$}};
   \node at (0.7,-0.15) {\small \textcolor{gray}{$x_4$}};
   \node at (-0.15,-0.7) {\small \textcolor{gray}{$x_6$}};
   \node at (-0.7,0.15) {\small \textcolor{gray}{$x_8$}};
\end{tikzpicture}
\caption{The exchange graph $G$ for the affine charts covering $\sU$.}
\label{fig-OGR510}
\end{center}
\end{figure}
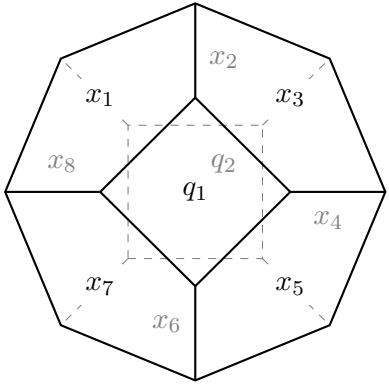
The graph is the 1-skeleton of a 3-dimensional polytope and we can label the faces of this polytope so that the three faces around each vertex give the coordinates on the torus factor of each chart. The edges of $G$ correspond to the three different types of exchange relation:
\begin{align*}
x_{i-1}x_{i+2} &= x_iy_{i+3} + x_{i+1}y_{i-2} + y_iy_{i+1}, \\
x_{i-2}x_{i+2} &= q_i + y_{i-2}y_{i+2}, \\
x_iq_{i-1} &= x_{i-1}x_{i+1}y_{i+4} + x_{i-1}y_{i+1}y_{i+2} + x_{i+1}y_{i-1}y_{i-2} + y_{i-1}y_iy_{i+1}. 
\end{align*}

\begin{rmk}\label{rmk!why-q1-q2}
We can also now see the geometrical reason as to why we are led to include $q_1$ and $q_2$ in our list of cluster variables. In the coordinates on the fibre $U_{\lambda,\mu}$ of $\mathcal U$ described above, and in the case $\lambda,\mu\neq1$, from the blowup up description $U_{\lambda,\mu}=\operatorname{Bl}_Z(\Aa^3_{x_1,x_2,x_3})\setminus \widetilde H$ of Proposition~\ref{prop!blowup} we see that the exchange map 
\[ \left(x_1,x_2,x_3\right)\mapsto \left(x_2,x_3,\frac{\lambda x_2+x_3 + \lambda }{x_1}\right), \]
is a mutation along one of the two line components of $Z$. Similarly, the exchange map 
\[ \left( x_1, x_2, x_3 \right) \mapsto \left( x_1, \frac{\mu x_1x_3 + x_1 + x_3 + 1}{x_2}, x_3 \right), \]
is a mutation along the conic component of $Z$. This transforms the centres along which we obtain our blowup description as follows: now $U_{\lambda,\mu}=\operatorname{Bl}_{Z'}(\Aa^3_{x_1,q_1,x_3})\setminus {\widetilde H}'$ where $H'=\VV(x_1q_1x_3)$ and $Z'\subset H'$ is the union of the following three components.
\begin{center}\begin{tikzpicture}[scale=1.2]
	\begin{scope}
		\draw[gray,dashed] ({sin(120)},{cos(120)}) -- (0,0) -- ({sin(240)},{cos(240)}) (0,0) -- (0,1);
		\draw[thick] (0,1/3) -- ({sin(120)},{1/3+cos(120)});
		\draw[thick,domain=120:240, samples=30] plot ({sin(\x)/2}, {cos(\x)/2} );
		\draw[thick] (0,2/3) -- ({sin(240)},{2/3+cos(240)});
	\end{scope}
	
	\node at (4.5,0.2) {$Z' = \bigcup \begin{cases} \VV(x_1, \lambda q_1 + 1) \\ \VV(q_1, \mu x_1x_3 + x_1 + x_3 + 1 ) \\ \VV(x_3, q_1 + 1) \end{cases}$};
\end{tikzpicture}\end{center}
From this we see that the final type of exchange map
\[ \left( x_1, q_1, x_3 \right) \mapsto \left( x_3, q_1, \frac{\lambda q_1 + 1}{x_1} \right), \]
is a mutation along one of the two line components of $Z'$. The nice surprise is that this system of torus charts and mutations forms a closed system (i.e.\ composing mutations around closed cycles in $G$ gives the identity map on the corresponding torus chart).
\end{rmk}

\begin{rmk}
In the cases where one or both of $\lambda,\mu=1$ the number of centres to blow up in the locus $Z$ increases from three to either four or five respectively. Therefore the valency of the mutation graph also increases. After a fairly involved computation explicitly tracking the centres of $Z$ (which includes having to deal with unexpected jumps similar to those described in \S\ref{sec!mutation-examples}), it can be shown that the exchange graph of $U_{1,\mu}$ (or equivalently $U_{\lambda,1}$) is 4-valent with 28 vertices. Similarly the exchange graph of $U_{1,1}$ is 5-valent with 48 vertices. In particular, these exchange graphs are also finite. The exchange graph for the rank 5 case is a bit too complicated to draw elegantly, however the rank 4 case gives the following graph
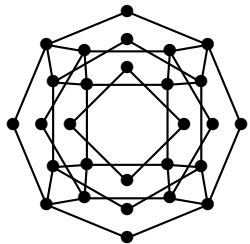
\begin{center}
\begin{tikzpicture}[scale=1.5]
   \begin{scope}[xshift = 2.5cm]
   \draw[thick] ({cos(45)/2},{sin(45)/2}) -- ({cos(135)/2},{sin(135)/2}) -- ({cos(225)/2},{sin(225)/2}) -- ({cos(315)/2},{sin(315)/2}) -- cycle;
   \draw[thick] (0.5,0) -- (0,0.5) -- (-0.5,0) -- (0,-0.5) -- cycle;
   \draw[thick] ({cos(0)},{sin(0)}) -- ({cos(45)},{sin(45)}) -- ({cos(90)},{sin(90)}) -- ({cos(135)},{sin(135)}) -- ({cos(180)},{sin(180)}) -- ({cos(225)},{sin(225)}) -- ({cos(270)},{sin(270)}) -- ({cos(315)},{sin(315)}) -- cycle;
   
  \foreach \i in {0,...,7} {
  	\node at ({cos(45*\i)},{sin(45*\i)}) {$\bullet$};
  	\node at ({cos(45*\i)/2},{sin(45*\i)/2}) {$\bullet$};
  };
  \foreach \i in {0,...,11} {
  	\node at ({0.75*cos(30*\i)},{0.75*sin(30*\i)}) {$\bullet$};
	\draw[thick] ({0.75*cos(30*\i)},{0.75*sin(30*\i)}) -- ({0.75*cos(30*\i+60)},{0.75*sin(30*\i+60)});
  };
  \foreach \i in {0,...,3} {
  	\draw[thick] ({cos(90*\i+45)/2},{sin(90*\i+45)/2}) -- ({0.75*cos(90*\i+30)},{0.75*sin(90*\i+30)}) -- ({cos(90*\i+45)},{sin(90*\i+45)}) -- ({0.75*cos(90*\i+60)},{0.75*sin(90*\i+60)}) -- cycle;
  };
   \end{scope}
\end{tikzpicture}
\end{center}
which is the 1-skeleton of a 4-dimensional polytope. Visually, we see that it can be `collapsed' onto the exchange graph $G$ of Figure~\ref{fig-OGR510}.

\end{rmk}

\subsubsection{The interior divisors inside $U_{\lambda,\mu}$}

We consider the following ten \emph{interior divisors} of $U_{\lambda,\mu}$ 
\[ E_i=\VV(x_i) \text{ for } i=1,\ldots,8 \qquad \text{and} \qquad F_i=\VV(q_i) \text{ for } i=1,2. \]
These are all the divisors that appear in the complement of one of the 16 cluster tori and are the analogue of the five interior $(-1)$-curves of the del Pezzo surface in \S\ref{sect!dP5}. Because of the jump in the rank of the class group $\Cl(U_{\lambda,\mu})$, something special happens to these divisors when either or both of $\lambda,\mu=1$. 

\begin{prop} \label{prop!div} \leavevmode
\begin{enumerate}
\item If $\lambda,\mu\neq1$, then the divisors $E_i$ and $F_i$ are all irreducible.
\item If $\mu=1$, the four `odd' divisors $E_1,E_3,E_5,E_7$ break up into four overlapping irreducible components $E_{13},E_{35},E_{57},E_{71}$, where $E_{13} = \VV(x_1,x_3)$ etc. Moreover $F_2$ breaks into two components as $F_2=F_{26}\cup F_{48}$, where $F_{26} = \VV(1+x_2,1+x_6)$ and $F_{48} = \VV(1+x_4,1+x_8)$.
\item If $\lambda=1$, the four `even' divisors $E_2,E_4,E_6,E_8$ and $F_1$ all break into two components, with the analogous description to (2).
\end{enumerate}
In particular, if $\lambda=\mu=1$ then all ten divisors break up into a total of twelve components.
\end{prop}

\begin{proof}
We consider the image of these divisors under the projection $g\colon U_{\lambda,\mu}\to \Aa^3_{x_1,x_2,x_3}$ with the discriminant locus $Z\subset \Aa^3$ as in Proposition~\ref{prop!blowup}. For $\lambda,\mu\neq1$ we find that $E_1,E_2,E_3$ are the exceptional divisors over the three components of $Z$ and the remaining divisors are given by the strict transform under $g^{-1}$ of the following (irreducible) divisors in $\Aa^3$.
\begin{center}
$\begin{array}{lll}
E_1: x_1 = \lambda + x_2 + \lambda x_3 = 0 && E_6: (1+x_1+x_2)(\lambda+x_2+\lambda x_3) = \lambda(1-\mu)x_1x_3 \\
E_2: x_2 = 1 + x_1 + x_3 + \mu x_1x_3 = 0 && E_7: 1 + x_1 + x_2 + x_3 + \mu x_1x_3 = 0 \\
E_3: x_3 = 1 + x_1 + x_2 = 0 && E_8: 1 + x_1 + x_2 = 0 \\
E_4: \lambda + x_2 + \lambda x_3 = 0 && F_1: 1 + x_1 + x_3 + \mu x_1x_3 = 0  \\
E_5: \lambda + \lambda x_1 + x_2 + \lambda x_3 + \lambda\mu x_3 = 0 && F_2: (\lambda + x_2 + \lambda x_3)(1 + x_1 + x_2) = \lambda x_1x_3 \end{array}$
\end{center}
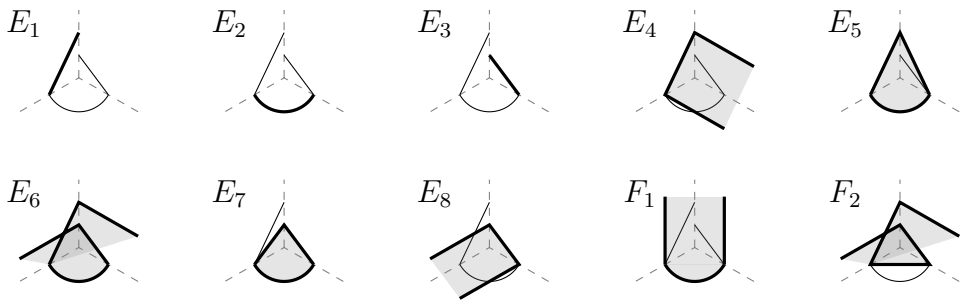
\begin{center}\begin{tikzpicture}[scale=0.9]
	\begin{scope}
		\node at (-0.8,0.8) {$E_1$};
		\draw[gray,dashed] ({sin(120)},{cos(120)}) -- (0,0) -- ({sin(240)},{cos(240)}) (0,0) -- (0,1);
		\draw (0,1/3) -- ({sin(120)/2},{cos(120)/2});
		\draw[domain=120:240, samples=30] plot ({sin(\x)/2}, {cos(\x)/2} );
		\draw[very thick] (0,2/3) -- ({sin(240)/2},{cos(240)/2});
	\end{scope}
	\begin{scope}[xshift=3cm]
		\node at (-0.8,0.8) {$E_2$};
		\draw[gray,dashed] ({sin(120)},{cos(120)}) -- (0,0) -- ({sin(240)},{cos(240)}) (0,0) -- (0,1);
		\draw (0,1/3) -- ({sin(120)/2},{cos(120)/2});
		\draw (0,2/3) -- ({sin(240)/2},{cos(240)/2});
		\draw[very thick, domain=120:240, samples=30] plot ({sin(\x)/2}, {cos(\x)/2} );
	\end{scope}
	\begin{scope}[xshift=6cm]
		\node at (-0.8,0.8) {$E_3$};
		\draw[gray,dashed] ({sin(120)},{cos(120)}) -- (0,0) -- ({sin(240)},{cos(240)}) (0,0) -- (0,1);
		\draw (0,2/3) -- ({sin(240)/2},{cos(240)/2});
		\draw[domain=120:240, samples=30] plot ({sin(\x)/2}, {cos(\x)/2} );
		\draw[very thick] (0,1/3) -- ({sin(120)/2},{cos(120)/2});
	\end{scope}
	\begin{scope}[xshift=9cm]
		\node at (-0.8,0.8) {$E_4$};
		\draw[gray,dashed] ({sin(120)},{cos(120)}) -- (0,0) -- ({sin(240)},{cos(240)}) (0,0) -- (0,1);
		\draw (0,1/3) -- ({sin(120)/2},{cos(120)/2});
		\draw[domain=120:240, samples=30] plot ({sin(\x)/2}, {cos(\x)/2} );
		\draw[very thick, fill = gray, fill opacity = 0.2] ({sin(120)},{2/3+cos(120)}) -- (0,2/3) -- ({sin(240)/2},{cos(240)/2}) -- ({sin(120)+sin(240)/2},{cos(120)+cos(240)/2});
	\end{scope}
	
	\begin{scope}[xshift=12cm]
		\node at (-0.8,0.8) {$E_5$};
		\draw[gray,dashed] ({sin(120)},{cos(120)}) -- (0,0) -- ({sin(240)},{cos(240)}) (0,0) -- (0,1);
		\draw (0,1/3) -- ({sin(120)/2},{cos(120)/2});
		\draw[very thick, domain=120:240, samples=30, fill = gray, fill opacity = 0.2] plot ({sin(\x)/2}, {cos(\x)/2} );
		\fill[gray, opacity = 0.2] ({sin(240)/2},{cos(240)/2}) -- (0,2/3) -- ({sin(120)/2},{cos(120)/2}) -- cycle;
		\draw[very thick] ({sin(120)/2},{cos(120)/2}) -- (0,2/3) -- ({sin(240)/2},{cos(240)/2});
	\end{scope}
	\begin{scope}[xshift=0cm,yshift=-2.5cm]
		\node at (-0.8,0.8) {$E_6$};
		\draw[gray,dashed] ({sin(120)},{cos(120)}) -- (0,0) -- ({sin(240)},{cos(240)}) (0,0) -- (0,1);
		\draw[very thick, domain=120:240, samples=30, fill = gray, fill opacity = 0.2] plot ({sin(\x)/2}, {cos(\x)/2} );
		\fill[gray, opacity = 0.2] ({sin(120)},{2/3+cos(120)}) -- (0,2/3) -- ({sin(240)/2},{cos(240)/2}) -- cycle;
		\fill[gray, opacity = 0.2] ({sin(240)},{1/3+cos(240)}) -- (0,1/3) -- ({sin(120)/2},{cos(120)/2}) -- ({sin(240)/2},{cos(240)/2}) -- cycle;
		\draw[very thick] ({sin(120)},{2/3+cos(120)}) -- (0,2/3) -- ({sin(240)/2},{cos(240)/2})  ({sin(120)/2},{cos(120)/2}) -- (0,1/3) -- ({sin(240)},{1/3+cos(240)});
	\end{scope}
	\begin{scope}[xshift=3cm,yshift=-2.5cm]
		\node at (-0.8,0.8) {$E_7$};
		\draw[gray,dashed] ({sin(120)},{cos(120)}) -- (0,0) -- ({sin(240)},{cos(240)}) (0,0) -- (0,1);
		\draw (0,2/3) -- ({sin(240)/2},{cos(240)/2});
		\draw[very thick, domain=120:240, samples=30, fill = gray, fill opacity = 0.2] plot ({sin(\x)/2}, {cos(\x)/2} );
		\fill[gray, opacity = 0.2] ({sin(240)/2},{cos(240)/2}) -- (0,1/3) -- ({sin(120)/2},{cos(120)/2}) -- cycle;
		\draw[very thick] ({sin(120)/2},{cos(120)/2}) -- (0,1/3) -- ({sin(240)/2},{cos(240)/2});
	\end{scope}
	\begin{scope}[xshift=6cm,yshift=-2.5cm]
		\node at (-0.8,0.8) {$E_8$};
		\draw[gray,dashed] ({sin(120)},{cos(120)}) -- (0,0) -- ({sin(240)},{cos(240)}) (0,0) -- (0,1);
		\draw (0,2/3) -- ({sin(240)/2},{cos(240)/2});
		\draw[domain=120:240, samples=30] plot ({sin(\x)/2}, {cos(\x)/2} );
		\draw[very thick, fill = gray, fill opacity = 0.2] ({sin(240)},{1/3+cos(240)}) -- (0,1/3) -- ({sin(120)/2},{cos(120)/2}) -- ({sin(240)+sin(120)/2},{cos(240)+cos(120)/2});
	\end{scope}

	\begin{scope}[xshift=9cm,yshift=-2.5cm]
		\node at (-0.8,0.8) {$F_1$};
		\draw[gray,dashed] ({sin(120)},{cos(120)}) -- (0,0) -- ({sin(240)},{cos(240)}) (0,0) -- (0,1);
		\draw (0,2/3) -- ({sin(240)/2},{cos(240)/2});
		\draw (0,1/3) -- ({sin(120)/2},{cos(120)/2});
		\draw[very thick, domain=120:240, samples=30, fill = gray, fill opacity = 0.2] plot ({sin(\x)/2}, {cos(\x)/2} );
		\fill[gray, opacity = 0.2] ({sin(120)/2},{cos(120)/2+1}) -- ({sin(120)/2},{cos(120)/2}) -- ({sin(240)/2},{cos(240)/2}) -- ({sin(240)/2},{cos(240)/2+1}) -- cycle;
		\draw[very thick] ({sin(120)/2},{cos(120)/2+1}) -- ({sin(120)/2},{cos(120)/2})  ({sin(240)/2},{cos(240)/2}) -- ({sin(240)/2},{cos(240)/2+1});
	\end{scope}
	\begin{scope}[xshift=12cm,yshift=-2.5cm]
		\node at (-0.8,0.8) {$F_2$};
		\draw[gray,dashed] ({sin(120)},{cos(120)}) -- (0,0) -- ({sin(240)},{cos(240)}) (0,0) -- (0,1);
		\draw[domain=120:240, samples=30] plot ({sin(\x)/2}, {cos(\x)/2} );
		\fill[gray, opacity = 0.2] ({sin(120)},{2/3+cos(120)}) -- (0,2/3) -- ({sin(240)/2},{cos(240)/2}) -- cycle;
		\fill[gray, opacity = 0.2] ({sin(240)},{1/3+cos(240)}) -- (0,1/3) -- ({sin(120)/2},{cos(120)/2}) -- ({sin(240)/2},{cos(240)/2}) -- cycle;
		\draw[very thick] ({sin(120)},{2/3+cos(120)}) -- (0,2/3) -- ({sin(240)/2},{cos(240)/2}) -- ({sin(120)/2},{cos(120)/2}) -- (0,1/3) -- ({sin(240)},{1/3+cos(240)});
	\end{scope}
\end{tikzpicture}\end{center}

If $\mu=1$ the conic component of $Z$ breaks into two pieces. This has the effect of breaking all the divisors $E_2,E_4,E_6,E_8$ and $F_1$ in two. If $\lambda=1$ then the two line components of $Z$ meet at an embedded point. This has the effect of breaking all the divisors $E_1,E_3,E_5,E_7$ and $F_2$ in two. In the most extreme case when $\lambda=\mu=1$, the ten divisors break up into twelve components as follows.
\begin{center}
\resizebox{\textwidth}{!}{
$\begin{array}{lllll} 
E_{71}: x_1 = 1 + x_2 + x_3 = 0 && E_{35}: x_3 = 1 + x_1 + x_2 = 0 &&  F_{15}: 1 + x_1 = 0 \\
E_{82}: x_2 = 1 + x_1 = 0 && E_{46}: 1 + x_2 + x_3 = 0 && F_{26}: 1 + x_2 = 0  \\
E_{13}: x_1 = 1 + x_2 = x_3 = 0 && E_{57}: 1 +x_1+x_2+x_3+x_1x_3 = 0 && F_{37}: 1+x_3 = 0 \\
E_{24}: x_2 = 1 + x_3 = 0 && E_{68}: 1 + x_1 + x_2 = 0 && F_{48}: 1+x_1+x_2+x_3 = 0 
\end{array}$}
\end{center}
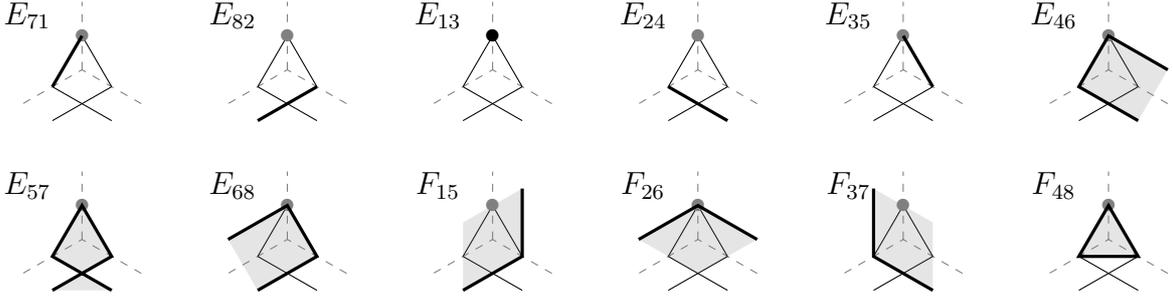
\begin{center}\begin{tikzpicture}[scale=0.9]

	\begin{scope}
		\node at (-0.8,0.8) {$E_{71}$};
		\draw[gray,dashed] ({sin(120)},{cos(120)}) -- (0,0) -- ({sin(240)},{cos(240)}) (0,0) -- (0,1);
		\draw (0,1/2) -- ({sin(120)/2},{cos(120)/2});
		\draw ({sin(120)/2},{cos(120)/2}) -- ({sin(120)/2+sin(240)},{cos(120)/2+cos(240)});
		\draw ({sin(240)/2},{cos(240)/2}) -- ({sin(120)+sin(240)/2},{cos(120)+cos(240)/2});
		\node[gray] at (0,1/2) {$\bullet$};
		\draw[very thick] (0,1/2) -- ({sin(240)/2},{cos(240)/2});
	\end{scope}
	\begin{scope}[xshift=3cm]
		\node at (-0.8,0.8) {$E_{82}$};
		\draw[gray,dashed] ({sin(120)},{cos(120)}) -- (0,0) -- ({sin(240)},{cos(240)}) (0,0) -- (0,1);
		\draw (0,1/2) -- ({sin(120)/2},{cos(120)/2});
		\draw ({sin(240)/2},{cos(240)/2}) -- ({sin(120)+sin(240)/2},{cos(120)+cos(240)/2});
		\draw (0,1/2) -- ({sin(240)/2},{cos(240)/2});
		\draw[very thick] ({sin(120)/2},{cos(120)/2}) -- ({sin(120)/2+sin(240)},{cos(120)/2+cos(240)});
		\node[gray] at (0,1/2) {$\bullet$};
	\end{scope}
	\begin{scope}[xshift=6cm]
		\node at (-0.8,0.8) {$E_{13}$};
		\draw[gray,dashed] ({sin(120)},{cos(120)}) -- (0,0) -- ({sin(240)},{cos(240)}) (0,0) -- (0,1);
		\draw (0,1/2) -- ({sin(120)/2},{cos(120)/2});
		\draw ({sin(120)/2},{cos(120)/2}) -- ({sin(120)/2+sin(240)},{cos(120)/2+cos(240)});
		\draw ({sin(240)/2},{cos(240)/2}) -- ({sin(120)+sin(240)/2},{cos(120)+cos(240)/2});
		\draw (0,1/2) -- ({sin(240)/2},{cos(240)/2});
		\node at (0,1/2) {$\bullet$};
	\end{scope}
	\begin{scope}[xshift=9cm]
		\node at (-0.8,0.8) {$E_{24}$};
		\draw[gray,dashed] ({sin(120)},{cos(120)}) -- (0,0) -- ({sin(240)},{cos(240)}) (0,0) -- (0,1);
		\draw (0,1/2) -- ({sin(120)/2},{cos(120)/2});
		\draw ({sin(120)/2},{cos(120)/2}) -- ({sin(120)/2+sin(240)},{cos(120)/2+cos(240)});
		\draw (0,1/2) -- ({sin(240)/2},{cos(240)/2});
		\draw[very thick] ({sin(240)/2},{cos(240)/2}) -- ({sin(120)+sin(240)/2},{cos(120)+cos(240)/2});
		\node[gray] at (0,1/2) {$\bullet$};
	\end{scope}
	\begin{scope}[xshift=12cm]
		\node at (-0.8,0.8) {$E_{35}$};
		\draw[gray,dashed] ({sin(120)},{cos(120)}) -- (0,0) -- ({sin(240)},{cos(240)}) (0,0) -- (0,1);
		\draw (0,1/2) -- ({sin(240)/2},{cos(240)/2});
		\draw ({sin(120)/2},{cos(120)/2}) -- ({sin(120)/2+sin(240)},{cos(120)/2+cos(240)});
		\draw ({sin(240)/2},{cos(240)/2}) -- ({sin(120)+sin(240)/2},{cos(120)+cos(240)/2});
		\node[gray] at (0,1/2) {$\bullet$};
		\draw[very thick] (0,1/2) -- ({sin(120)/2},{cos(120)/2});
	\end{scope}
	\begin{scope}[xshift=15cm]
		\node at (-0.8,0.8) {$E_{46}$};
		\draw[gray,dashed] ({sin(120)},{cos(120)}) -- (0,0) -- ({sin(240)},{cos(240)}) (0,0) -- (0,1);
		\draw (0,1/2) -- ({sin(120)/2},{cos(120)/2});
		\draw ({sin(120)/2},{cos(120)/2}) -- ({sin(120)/2+sin(240)},{cos(120)/2+cos(240)});
		\node[gray] at (0,1/2) {$\bullet$};
		\draw[very thick, fill = gray, fill opacity = 0.2] ({sin(120)},{cos(120)+1/2}) -- (0,1/2) -- ({sin(240)/2},{cos(240)/2}) -- ({sin(120)+sin(240)/2},{cos(120)+cos(240)/2});
	\end{scope}

	\begin{scope}[yshift=-2.5cm]
		\node at (-0.8,0.8) {$E_{57}$};
		\draw[gray,dashed] ({sin(120)},{cos(120)}) -- (0,0) -- ({sin(240)},{cos(240)}) (0,0) -- (0,1);
		\node[gray] at (0,1/2) {$\bullet$};
		\draw[very thick, fill = gray, fill opacity = 0.2] ({sin(120)+sin(240)/2},{cos(120)+cos(240)/2}) -- ({sin(240)/2},{cos(240)/2}) -- (0,1/2) -- ({sin(120)/2},{cos(120)/2}) -- ({sin(120)/2+sin(240)},{cos(120)/2+cos(240)});
	\end{scope}
	\begin{scope}[xshift=3cm,yshift=-2.5cm]
		\node at (-0.8,0.8) {$E_{68}$};
		\draw[gray,dashed] ({sin(120)},{cos(120)}) -- (0,0) -- ({sin(240)},{cos(240)}) (0,0) -- (0,1);
		\draw ({sin(240)/2},{cos(240)/2}) -- ({sin(120)+sin(240)/2},{cos(120)+cos(240)/2});
		\draw (0,1/2) -- ({sin(240)/2},{cos(240)/2});
		\node[gray] at (0,1/2) {$\bullet$};
		\draw[very thick, fill = gray, fill opacity = 0.2] ({sin(240)},{1/2+cos(240)}) -- (0,1/2) -- ({sin(120)/2},{cos(120)/2}) -- ({sin(120)/2+sin(240)},{cos(120)/2+cos(240)});
	\end{scope}
	\begin{scope}[xshift=6cm,yshift=-2.5cm]
		\node at (-0.8,0.8) {$F_{15}$};
		\draw[gray,dashed] ({sin(120)},{cos(120)}) -- (0,0) -- ({sin(240)},{cos(240)}) (0,0) -- (0,1);
		\draw (0,1/2) -- ({sin(120)/2},{cos(120)/2});
		\draw ({sin(240)/2},{cos(240)/2}) -- ({sin(120)+sin(240)/2},{cos(120)+cos(240)/2});
		\draw (0,1/2) -- ({sin(240)/2},{cos(240)/2});
		\fill[gray, opacity = 0.2] ({sin(120)/2},{1+cos(120)/2}) -- ({sin(120)/2},{cos(120)/2}) -- ({sin(120)/2+sin(240)},{cos(120)/2+cos(240)}) -- ({sin(120)/2+sin(240)},{1+cos(120)/2+cos(240)}) -- cycle;
		\draw[very thick] ({sin(120)/2},{1+cos(120)/2}) -- ({sin(120)/2},{cos(120)/2}) -- ({sin(120)/2+sin(240)},{cos(120)/2+cos(240)});
		\node[gray] at (0,1/2) {$\bullet$};
	\end{scope}
	\begin{scope}[xshift=9cm,yshift=-2.5cm]
		\node at (-0.8,0.8) {$F_{26}$};
		\draw[gray,dashed] ({sin(120)},{cos(120)}) -- (0,0) -- ({sin(240)},{cos(240)}) (0,0) -- (0,1);
		\draw (0,1/2) -- ({sin(120)/2},{cos(120)/2});
		\draw ({sin(120)/2},{cos(120)/2}) -- ({sin(120)/2+sin(240)},{cos(120)/2+cos(240)});
		\draw ({sin(240)/2},{cos(240)/2}) -- ({sin(120)+sin(240)/2},{cos(120)+cos(240)/2});
		\draw (0,1/2) -- ({sin(240)/2},{cos(240)/2});
		\fill[gray, opacity = 0.2] ({sin(240)},{cos(240)+1/2}) -- (0,1/2) -- ({sin(120)},{cos(120)+1/2}) -- (0,-1/2) -- cycle;
		\node[gray] at (0,1/2) {$\bullet$};
		\draw[very thick] ({sin(240)},{cos(240)+1/2}) -- (0,1/2) -- ({sin(120)},{cos(120)+1/2});
	\end{scope}
	\begin{scope}[xshift=12cm,yshift=-2.5cm]
		\node at (-0.8,0.8) {$F_{37}$};
		\draw[gray,dashed] ({sin(120)},{cos(120)}) -- (0,0) -- ({sin(240)},{cos(240)}) (0,0) -- (0,1);
		\draw (0,1/2) -- ({sin(120)/2},{cos(120)/2});
		\draw ({sin(120)/2},{cos(120)/2}) -- ({sin(120)/2+sin(240)},{cos(120)/2+cos(240)});
		\draw (0,1/2) -- ({sin(240)/2},{cos(240)/2});
		\fill[gray, opacity = 0.2] ({sin(240)/2},{1+cos(240)/2}) -- ({sin(240)/2},{cos(240)/2}) -- ({sin(240)/2+sin(120)},{cos(240)/2+cos(120)}) -- ({sin(240)/2+sin(120)},{1+cos(240)/2+cos(120)}) -- cycle;
		\draw[very thick] ({sin(240)/2},{1+cos(240)/2}) -- ({sin(240)/2},{cos(240)/2}) -- ({sin(120)+sin(240)/2},{cos(120)+cos(240)/2});
		\node[gray] at (0,1/2) {$\bullet$};
	\end{scope}
	\begin{scope}[xshift=15cm,yshift=-2.5cm]
		\node at (-0.8,0.8) {$F_{48}$};
		\draw[gray,dashed] ({sin(120)},{cos(120)}) -- (0,0) -- ({sin(240)},{cos(240)}) (0,0) -- (0,1);
		\draw ({sin(120)/2},{cos(120)/2}) -- ({sin(120)/2+sin(240)},{cos(120)/2+cos(240)});
		\draw ({sin(240)/2},{cos(240)/2}) -- ({sin(120)+sin(240)/2},{cos(120)+cos(240)/2});
		\node[gray] at (0,1/2) {$\bullet$};
		\draw[very thick, fill = gray, fill opacity = 0.2] (0,1/2) -- ({sin(240)/2},{cos(240)/2}) -- ({sin(120)/2},{cos(120)/2}) -- cycle;
	\end{scope}
\end{tikzpicture}\end{center}
\end{proof}

\subsection{The tropicalisation of $U_{\lambda,\mu}$}

To simplify the notation in this section we let $U=U_{\lambda,\mu}$ so that the dependence on $\lambda,\mu$ is now left implicit. We now want to construct the tropicalisation $N_U$ of $U$ (as in \S\ref{sec!trop}), the dual space $M_U$ and the dual intersection pairing $\langle {\cdot},{\cdot}\rangle\colon N_U\times M_U\to \RR$.

\subsubsection{The intersection pairing}

Recall that we have a set of ten cluster variables $\{x_1,\ldots,x_8,q_1,q_2\}$ and a set of ten boundary components $\{D_{123},\ldots,D_{812},D_{1357},D_{2468}\}$ in the boundary divisor $D$ of our projective compactification $(X,D)$ of $U$. For any cluster monomial $\vartheta_m\in \CC[U]$ and any boundary component $D_n$ we let $\langle D_n,\vartheta_m \rangle := \ord_{D_n}(\vartheta_m)$ denote the order of vanishing of $\vartheta_m$ along $D_n$. We start by computing the analogous matrix to Table~\ref{table!dP5} that we will use to define the intersection pairing $\langle {\cdot},{\cdot} \rangle\colon N_U\times M_U\to \RR$. Since this matrix is symmetric, this realises a one-to-one correspondence between the ten cluster variables $x_1,\ldots,x_8,q_1,q_2$ in $\CC[U]$ and the ten boundary divisors $D_{123},\ldots, D_{812},D_{1357},D_{2468}$ of $(X,D)$.

\begin{prop}
Considered as rational functions on $X$, the cluster variables $x_1$ and $q_1$ have divisor
\begin{align*}
\divv x_1 &= E_1 + D_{456} - D_{781} - D_{812} - D_{123} - D_{1357} \\
\divv q_1 &= F_1 + D_{2468} - D_{123} - D_{345} - D_{567} - D_{781} - 2D_{1357} 
\end{align*}
and similarly for the other cluster variables, up to the $\Dih_8$ action. In particular, each cluster variable vanishes along a unique boundary component of $(X,D)$, giving a one-to-one correspondence between the ten cluster variables of $U$ and the ten components of $D$. With respect to this ordering, the pairing between these cluster variables and boundary components is represented by the symmetric $10\times 10$ matrix given in Table~\ref{table!V12}. 
\end{prop}

\begin{proof}
In terms of the homogeneous equations that define $X$ we are required to compute $\divv(x_0^{-1}x_1)$ and $\divv(x_0^{-2}q_1)$, where $x_0\in|\sO_X(1)|$ is the homogenising variable. Since $\divv x_0 = D = D_{123} + \ldots + D_{812} + D_{1357} + D_{2468}$, the first formula follows from showing that, considered as a section $x_1\in|\sO_X(1)|$, we have
\[ \divv x_1 = E_1 + D_{234} + D_{345} + 2D_{456}  + D_{567}  + D_{678} + D_{2468} \]
and the second formula follows similarly for $q_1\in |\sO_X(2)|$. This now follows from an explicit calculation with the equations of $X$.
\end{proof}

\begin{table}[htp]
\caption{The intersection numbers $\langle D_i,x_j \rangle$ for $U$.}
\begin{center}
\renewcommand{\arraystretch}{1.2}
\begin{tabular}{|c|cccccccccc|} \hline
 & $D_{456}$ & $D_{567}$ & $D_{678}$ & $D_{781}$ & $D_{812}$ & $D_{123}$ & $D_{234}$ & $D_{345}$ & $D_{2468}$ & $D_{1357}$ \\ \hline
$x_1$ & $ 1$ & $ 0$ & $ 0$ & $-1$ & $-1$ & $-1$ & $ 0$ & $ 0$ & $ 0$ & $-1$ \\
$x_2$ & $ 0$ & $ 1$ & $ 0$ & $ 0$ & $-1$ & $-1$ & $-1$ & $ 0$ & $-1$ & $ 0$ \\
$x_3$ & $ 0$ & $ 0$ & $ 1$ & $ 0$ & $ 0$ & $-1$ & $-1$ & $-1$ & $ 0$ & $-1$ \\
$x_4$ & $-1$ & $ 0$ & $ 0$ & $ 1$ & $ 0$ & $ 0$ & $-1$ & $-1$ & $-1$ & $ 0$ \\
$x_5$ & $-1$ & $-1$ & $ 0$ & $ 0$ & $ 1$ & $ 0$ & $ 0$ & $-1$ & $ 0$ & $-1$ \\
$x_6$ & $-1$ & $-1$ & $-1$ & $ 0$ & $ 0$ & $ 1$ & $ 0$ & $ 0$ & $-1$ & $ 0$ \\
$x_7$ & $ 0$ & $-1$ & $-1$ & $-1$ & $ 0$ & $ 0$ & $ 1$ & $ 0$ & $ 0$ & $-1$ \\
$x_8$ & $ 0$ & $ 0$ & $-1$ & $-1$ & $-1$ & $ 0$ & $ 0$ & $ 1$ & $-1$ & $ 0$ \\
$q_1$ & $ 0$ & $-1$ & $ 0$ & $-1$ & $ 0$ & $-1$ & $ 0$ & $-1$ & $ 1$ & $-2$ \\
$q_2$ & $-1$ & $ 0$ & $-1$ & $ 0$ & $-1$ & $ 0$ & $-1$ & $ 0$ & $-2$ & $ 1$ \\ \hline
\end{tabular}
\end{center}
\label{table!V12}
\end{table}%

Given the correspondence in Table~\ref{table!V12} it is now convenient to rename the boundary divisors as $D_1,\ldots, D_8,D_{q_1},D_{q_2}$ so that $D_i = D_{i+3,i+4,i+5}$ corresponds to $x_i$ for $i\in \ZZ/8\ZZ$ and $D_{q_1}=D_{2468}$ and $D_{q_2}=D_{1357}$ correspond to $q_1$ and $q_2$. We will us Table~\ref{table!V12} to define a complete fan $\mathcal F$ in $\RR^3$ which will give us both a toric model and a scattering diagram for $U$.

\subsubsection{A toric model for $U$}

\paragraph{The fan $\mathcal F$.}
In order to tropicalise $U$, we must first start by choosing a cluster torus chart for $U$. We consider the cluster torus $j\colon \TT_{1q3}\hookrightarrow U$ with coordinates $x_1,q_1,x_3$, determined by the seed $Z\subset \Aa^3$ described in Remark~\ref{rmk!why-q1-q2}.\footnote{We could have chosen to consider the `standard' torus chart $\TT_{123}$ given to us by the Lyness map, but we prefer $\TT_{1q3}$ since (unlike the case in Proposition~\ref{prop!blowup}) the centres $Z\subset \Aa^3$ that we are required to blow up are already disjoint, without having to first blow up any of the 1-dimensional strata of $\Aa^3$.} Now consider the ten primitive integral vectors $v_1,\ldots,v_8,w_1,w_2\in\ZZ^3$ corresponding to the following columns 
\[ \begin{array}{cccccccccc}
v_1 & v_2 & v_3 & v_4 & v_5 & v_6 & v_7 & v_8 & w_1 & w_2 \\ \hline 
-1 & 0 & 0 & 1 & 1 & 1 & 0 & 0 & 0 & 1 \\
 0 & 1 & 0 & 1 & 0 & 1 & 0 & 1 &-1 & 2 \\
 0 & 0 &-1 & 0 & 0 & 1 & 1 & 1 & 0 & 1
\end{array} \]
read off as the negative of the rows corresponding to $x_1,q_1,x_3$ in Table~\ref{table!V12}. We can define a complete fan $\mathcal F$ in $\RR^3$ which has rays generated by these ten vectors and whose cones are dual to the exchange graph in Figure~\ref{fig-OGR510}. Thus $\mathcal F$ has 16 3-dimensional cones corresponding to the 16 cluster torus charts of $U$ and 24 2-dimensional cones corresponding to the 24 possible mutations between cluster torus charts. The fan $\mathcal F$ is displayed in Figure~\ref{fig!3d-fan}.
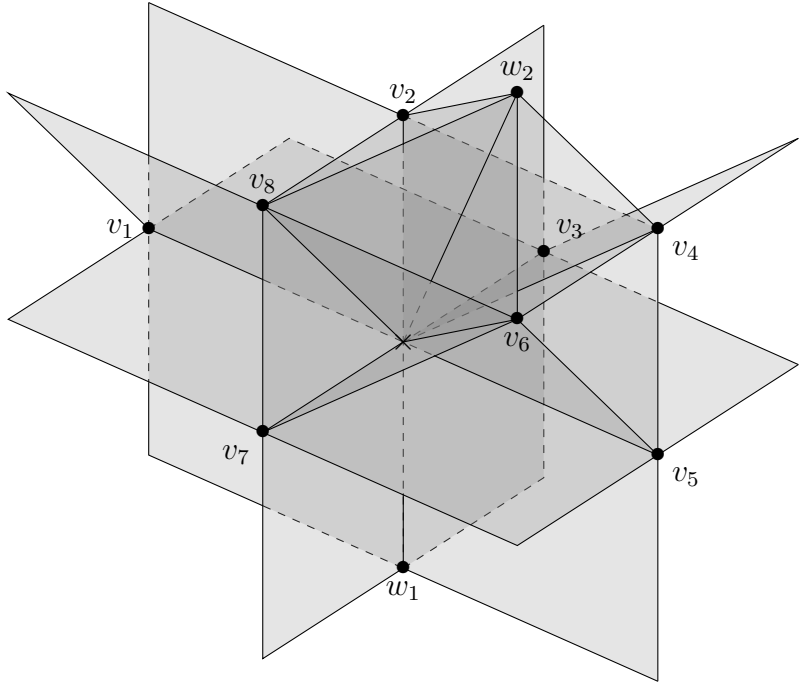
\begin{figure}[htbp]
\begin{center}
\begin{tikzpicture}[scale=3]
   \node[inner sep = -2pt] (o) at (0,0) {};
   \node (x) at ({cos(330)+0.25},{sin(330)}) {};
   \node (y) at ({cos(90)},{sin(90)}) {};
   \node (z) at ({cos(210)+0.25},{sin(210)+0.1}) {};
   \draw[dashed] ($(o)-(y)-({0.55*(cos(330)+0.25)},{0.55*sin(330)})$) -- ($(o)-(y)$);
   \draw[dashed] ($(o)-(y)$) -- ($(o)$) -- ($(o)-(z)$) -- ($({0.32*(cos(330)+0.25)},{0.32*(sin(330)+1)})-(z)$);
   \draw[dashed] ($(o)-(x)$) -- ($(o)-(x)-(z)$) -- ($({0.45*(cos(330)+0.25)},{0.45*(sin(330))})-(z)$);
   \draw[dashed] (0,0.87) -- (o);
   \draw[dashed] ({-0.55*(cos(330)+0.25)},{-0.55*sin(330)}) -- ({0.12*(cos(330)+0.25)},{0.12*sin(330)});
   \draw[dashed] ($(0,0.32)-(x)$) -- ($(o)-(0,0.68)-(x)$);
   \draw[dashed] ($(o)$) -- ({0.49*(cos(330)/2+cos(210)/2+0.25)},{0.49*(sin(330)/2+sin(210)/2+1.05)});
   \draw[dashed] ($(y)$) -- ($(x)+(y)$) ($(o)$) -- ($({0.45*(cos(330)+0.25)},{0.45*(sin(330)+1)})$) ($(o)-(y)$) -- ($(o)-(y)-(z)$) -- ({-cos(210)-0.25},{-sin(210+0.1)+0.22});
   \fill[gray,opacity=0.2] ($(o)-(x)-(z)$) -- ($(o)-(x)+(z)$) -- ($(o)+(x)+(z)$) -- ($(o)+(x)-(z)$) -- cycle;
   \fill[gray,opacity=0.2] ($(o)+(x)$) -- ($(x)+(y)+(z)$) -- ($(o)-(x)+(y)+(z)$) -- ($(o)-(x)$) -- cycle;
   \fill[gray,opacity=0.2] ($(o)+(z)$) -- ($(x)+(y)+(z)$) -- ($(o)-(z)+(y)+(x)$) -- ($(o)-(z)$) -- cycle;
   \fill[gray,opacity=0.2] ($(o)-(y)+(z)$) -- ($(o)+(y)+(z)$) -- ($(o)+(y)-(z)$) -- ($(o)-(y)-(z)$) -- cycle;
   \fill[gray,opacity=0.2] ($(o)-(y)+(x)$) -- ($(o)+(y)+(x)$) -- ($(o)+(y)-(x)$) -- ($(o)-(y)-(x)$) -- cycle;
   \fill[gray,opacity=0.2] ($(o)$) -- ($(x)+(y)$) -- ($(x)+(y)+(y)+(z)$) -- ($(y)+(z)$) -- cycle;
   \fill[gray,opacity=0.2] ($(o)$) -- ($(x)+(y)+(z)$) -- ($(x)+(y)+(y)+(z)$) -- ($(y)$) -- cycle;
   \draw ($(o)+(x)+(y)-(z)$) -- ($(o)+(x)+(y)+(z)$) -- ($(o)-(x)+(y)+(z)$) -- ($(o)-(x)$) -- ($(o)-(x)+(z)$) -- ($(o)+(x)+(z)$) -- ($(o)+(x)-(z)$);
   \draw ($(o)-(x)+(y)$) -- ($(y)$) ($(o)+(x)+(y)$) -- ($(o)+(x)-(y)$) ($(o)-(z)+(y)$) -- ($(o)+(z)+(y)$) -- ($(o)+(z)-(y)$) -- ($(o)-(y)$) ($(o)+(x)+(z)+(y)$) -- ($(o)+(z)$) -- ($(o)$) -- ($(o)+(x)+(y)+(z)$) -- ($(x)$);
   \draw ($(o)+(x)+(y)$) -- ($(o)+(x)+(y)+(y)+(z)$) -- ($(o)+(z)+(y)$) -- ($(o)$);
   \draw ($(o)+(x)-(y)$) -- ($(o)-(y)$) -- ({0},{-0.67});
   \draw ($(o)+(x)$) -- ({0.12*(cos(330)+0.25)},{0.12*sin(330)});
   \draw ($(o)+(y)$) -- (0,0.87);
   \draw ($(o)+(y)-(x)$) -- ($(0,0.32)-(x)$);
   \draw ($(o)+(y)-(z)$) -- ({-cos(210)-0.25},{-sin(210+0.1)+0.22});
   \draw ($(o)+(x)+(y)-(z)$) -- ($({0.32*(cos(330)+0.25)},{0.32*(sin(330)+1)})-(z)$);
   \draw ($(o)+(x)-(z)$) -- ($({0.45*(cos(330)+0.25)},{0.45*(sin(330))})-(z)$);
   \draw ($(o)+(x)+(y)$) -- ($({0.45*(cos(330)+0.25)},{0.45*(sin(330)+1)})$);
   \draw ($(o)-(x)$) -- ({-0.55*(cos(330)+0.25)},{-0.55*sin(330)});
   \draw ($(o)-(0,0.68)-(x)$) -- ($(o)-(y)-(x)$) -- ($(o)-(y)-({0.55*(cos(330)+0.25)},{0.55*sin(330)})$);
   \draw ($(y)$) -- ($(x)+(y)+(y)+(z)$) -- ($(x)+(y)+(z)$);
   \draw ({0.49*(cos(330)/2+cos(210)/2+0.25)},{0.49*(sin(330)/2+sin(210)/2+1.05)}) -- ($(x)+(y)+(y)+(z)$);

   \node[inner sep = -2pt] (x1) at ($(o)-(x)$) [label={[label distance=0pt]180:$v_1$}] {$\bullet$};
   \node[inner sep = -2pt] (x2) at ($(o)-(y)$) [label={[label distance=0pt]270:$w_1$}] {$\bullet$};
   \node[inner sep = -2pt] (x3) at ($(o)-(z)$) [label={[label distance=0pt]60:$v_3$}] {$\bullet$};
   \node[inner sep = -2pt] (x4) at ($(x)$) [label={[label distance=0pt]-30:$v_5$}] {$\bullet$};
   \node[inner sep = -2pt] (x5) at ($(x)+(y)$) [label={[label distance=0pt]-30:$v_4$}] {$\bullet$};
   \node[inner sep = -2pt] (x6) at ($(x)+(y)+(z)$) [label={[label distance=0pt]-90:$v_6$}] {$\bullet$};
   \node[inner sep = -2pt] (x7) at ($(y)+(z)$) [label={[label distance=0pt]90:$v_8$}] {$\bullet$};
   \node[inner sep = -2pt] (x8) at ($(z)$) [label={[label distance=0pt]235:$v_7$}] {$\bullet$};
   \node[inner sep = -2pt] (q1) at ($(y)$) [label={[label distance=0pt]90:$v_2$}] {$\bullet$};
   \node[inner sep = -2pt] (q2) at ($(x)+(y)+(y)+(z)$) [label={[label distance=0pt]90:$w_2$}] {$\bullet$};
   
   \node at ($(o)$) {$\times$};
\end{tikzpicture} 
\caption{The fan $\mathcal F$ giving both a toric model of $U$, and the scattering diagram for $\CC[U]$.}
\label{fig!3d-fan}
\end{center}
\end{figure}

\paragraph{The toric model for $U$.}
We consider the toric variety $(T,B)$ defined by the fan $\mathcal F$, and the locus $Z\subset T$ determined by the seed 
\[ S = \left\{ \left( (-1,0,0), \; \lambda q_1 + 1 \right), \; \left( (0,-1,0), \; \mu x_1x_3 + x_1 + x_3 + 1 \right), \; \left( (0,0,-1), \; q_1 + 1\right) \right\} \]
described in Remark~\ref{rmk!why-q1-q2}.

\begin{prop}
The pair $\pi\colon (\widetilde X,\widetilde D) \to (T,B)$ obtained by blowing up $Z\subset T$ is a toric model for $U$. This compactification $(\widetilde X,\widetilde D)$ is a small resolution of the projective compactification $(X,D)$ of $U$, and identifies the boundary components $D_1,\ldots, D_8,D_{q_1},D_{q_2}$ of $D$ with the ten components of $\widetilde D$ corresponding to the vectors $v_1,\ldots,v_8,w_1,w_2$ respectively.
\end{prop}

\begin{proof}
The proof that $\pi \colon (X,D)\to (T,B)$ is a toric model for $U$ follows from extending a similar calculation to the proof of Proposition~\ref{prop!blowup} from the toric variety $\Aa^3$ to the toric variety $T$. To see that the boundary components are identified as claimed, we can check that the identifications hold for the affine chart $\Aa^3_\sigma = \Spec\CC[\sigma^\vee\cap M]\subset T\dashrightarrow X$ for each maximal cone $\sigma\subset \mathcal F$. 

After blowing up the locus $Z\subset T$ we see that the boundary divisor $\widetilde D$ is isomorphic to $B$, except for the fact that each of $\widetilde D_1,\widetilde D_3,\widetilde D_5$ and $\widetilde D_7$ are blown-up in two boundary points, turning these four boundary components into del Pezzo surfaces of degree 5 with an anticanonical pentagon of $(-1)$-curves (as in Example~\ref{eg!dP5-pentagon}). The model $(X,D)$ is now obtained from $(\widetilde X,\widetilde D)$ by contracting sixteen $\mathcal{O}(-1,-1)$-curves in the boundary (giving the sixteen nodes of Lemma~\ref{lem!sing}). These are the eight boundary strata corresponding to the cones $\langle x_i,q_i \rangle$ for $i\in\ZZ/8\ZZ$ and eight more curves, four in each of the boundary components $\widetilde D_1$, $\widetilde D_3$, $\widetilde D_5$, $\widetilde D_7$, giving the eight nodes along the lines~$\PP^1_{x_i,x_{i+1}}$. These sixteen curves are indicated by the dashed lines in Figure~\ref{fig!toric-model}.
\end{proof}

\begin{figure}[htbp]
\begin{center}
\begin{tikzpicture}[scale=0.8]
   \begin{scope}
   \draw (0,4) -- (2,3) (0,4) -- (-2,3) (-4,-2) -- (0,0) (4,-2) -- (0,0);
   \draw[dashed] (4,1) -- (4,-2) (-4,1) -- (-4,-2) (0,0) -- (0,4);
   \draw[very thick] (0,2) -- (-4,0) (-1.05,-0.525) -- (-0.5,-0.75) -- (0.5,-0.75) -- (1.05,-0.525) (0,1) -- (4,-1) (-0.5,-0.75) -- (-3.5,-2.25) (0.5,-0.75) -- (3.5,-2.25);
   \draw (-2,3) -- (-4,1) (2,3) -- (4,1)  (-4,-2) -- (0,-4) -- (4,-2);
   \node at (-2.5,0) {\scriptsize $(-1,0,0)$};
   \node at (0,-5/2) {\scriptsize $(0,-1,0)$};
   \node at (2.5,0.5) {\scriptsize $(0,0,-1)$};
   \draw[dashed] (-1.1,3.45) -- (-1.1,-0.55) (1.1,3.45) -- (1.1,-0.55);
   \draw[dashed] (-1,-0.5) -- (-1,0.5) (0,1) -- (-1,0.5) -- (-2.75,2.25);
   \draw[dashed] (1,-0.5) -- (1,1.5) (0,2) -- (1,1.5) -- (2.25,2.75);
   \end{scope}
   
   \begin{scope}[xshift=10cm]
   \draw (4,1) -- (2,3) -- (0,4) -- (-2,3) -- (-4,1) (-4,-2) -- (0,-4) -- (4,-2);
   \draw[dashed] (4,-2) -- (4,1) (-4,-2) -- (-4,1) (-2,3) -- (-1,5/2) (1,5/2) -- (2,3) (-2,0) -- (-1,1) (1,1) -- (2,0) (0,-2) -- (0,-4);
   \draw (-1,5/2) -- (-1,1) -- (1,1) -- (1,5/2) -- cycle;
   \draw (-4,1) -- (-2,0) (2,0) -- (4,1) (-2,0) -- (0,-2) -- (2,0);
   \node at (0,3.25) {\scriptsize $(0,1,0)$};
   \node at (0,0) {\scriptsize $(1,1,1)$};
   \node at (0,1.75) {\scriptsize $(1,2,1)$};
   \node at (-2,1.5) {\scriptsize $(0,1,1)$};
   \node at (2,1.5) {\scriptsize $(1,1,0)$};
   \node at (-1.5,-2) {\scriptsize $(0,0,1)$};
   \node at (1.5,-2) {\scriptsize $(1,0,0)$};
   \draw[dashed] (3.05,-2.475) -- (3.05,0.525);
   \draw[dashed] (-3.05,-2.475) -- (-3.05,0.525);
   \draw[dashed] (-4,-1) -- (-2.95,-1.525) (-2.95,-2.525) -- (-2.95,-1.525) -- (-1.5,-0.5);
   \draw[dashed] (4,-1) -- (2.95,-1.525) (2.95,-2.525) -- (2.95,-1.525) -- (1.5,-0.5);
   \end{scope}
\end{tikzpicture} 
\caption{A toric model $\pi\colon (\widetilde X,\widetilde D)\to (T,B)$ for $U$ is constructed by blowing up the locus $Z\subset B$ which is a union of three rational curves (drawn tropically with thick lines). The model $(X,D)$ is obtained from $(\widetilde X,\widetilde D)$ as a small contraction of the sixteen $\mathcal O(-1,-1)$-curves (drawn tropically with dashed lines) to ordinary nodes.}
\label{fig!toric-model}
\end{center}
\end{figure}
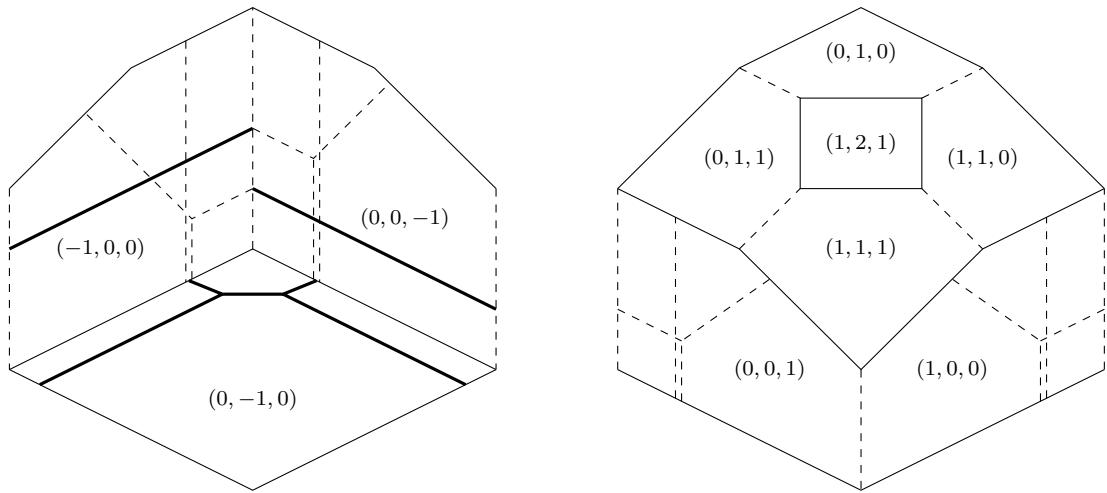

\paragraph{The tropicalisation $N_U$.}
We can now construct $N_U$, the tropicalisation of $U$ with respect to this seed, by altering the integral affine structure along the codimension 1 cones of $\mathcal F$ by following the process described in \S\ref{sec!trop}. We arrive at the following description.

\begin{lem} \label{lem!bending}
The integral affine structure on $N_U$ is obtained by bending any line that passes through one of the following seven walls of $\mathcal F$: 
\[ \langle v_7,v_1\rangle, \: \langle v_1,v_3\rangle, \: \langle v_3,v_5\rangle, \: \langle v_1,w_1\rangle, \: \langle v_3,w_1\rangle, \: \langle v_5,w_1\rangle, \: \langle v_7,w_1\rangle. \] 
If $M_{i,i+2}$ represents the bend from passing from the cone $\langle v_i,w_i,v_{i+2}\rangle$ to the cone $\langle v_i,v_{i+1},v_{i+2}\rangle$ through the wall $\langle v_i,v_{i+2}\rangle$ for $i=7,1,3$, then 
\[  
M_{71} = \begin{pmatrix}
1 & 1 & 0 \\
0 & 1 & 0 \\
0 & 0 & 1
\end{pmatrix}, \quad M_{13} = \begin{pmatrix}
1 & 1 & 0 \\
0 & 1 & 0 \\
0 & 1 & 1
\end{pmatrix}, \quad M_{35} = \begin{pmatrix}
1 & 0 & 0 \\
0 & 1 & 0 \\
0 & 1 & 1
\end{pmatrix}, \]
and if $M_{iq}$ represents the bend from passing from the cone $\langle v_{i-2},v_i,w_i\rangle$ to the cone $\langle v_i,v_{i+2},w_i\rangle$ through the wall $\langle v_i,w_i\rangle$ for $i=1,3,5,7$, then
\[ M_{1q} = \begin{pmatrix}
 1 & 0 & 0 \\
-1 & 1 & 0 \\
 0 & 0 & 1
\end{pmatrix}, \quad M_{3q} = \begin{pmatrix}
1 & 0 & 0 \\
0 & 1 & 1 \\
0 & 0 & 1
\end{pmatrix}, \quad M_{5q} = M_{1q}^{-1}, \quad M_{7q} = M_{3q}^{-1}. \]
The singular locus of $N_U$ is given by the four rays $\RR_{\geq0}v_i$ for $i=1,3,5,7$, corresponding to the four non-toric boundary divisors of $(\widetilde X,\widetilde D)$. 
\end{lem}

\subsubsection{A scattering diagram for $U$} \label{sec!V12-scattering}

The finite cluster structure on $U$ is equivalent to the existence of a consistent scattering diagram with a finite number of chambers. As described in \S\ref{sec!scattering}, by \cite{ag} this scattering diagram can be computed inductively from some initial scattering diagram for any log Calabi--Yau variety with a suitably nice toric model. Consider the fan $\mathcal F$ defined above. Let $\mathfrak d_{ij}$ be the wall spanned by $x_i$ and $x_j$, and let $\mathfrak d_{iq}$ be the wall spanned by $x_i$ and $q_i$. In what follows, when we refer to a wall $\mathfrak d_{ij}$ it is convenient to implicitly allow $j=q$.

\begin{prop}
Decorating the walls of the fan $\mathcal F$ with the wall functions given in Table~\ref{table!scattering} defines a consistent scattering diagram $\mathfrak D_{\lambda,\mu}$ in $N_U$.
\end{prop}

\begin{table}[htp]
\caption{Wall functions for the scattering diagram $\mathfrak D_{\lambda,\mu}$.}
\begin{center} \renewcommand{\arraystretch}{1.2}
\resizebox{\textwidth}{!}{\begin{tabular}{|cl|cl|cl|}\hline 
Wall & Function & Wall & Function & Wall & Function \\ \hline 
  $\mathfrak d_{12}$ & $1 + z_2 + z_1z_2$			
& $\mathfrak d_{13}$ & $1 + z_1 + z_3 + \mu z_1z_3$ 
& $\mathfrak d_{1q}$ & $1 + z_2$  \\
  $\mathfrak d_{23}$ & $1 + \lambda z_2 + \lambda z_2z_3$	
& $\mathfrak d_{24}$ & $1 + z_2 + z_1z_2 + \lambda z_1z_2^2$ 
& $\mathfrak d_{2q}$ & $1 + \lambda z_1z_2^2z_3$  \\
  $\mathfrak d_{34}$ & $1 + \lambda z_2z_3 + \lambda\mu z_1z_2z_3$				
& $\mathfrak d_{35}$ & $1 + z_1 + z_3 + \mu z_1z_3$ 
& $\mathfrak d_{3q}$ & $1 + \lambda z_2$  \\
  $\mathfrak d_{45}$ & $1 + z_2 + z_1z_2$	
& $\mathfrak d_{46}$ & $1 + \lambda z_1z_2 + \lambda\mu z_1z_2z_3 + \lambda\mu z_1^2z_2^2z_3$ 
& $\mathfrak d_{4q}$ & $1 + \lambda\mu z_1z_2^2z_3$  \\
  $\mathfrak d_{56}$ & $1 + z_2z_3 + \mu z_1z_2z_3$ 		
& $\mathfrak d_{57}$ & $1 + z_1 + z_3 + \mu z_1z_3$ 
& $\mathfrak d_{5q}$ & $1 + z_2$  \\
  $\mathfrak d_{67}$ & $1 + \lambda z_2z_3 + \lambda\mu z_1z_2z_3$	
& $\mathfrak d_{68}$ & $1 + z_2z_3 + \mu z_1z_2z_3 + \lambda\mu z_1z_2^2z_3^2$ 
& $\mathfrak d_{6q}$ & $1 + \lambda z_1z_2^2z_3$  \\
  $\mathfrak d_{78}$ & $1 + \lambda z_2 + \lambda z_2z_3$	
& $\mathfrak d_{71}$ & $1 + z_1 + z_3 + \mu z_1z_3$ 
& $\mathfrak d_{7q}$ & $1 + \lambda z_2$  \\
  $\mathfrak d_{81}$ & $1 + z_2z_3 + \mu z_1z_2z_3$		 	
& $\mathfrak d_{82}$ & $1 + \lambda z_2 + \lambda z_2z_3 + \lambda z_2^2z_3$ 
& $\mathfrak d_{8q}$ & $1 + \lambda\mu z_1z_2^2z_3$  \\\hline 
\end{tabular}}
\end{center}
\label{table!scattering}
\end{table}%

\begin{proof}
It is straightforward to check (via computer algebra) that composing the wall-crossing automorphisms corresponding to an oriented loop around any joint (i.e.\ a codimension 2 cone of $\mathcal F$) in $\mathfrak D_{\lambda,\mu}$ gives the identity.
\end{proof}

\begin{rmk}
The wall functions can be easily read off from the mutation relations. First note that, when expanded as a Laurent polynomial in $x_1,q_1,x_3$, the cluster variables all have a monic leading monomial\footnote{This is not a happy accident. Indeed the rescaling that happens in \S\ref{sect!U-lambda-mu} whilst deriving the equations of $U$ was specially chosen in order to achieve this very fact.}, i.e.\ the constant term of their numerator is equal to $1$. Now the scattering function on the wall $\mathfrak d_{ij}$ is obtained by substituting the leading monomial of each cluster variable (written in terms of $z_1,z_2,z_3$) into the righthand side of the corresponding mutation relation and clearing the denominators. For example crossing the wall $\mathfrak{d}_{68}$ corresponds to the mutation $x_7q_2 = \lambda\mu + x_6 + \mu x_8 + x_6x_8$. If we substitute $x_6\mapsto (z_1z_2z_3)^{-1}$ and $x_8\mapsto (z_2z_3)^{-1}$ into the righthand side and clear the denominator we get the scattering function $f_{68}$ in Table~\ref{table!scattering}. Of course in general the scattering diagram is usually introduced in order to derive the mutation relations, not the other way round.
\end{rmk}

Since the mutation relations generating the cluster exchange graph for $U$ are now easily read off as the wall-crossing automorphisms in $\mathfrak D_{\lambda,\mu}$, this shows that the cluster variables $x_1,\ldots,x_8,q_1,q_2$ can now be recovered from $\mathfrak D_{\lambda,\mu}$ as the theta functions attached to the integral points $v_1,\ldots,v_8,v_{q_1},v_{q_2}\in N_U(\ZZ)$ respectively. Moreover, these could be computed by counting broken lines in $\mathfrak D_{\lambda,\mu}$.

Since the tropicalisation $N_U$ also has the nice property that once a straight line has left a cone of $\mathcal F$ it can never return, which can been seen by an explicit computation with affine structure of $N_U$ as described in Lemma~\ref{lem!bending}. Therefore, by a similar argument to the dimension 2 case \S\ref{sec!dp5-scattering}, for all $a,b,c\geq0$ the integral point $av_1+bv_2+cv_3\in N_U(\ZZ)$ belonging to the cone $\langle v_1,v_2,v_3 \rangle$ of $\mathcal F$ must parameterise the theta function $\vartheta_{av_1+bv_2+cv_3} = \vartheta_{v_1}^a\vartheta_{v_2}^b\vartheta_{v_3}^c$, and similarly for all of the other cones of $\mathcal F$. This gives a complete description of the theta functions on~$U$.

\subsection{Applications to mirror symmetry}

In this section, and until the end of the paper, we restrict to looking at the special fibre $U:=U_{1,1}$ of our family $\mathcal U$ of cluster varieties, corresponding to the values $\lambda=\mu=1$. This is the most degenerate fibre, since it gains the two extra interior nodal singularities, and it is precisely the affine variety defined by the original recurrence relation (LR$_3$). The reason for this specialisation is that the expansions of the cluster variables as Laurent polynomials on a given cluster torus $\TT_{123}$ all now satisfy the binomial edge coefficient condition of Remark~\ref{rmk!binomial-edge-coeffs}, which we expect for our mirror Landau--Ginzburg potentials. We continue to let $(X,D)$ be the pair obtained by taking the projective closure $X = \overline{U}\subset \PP^8$ with boundary divisor $D$.

\subsubsection{Landau--Ginzburg models mirror to $V_{12}$ and $V_{16}$} \label{sec!LG-V12-V16}

We now describe how to construct explicit Landau--Ginzburg models which are mirror to the Fano 3-fold pairs considered in \S\ref{sect!LR3-terms}.

\begin{prop}
The polytopes $P = \conv \left( x_1,\ldots,x_8 \right)$ and $Q = \conv \left( x_1,\ldots,x_8, q_1, q_2 \right)$ are a pair of dual polytopes in $M_U$. Moreover they are, combinatorially, a realisation of the two polytopes of Figure~\ref{fig!polytopes} when considered in the affine structure of $M_U$.
\end{prop}

\begin{proof}
When drawn with respect to the fan $\mathcal F$ in $\RR^3$, the polytopes $P$ and $Q$ are pictured as follows
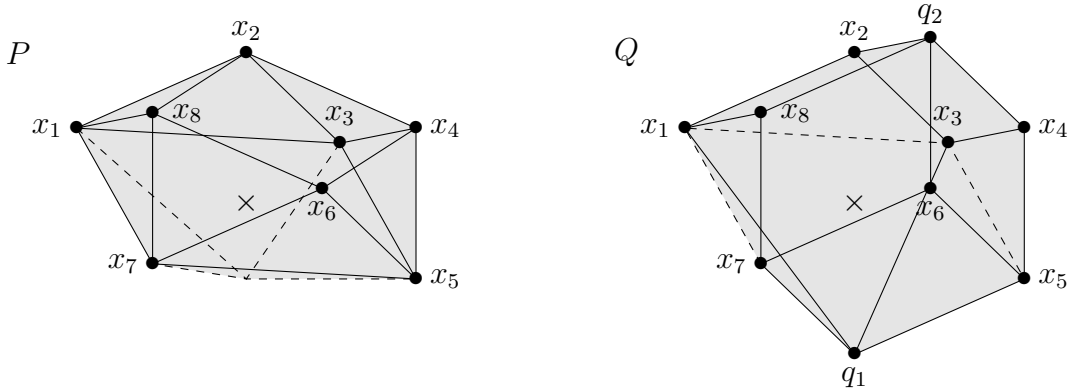
\begin{center}
\begin{tikzpicture}[scale=2]
   \begin{scope}
   \node at (-1.5,1) {$P$};
   \node[inner sep = -2pt] (o) at (0,0) {};
   \node (x) at ({cos(330)+0.25},{sin(330)}) {};
   \node (y) at ({cos(90)},{sin(90)}) {};
   \node (z) at ({cos(210)+0.25},{sin(210)+0.1}) {};

    \fill[gray,opacity = 0.2] ($(y)$) -- ($(x)+(y)$) -- ($(x)$) -- (0,-1/2) -- ($(z)$) -- ($(o)-(x)$) -- cycle;
    \draw ($(y)$) -- ($(x)+(y)$) -- ($(x)+(y)+(z)$) -- ($(z)+(y)$) -- cycle;
    \draw ($(o)-(x)$) -- ($(o)-(z)$) -- ($(x)$) -- ($(z)$) -- cycle;
    \draw ($(o)-(x)$) -- ($(y)$) -- ($(o)-(z)$) -- ($(x)+(y)$) -- ($(x)$) -- ($(x)+(y)+(z)$) -- ($(z)$) -- ($(z)+(y)$) -- cycle;
    \draw[dashed] ($(o)-(x)$) -- (0,-1/2) -- ($(x)$) ($(o)-(z)$) -- (0,-1/2) -- ($(z)$);
   \node at (o) {$\times$};
   \node[inner sep = -2pt] (x1) at ($(o)-(x)$) [label={[label distance=0pt]180:$x_1$}] {$\bullet$};
   \node[inner sep = -2pt] (x2) at ($(y)$) [label={[label distance=0pt]90:$x_2$}] {$\bullet$};
   \node[inner sep = -2pt] (x3) at ($(o)-(z)$) [label={[label distance=2pt]90:$x_3$}] {$\bullet$};
   \node[inner sep = -2pt] (x4) at ($(x)+(y)$) [label={[label distance=0pt]0:$x_4$}] {$\bullet$};
   \node[inner sep = -2pt] (x5) at ($(x)$) [label={[label distance=0pt]0:$x_5$}] {$\bullet$};
   \node[inner sep = -2pt] (x6) at ($(x)+(y)+(z)$) [label={[label distance=0pt]270:$x_6$}] {$\bullet$};
   \node[inner sep = -2pt] (x7) at ($(z)$) [label={[label distance=0pt]180:$x_7$}] {$\bullet$};
   \node[inner sep = -2pt] (x8) at ($(y)+(z)$) [label={[label distance=2pt]0:$x_8$}] {$\bullet$};
   \end{scope}
   
   \begin{scope}[xshift=4cm]
   \node at (-1.5,1) {$Q$};
   \node[inner sep = -2pt] (o) at (0,0) {};
   \node (x) at ({cos(330)+0.25},{sin(330)}) {};
   \node (y) at ({cos(90)},{sin(90)}) {};
   \node (z) at ({cos(210)+0.25},{sin(210)+0.1}) {};

    \fill[gray,opacity = 0.2] ($(y)$) -- ($(x)+(y)+(y)+(z)$) -- ($(x)+(y)$) -- ($(x)$) -- ($(o)-(y)$) -- ($(z)$) -- ($(o)-(x)$) -- cycle;
    
    \draw ($(x)+(y)$) -- ($(x)+(y)+(y)+(z)$) -- ($(x)+(y)+(z)$);
    \draw ($(z)+(y)$) -- ($(x)+(y)+(y)+(z)$) -- ($(y)$);
    \draw ($(o)-(x)$) -- ($(o)-(y)$) -- ($(x)$);
    \draw ($(o)-(z)$) -- ($(o)-(y)$) -- ($(z)$);
    \draw ($(o)-(x)$) -- ($(y)$) -- ($(o)-(z)$) -- ($(x)+(y)$) -- ($(x)$) -- ($(x)+(y)+(z)$) -- ($(z)$) -- ($(z)+(y)$) -- cycle; 
    \draw[dashed] ($(z)$) -- ($(o)-(x)$) -- ($(o)-(z)$) -- ($(x)$);
   
   \node at (o) {$\times$};
   \node[inner sep = -2pt] (x1) at ($(o)-(x)$) [label={[label distance=0pt]180:$x_1$}] {$\bullet$};
   \node[inner sep = -2pt] (x2) at ($(y)$) [label={[label distance=0pt]90:$x_2$}] {$\bullet$};
   \node[inner sep = -2pt] (x3) at ($(o)-(z)$) [label={[label distance=2pt]90:$x_3$}] {$\bullet$};
   \node[inner sep = -2pt] (x4) at ($(x)+(y)$) [label={[label distance=0pt]0:$x_4$}] {$\bullet$};
   \node[inner sep = -2pt] (x5) at ($(x)$) [label={[label distance=0pt]0:$x_5$}] {$\bullet$};
   \node[inner sep = -2pt] (x6) at ($(x)+(y)+(z)$) [label={[label distance=0pt]270:$x_6$}] {$\bullet$};
   \node[inner sep = -2pt] (x7) at ($(z)$) [label={[label distance=0pt]180:$x_7$}] {$\bullet$};
   \node[inner sep = -2pt] (x8) at ($(y)+(z)$) [label={[label distance=2pt]0:$x_8$}] {$\bullet$};
   \node[inner sep = -2pt] (q1) at ($(o)-(y)$) [label={[label distance=0pt]-90:$q_1$}] {$\bullet$};
   \node[inner sep = -2pt] (q2) at ($(x)+(y)+(y)+(z)$) [label={[label distance=0pt]90:$q_2$}] {$\bullet$};
   \end{scope}
\end{tikzpicture} 
\end{center}
and so we are required to show that the dashed edges in the diagram are really flat when considered with respect to the affine structure on $M_U$.

The polytope $P$ gives an integral piecewise-linear function $\phi_P\colon N_U\to \RR$ given by $\phi_P(n) = \min\{ c\in \RR : n\in cP \}$, which is linear in each cone of our fan $\mathcal F$. Since $\phi_P(v_i) = 1$ for all $i\in\ZZ/8\ZZ$ and $\phi_P(w_1)=\phi_P(w_2)=2$, this determines a Weil divisor $\Xi_P=D_1 + \cdots + D_8 + 2D_{q_1} + 2D_{q_2}$ on $\widetilde X$. Now any codimension 2 cone $\tau$ of $\mathcal F$, corresponds to a curve $C_\tau$ in the boundary of $\widetilde X$. Using the criterion described in \S\ref{sec!trop}, the map $\phi_P$ is linear with respect to the affine structure on $N_U$ if $\Xi_P\cdot C_\tau=0$. Now we easily check that $\Xi_P$ is only nonzero on the curve classes corresponding to the cones $\langle x_i,x_{i+1} \rangle$ and $\langle x_i,q_i \rangle$. 

The computation is similar for the polytope $Q$, which corresponds to the Weil divisor $\Xi_Q = D_1 +\cdots + D_8 + D_{q_1} + D_{q_2}$ on $\widetilde X$.
\end{proof}

We now consider the two compactifications $(X_P,D_P)$ and $(X_Q,D_Q)$ of $U$ which were discussed in \S\ref{sect!LR3-terms}. Note that $(X_P,D_P)=(X,D)$ is the standard projective compactification of $U$ that we have been considering throughout this section.

\paragraph{An unprojection.}

Before describing the mirror Landau-Ginzburg models to $(X_P,D_P)$ and $(X_Q,D_Q)$, we explain how to go from one model to the other by making a birational modification $\psi\colon X_P\dashrightarrow X_Q$, called an \emph{unprojection}. Geometrically, this contracts the two $\PP^1\times\PP^1$ divisors $D_{1357},D_{2468}\subset D$. This is done by adjoining $q_1 = x_0^{-1}(x_1x_5 - x_0^2)$ and $q_2 = x_0^{-1}(x_2x_6 - x_0^2)$, which are rational sections of $\sO_X(1)$, as generators to the homogeneous coordinate ring $\CC[X_P]$. 

\begin{prop}
Let $X_Q = \overline{\left\{ (x_0:x_1:\ldots:x_8:q_1:q_2) \in \PP^{10} : (x_0:x_1:\ldots:x_8)\in X_P \right\}}$ be the graph of the two rational sections $q_1,q_2$ on $X_P$. Then $X_Q\subset\PP^{10}$ is a projectively Gorenstein 3-fold of codimension 7 determined by the 21 equations
\begin{align}
x_ix_{i+3} &= x_0(x_0 + x_{i+1} + x_{i+2}) \tag{$\times8$} \\
x_ix_{i+4} &= x_0(q_i + x_0) \tag{$\times4$} \\
x_iq_{i+1} &= (x_0+x_{i+1})(x_0+x_{i-1})  \tag{$\times8$} \\
q_1q_2 &= x_0(4x_0+x_1+x_2+x_3+x_4+x_5+x_6+x_7+x_8) \tag{$\times1$}
\end{align}
and has graded resolution with Betti numbers $(1,21,64,70,70,64,21,1)$.  
\end{prop}

\begin{proof}
This is an application of the theory of unprojection developed by Papadakis--Reid~\cite{pap-reid}. The divisor $D_{q_1} = \VV(x_0,x_2,x_4,x_6,x_8,x_1x_5-x_3x_7)\subset \PP^8$ is a complete intersection of codimension 6 (hence projectively Gorenstein) and contained in $X_P\subset \PP^8$ which is projectively Gorenstein of codimension 5. By adjunction we have $K_{X_P}=\sO_{X_P}(-1)$ and $K_{D_{q_1}}=\sO_{D_{q_1}}(-2)$, and since $(-1)-(-2)=1>0$ we can apply \cite[Theorem~1.5]{pap-reid} to obtain an unprojection variable $q_1$ as a rational section of $\sO_{X_P}(1)$. We compute that $q_1=x_0^{-1}(x_1x_5-x_0^2)$ works as an unprojection variable, and satisfies the equations
\[ x_0q_1 = x_1x_5 - x_0^2 = x_3x_7 - x_0^2, \qquad x_{i}q_1 = (x_0+x_{i-1})(x_0+x_{i+1}) \quad \text{for } i=2,4,6,8 \]
and thus we obtain an unprojection $(D_{q_1}\subset X_P)\dashrightarrow(P_{q_1}\in X_P')$ which contracts $D_{q_1}$ to an ordinary node at the coordinate point $P_{q_1}\in X_P'$. Now we repeat the procedure to the strict transform of $D_{q_1}$ in $X_P'$ to get $q_2=x_0^{-1}(x_2x_6-x_0^2)$ and an unprojection $(D_{q_2}\subset X_P')\dashrightarrow (P_{q_2}\in X_Q)$.
\end{proof}

\paragraph{Geometry of the unprojection.}
The unprojection factors as a blowup of the (non-Cartier) Weil divisor $D_{q_1}$, which resolves four nodes of $X_P$, followed by the contraction of the strict transform to an ordinary node. The result of this operation is depicted below. To construct $X_Q$ we do this unprojection to both $D_{q_1}$ and $D_{q_2}\subset X_P$ (i.e.\ to both the front and back squares of the diagram). 
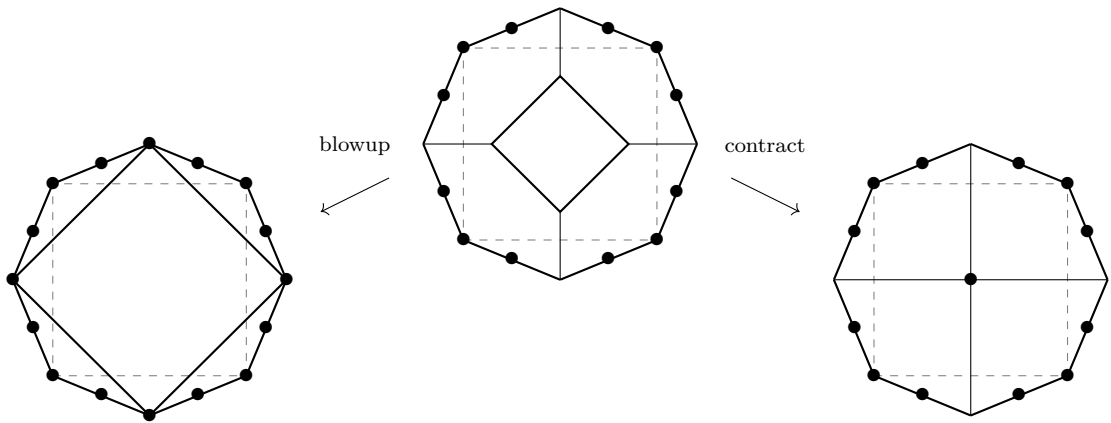
\begin{center}
\begin{tikzpicture}[scale=1.8]
  \draw[gray,dashed] ({cos(45)},{sin(45)}) -- ({cos(135)},{sin(135)}) -- ({cos(225)},{sin(225)}) -- ({cos(315)},{sin(315)}) -- cycle;
  \draw[thick] ({cos(0)},{sin(0)}) -- ({cos(90)},{sin(90)}) -- ({cos(180)},{sin(180)}) -- ({cos(270)},{sin(270)}) -- cycle;  
  
  \foreach \i in {0,...,7} {
  	\draw[thick] ({cos(45*\i)},{sin(45*\i)}) -- ({cos(45*\i+45)},{sin(45*\i+45)});
  	\node at ({cos(45*\i)},{sin(45*\i)}) {$\bullet$};
  	\node at ({0.92*cos(45*\i+22.5)},{0.92*sin(45*\i+22.5)}) {$\bullet$};
  };
  
  \begin{scope}[xshift=3cm,yshift=1cm]
  \draw[gray,dashed] ({cos(45)},{sin(45)}) -- ({cos(135)},{sin(135)}) -- ({cos(225)},{sin(225)}) -- ({cos(315)},{sin(315)}) -- cycle;
  \draw[thick] ({0.5*cos(0)},{0.5*sin(0)}) -- ({0.5*cos(90)},{0.5*sin(90)}) -- ({0.5*cos(180)},{0.5*sin(180)}) -- ({0.5*cos(270)},{0.5*sin(270)}) -- cycle;  
  
  \foreach \i in {0,...,7} {
  	\draw[thick] ({cos(45*\i)},{sin(45*\i)}) -- ({cos(45*\i+45)},{sin(45*\i+45)});
  	\node at ({0.92*cos(45*\i+22.5)},{0.92*sin(45*\i+22.5)}) {$\bullet$};
  };
  \foreach \i in {0,...,3} {
  	\node at ({cos(90*\i + 45)},{sin(90*\i + 45)}) {$\bullet$};
	\draw ({cos(90*\i)},{sin(90*\i)}) -- ({0.5*cos(90*\i)},{0.5*sin(90*\i)});
  };
  \draw[->] (-1.25,-0.25) to node[yshift=0.65cm] {\scriptsize blowup} (-1.75,-0.5);
  \draw[->] (1.25,-0.25) to node[yshift=0.65cm] {\scriptsize contract} (1.75,-0.5);
  \end{scope}
  
  \begin{scope}[xshift=6cm]
  \draw[gray,dashed] ({cos(45)},{sin(45)}) -- ({cos(135)},{sin(135)}) -- ({cos(225)},{sin(225)}) -- ({cos(315)},{sin(315)}) -- cycle;
  \node at (0,0) {$\bullet$};
  
  \foreach \i in {0,...,7} {
  	\draw[thick] ({cos(45*\i)},{sin(45*\i)}) -- ({cos(45*\i+45)},{sin(45*\i+45)});
  	\node at ({0.92*cos(45*\i+22.5)},{0.92*sin(45*\i+22.5)}) {$\bullet$};
  };
  \foreach \i in {0,...,3} {
  	\node at ({cos(90*\i + 45)},{sin(90*\i + 45)}) {$\bullet$};
	\draw ({cos(90*\i)},{sin(90*\i)}) -- (0,0);
  };
  \end{scope}
\end{tikzpicture}
\end{center}

Note that the unprojection $\psi\colon X_P\dashrightarrow X_Q$ restricts to an isomorphism on the open set $U=X_P\setminus D_P$, since $q_1$ and $q_2$ are regular functions on $U$. Therefore if $D_Q$ is the strict transform of the boundary divisor $D_P$ then $(X_Q,D_Q)$ is also a log Calabi--Yau compactification of $U$.

\paragraph{The potential $w_Q$.}

According to Remark~\ref{rmk!binomial-edge-coeffs}, since the ten nonzero integral points of $Q$ are all vertices, the Landau--Ginzburg potential supported on $Q$ that we consider is uniquely determined by summing the corresponding ten theta functions with coefficient~1,~i.e.\ 
\[ w_Q = x_1 + x_2 + x_3 + x_4 + x_5 + x_6 + x_7 + x_8 + q_1 + q_2. \]
We note that this is $\sigma_3$-invariant function and, indeed, restricting to the cluster torus chart $\TT_{123}\subset U$ gives
\[ (w_Q + 5)|_{\TT_{123}} = \frac{(1 + x_1 + x_2)(1 + x_2 + x_3)(1 + x_1 + x_2 + x_3 + x_1x_3)}{x_1x_2x_3} \]
which is well-known to be an invariant function for Lyness map $\sigma_3$. 

\paragraph{The period $\pi_{w_Q}(t)$.}
Consider the period 
\[ \pi_{w_Q}(t) = 1  + 48t^2 + 600t^3 + 13176t^4 + 276480t^5 + 6259800t^6 + 146064240t^7 + \cdots \]
which agrees with the regularised quantum period $\widehat{G}_X(t)$ for the Fano 3-fold $X$ of type $V_{12}$, as listed in \cite[\S13]{ccgk}. Alternatively, the shifted potential $w_Q+5$ has period
\[ \pi_{w_Q+5}(t) = 1 + 5t + 73t^2 + 1445t^3 + 33001t^4 + 819005t^5 + 21460825t^6+ \cdots \]
which can be recognised as the Ap\'ery series $\pi_{w+5}(t) = \sum_{n=0}^\infty\sum_{k=0}^n \binom{n}{k}^2 \binom{n+k}{k}^2 t^n$ and is also well-known as a mirror period sequence for the Fano 3-fold $V_{12}$, cf.\ \cite{gz}. 

\paragraph{A K3 fibration.}
In order to appreciate some beautiful geometry hiding here, it is convenient to consider extending the potential $w_Q$ to the compactification $X_Q$. In analogy to the critical value $-3$ for the Landau--Ginzburg potential on the del Pezzo surface, the factorisation of $(w_Q + 5)|_{\TT_{123}}$ corresponds to the fact that the interior divisor $E_Q := w_Q^{-1}(-5) \subset U$ is reducible.  
\begin{figure}[b!]
\begin{center}
\begin{tikzpicture}[scale=2.5]
  
	\node at (-1,1) {$D_P\cap E_P$};
  \draw[thick] ({cos(0)},{sin(0)}) -- ({cos(90)},{sin(90)}) -- ({cos(180)},{sin(180)}) -- ({cos(270)},{sin(270)}) -- cycle;
  \foreach \i in {0,...,7} {
  	\draw[thick] ({cos(45*\i)},{sin(45*\i)}) -- ({cos(45*\i+45)},{sin(45*\i+45)});
  };  
  \foreach \i in {0,...,7} {
  	\node at ({cos(45*\i)},{sin(45*\i)}) {$\bullet$};
  	\node at ({0.92*cos(45*\i+22.5)},{0.92*sin(45*\i+22.5)}) {$\bullet$};
  };  
  \foreach \i in {0,...,3} {
  	\draw[very thick] ({0.92*cos(90*\i+22.5)},{0.92*sin(90*\i+22.5)}) -- ({0.8*cos(90*\i+45)},{0.8*sin(90*\i+45)}) -- ({0.92*cos(90*\i+45+22.5)},{0.92*sin(90*\i+45+22.5)});
  };
  \draw[very thick] ({0.8*cos(45)},{0.8*sin(45)}) -- ({0.8*cos(180+45)},{0.8*sin(180+45)});
  \draw[very thick] ({0.8*cos(135)},{0.8*sin(135)}) -- ({0.8*cos(270+45)},{0.8*sin(270+45)});
  
  \begin{scope}[rotate = 45, opacity=0.5, dashed]
  \draw[thick] ({cos(0)},{sin(0)}) -- ({cos(90)},{sin(90)}) -- ({cos(180)},{sin(180)}) -- ({cos(270)},{sin(270)}) -- cycle;
  \foreach \i in {0,...,7} {
  	\draw[thick] ({cos(45*\i)},{sin(45*\i)}) -- ({cos(45*\i+45)},{sin(45*\i+45)});
  };  
  \foreach \i in {0,...,7} {
  	\node at ({cos(45*\i)},{sin(45*\i)}) {$\bullet$};
  	\node at ({0.92*cos(45*\i+22.5)},{0.92*sin(45*\i+22.5)}) {$\bullet$};
  };  
  \foreach \i in {0,...,3} {
  	\draw[very thick] ({0.92*cos(90*\i+22.5)},{0.92*sin(90*\i+22.5)}) -- ({0.8*cos(90*\i+45)},{0.8*sin(90*\i+45)}) -- ({0.92*cos(90*\i+45+22.5)},{0.92*sin(90*\i+45+22.5)});
  };
  \draw[very thick] ({0.8*cos(45)},{0.8*sin(45)}) -- ({0.8*cos(180+45)},{0.8*sin(180+45)});
  \draw[very thick] ({0.8*cos(135)},{0.8*sin(135)}) -- ({0.8*cos(270+45)},{0.8*sin(270+45)});
  \end{scope}
  
  \begin{scope}[xshift = -3cm]
	\node at (-1,1) {$D_Q\cap E_Q$};
  \draw[thick] ({cos(0)},{sin(0)}) -- ({cos(180)},{sin(180)});
  \draw[thick] ({cos(90)},{sin(90)}) -- ({cos(270)},{sin(270)});
  \foreach \i in {0,...,7} {
  	\draw[thick] ({cos(45*\i)},{sin(45*\i)}) -- ({cos(45*\i+45)},{sin(45*\i+45)});
  	\node at ({0.92*cos(45*\i+22.5)},{0.92*sin(45*\i+22.5)}) {$\bullet$};
  };
  \foreach \i in {0,...,3} {
  	\draw[very thick] ({0.92*cos(90*\i+22.5)},{0.92*sin(90*\i+22.5)}) -- ({0.5*cos(90*\i+90)},{0.5*sin(90*\i+90)});
  	\draw[very thick] ({0.92*cos(90*\i-22.5)},{0.92*sin(90*\i-22.5)}) -- ({0.5*cos(90*\i-90)},{0.5*sin(90*\i-90)});
  };
  \end{scope}
  
  \begin{scope}[xshift = -3cm, rotate = 45, opacity=0.5, dashed]
  \draw[thick] ({cos(0)},{sin(0)}) -- ({cos(180)},{sin(180)});
  \draw[thick] ({cos(90)},{sin(90)}) -- ({cos(270)},{sin(270)});
  \foreach \i in {0,...,7} {
  	\draw[thick] ({cos(45*\i)},{sin(45*\i)}) -- ({cos(45*\i+45)},{sin(45*\i+45)});
  	\node at ({0.92*cos(45*\i+22.5)},{0.92*sin(45*\i+22.5)}) {$\bullet$};
  };
  \foreach \i in {0,...,3} {
  	\draw[very thick] ({0.92*cos(90*\i+22.5)},{0.92*sin(90*\i+22.5)}) -- ({0.5*cos(90*\i+90)},{0.5*sin(90*\i+90)});
  	\draw[very thick] ({0.92*cos(90*\i-22.5)},{0.92*sin(90*\i-22.5)}) -- ({0.5*cos(90*\i-90)},{0.5*sin(90*\i-90)});
  };
  \end{scope}
\end{tikzpicture}
\caption{The sixteen lines of intersection for the boundary divisor $D_Q$ and the degenerate interior divisor $E_Q$ represented inside $D_Q$. (b) Similarly for the twelve lines of intersection of $D_P$ and $E_P$.}
\label{fig!D-and-E}
\end{center}
\end{figure}
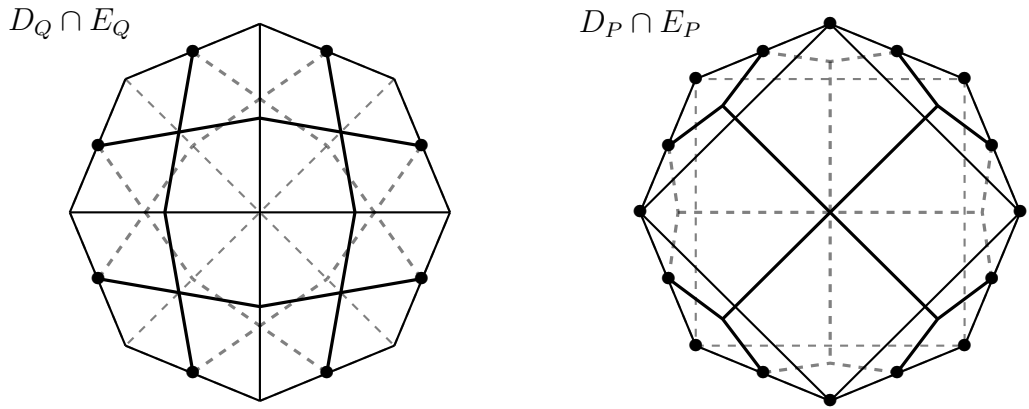
Then, in the notation of Proposition~\ref{prop!div}, the reducible fibre is $E_Q = \cup_{i=1}^8 E_{i,i+2}$ which is the union of eight copies of $\PP^1\times \PP^1$ and happens to be isomorphic to $D_Q$ (i.e.\ these eight components also meet like the faces of the polytope $Q$). The intersection of the divisors $D_Q\cap E_Q$ is a set of sixteen lines, given by the pair of lines on each face of $D_Q$ which pass through the eight nodes $(1:-1)\in \PP^1_{x_i,x_{i+1}}$ in the boundary of $X_Q$. These lines are shown in Figure~\ref{fig!D-and-E}.

Now the fibres of $w_Q\colon X_Q\dashrightarrow \PP^1$ are given by the anticanonical pencil $|D_Q,E_Q|$ with baselocus the sixteen lines described above. The general such member of this pencil is a K3 surface $S_t := w_Q^{-1}(t)$ which is smooth, apart from eight ordinary double points which lie at the eight nodes of $X$ that are contained in $D_Q\cap E_Q$. Therefore, after taking its minimal resolution, the K3 surface $S_t$ contains a configuration $\Sigma=\bigcup_{i=1}^{24}\Sigma_i$ of $24$ $(-2)$-curves with the dual intersection diagram displayed in Figure~\ref{fig!minus-2-curves}(a), where the black nodes correspond to the sixteen lines of $D_Q\cap E_Q$ and the white nodes correspond to the eight ordinary double points. Since the 24 $(-2)$-curves $\Sigma$ span a lattice of rank $19$ and $S_t$ deforms in a nontrivial family, it follows that $S_t$ has Picard rank 19.
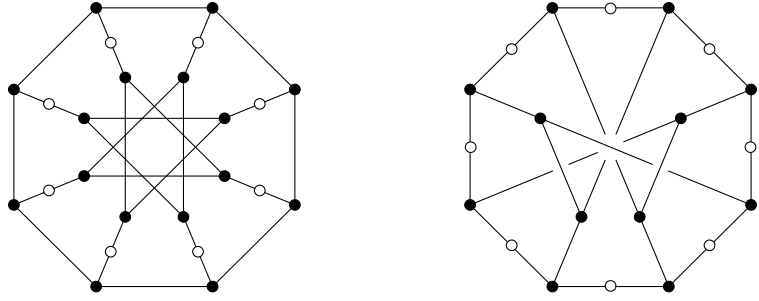
\begin{figure}[htbp]
\begin{center}\begin{tikzpicture}
\foreach \i in {1,...,8}
{
   \draw ({cos(45*\i+22.5)},{sin(45*\i+22.5)}) -- ({2*cos(45*\i+22.5)},{2*sin(45*\i+22.5)});
   \draw ({2*cos(45*\i+22.5)},{2*sin(45*\i+22.5)}) -- ({2*cos(45*(\i+1)+22.5)},{2*sin(45*(\i+1)+22.5)});
   \draw ({cos(45*\i+22.5)},{sin(45*\i+22.5)}) -- ({cos(45*(\i+3)+22.5)},{sin(45*(\i+3)+22.5)});
}
   \draw[fill=black] ({cos(0+22.5)},{sin(0+22.5)}) circle (2pt);
   \draw[fill=black] ({cos(45+22.5)},{sin(45+22.5)}) circle (2pt);
   \draw[fill=black] ({cos(90+22.5)},{sin(90+22.5)}) circle (2pt);
   \draw[fill=black] ({cos(135+22.5)},{sin(135+22.5)}) circle (2pt);
   \draw[fill=black] ({cos(180+22.5)},{sin(180+22.5)}) circle (2pt);
   \draw[fill=black] ({cos(225+22.5)},{sin(225+22.5)}) circle (2pt);
   \draw[fill=black] ({cos(270+22.5)},{sin(270+22.5)}) circle (2pt);
   \draw[fill=black] ({cos(315+22.5)},{sin(315+22.5)}) circle (2pt);
   \draw[fill=white] ({3/2*cos(0+22.5)},{3/2*sin(0+22.5)}) circle (2pt);
   \draw[fill=white] ({3/2*cos(45+22.5)},{3/2*sin(45+22.5)}) circle (2pt);
   \draw[fill=white] ({3/2*cos(90+22.5)},{3/2*sin(90+22.5)}) circle (2pt);
   \draw[fill=white] ({3/2*cos(135+22.5)},{3/2*sin(135+22.5)}) circle (2pt);
   \draw[fill=white] ({3/2*cos(180+22.5)},{3/2*sin(180+22.5)}) circle (2pt);
   \draw[fill=white] ({3/2*cos(225+22.5)},{3/2*sin(225+22.5)}) circle (2pt);
   \draw[fill=white] ({3/2*cos(270+22.5)},{3/2*sin(270+22.5)}) circle (2pt);
   \draw[fill=white] ({3/2*cos(315+22.5)},{3/2*sin(315+22.5)}) circle (2pt);
   \draw[fill=black] ({2*cos(0+22.5)},{2*sin(0+22.5)}) circle (2pt);
   \draw[fill=black] ({2*cos(45+22.5)},{2*sin(45+22.5)}) circle (2pt);
   \draw[fill=black] ({2*cos(90+22.5)},{2*sin(90+22.5)}) circle (2pt);
   \draw[fill=black] ({2*cos(135+22.5)},{2*sin(135+22.5)}) circle (2pt);
   \draw[fill=black] ({2*cos(180+22.5)},{2*sin(180+22.5)}) circle (2pt);
   \draw[fill=black] ({2*cos(225+22.5)},{2*sin(225+22.5)}) circle (2pt);
   \draw[fill=black] ({2*cos(270+22.5)},{2*sin(270+22.5)}) circle (2pt);
   \draw[fill=black] ({2*cos(315+22.5)},{2*sin(315+22.5)}) circle (2pt);
   
   \begin{scope}[xshift = 6cm]
\foreach \i in {1,...,8}
{
   \draw ({2*cos(45*\i+22.5)},{2*sin(45*\i+22.5)}) -- ({2*cos(45*(\i+1)+22.5)},{2*sin(45*(\i+1)+22.5)});
}
   \draw ({2*cos(45*0+22.5)},{2*sin(45*0+22.5)}) -- ({2*cos(45*4+22.5)},{2*sin(45*4+22.5)});
   \draw ({2*cos(45*1+22.5)},{2*sin(45*1+22.5)}) -- ({2*cos(45*5+22.5)},{2*sin(45*5+22.5)});
   \draw ({2*cos(45*2+22.5)},{2*sin(45*2+22.5)}) -- ({2*cos(45*6+22.5)},{2*sin(45*6+22.5)});
   \draw[white, fill=white] (0,0) circle (5pt);
   \draw ({2*cos(45*3+22.5)},{2*sin(45*3+22.5)}) -- ({2*cos(45*7+22.5)},{2*sin(45*7+22.5)});
   
   \draw[white,fill=white] ({0.71*cos(315+22.5)},{0.71*sin(315+22.5)}) circle (3pt);
   \draw[white,fill=white] ({0.71*cos(180+22.5)},{0.71*sin(180+22.5)}) circle (3pt);
   
   \draw ({cos(0+22.5)},{sin(0+22.5)}) -- ({cos(270+22.5)},{sin(270+22.5)});
   \draw ({cos(135+22.5)},{sin(135+22.5)}) -- ({cos(225+22.5)},{sin(225+22.5)});
   
   \draw[fill=black] ({cos(0+22.5)},{sin(0+22.5)}) circle (2pt);
   \draw[fill=black] ({cos(135+22.5)},{sin(135+22.5)}) circle (2pt);
   \draw[fill=black] ({cos(225+22.5)},{sin(225+22.5)}) circle (2pt);
   \draw[fill=black] ({cos(270+22.5)},{sin(270+22.5)}) circle (2pt);
   \draw[fill=white] ({1.84*cos(0)},{1.84*sin(0)}) circle (2pt);
   \draw[fill=white] ({1.84*cos(45)},{1.84*sin(45)}) circle (2pt);
   \draw[fill=white] ({1.84*cos(90)},{1.84*sin(90)}) circle (2pt);
   \draw[fill=white] ({1.84*cos(135)},{1.84*sin(135)}) circle (2pt);
   \draw[fill=white] ({1.84*cos(180)},{1.84*sin(180)}) circle (2pt);
   \draw[fill=white] ({1.84*cos(225)},{1.84*sin(225)}) circle (2pt);
   \draw[fill=white] ({1.84*cos(270)},{1.84*sin(270)}) circle (2pt);
   \draw[fill=white] ({1.84*cos(315)},{1.84*sin(315)}) circle (2pt);
   \draw[fill=black] ({2*cos(0+22.5)},{2*sin(0+22.5)}) circle (2pt);
   \draw[fill=black] ({2*cos(45+22.5)},{2*sin(45+22.5)}) circle (2pt);
   \draw[fill=black] ({2*cos(90+22.5)},{2*sin(90+22.5)}) circle (2pt);
   \draw[fill=black] ({2*cos(135+22.5)},{2*sin(135+22.5)}) circle (2pt);
   \draw[fill=black] ({2*cos(180+22.5)},{2*sin(180+22.5)}) circle (2pt);
   \draw[fill=black] ({2*cos(225+22.5)},{2*sin(225+22.5)}) circle (2pt);
   \draw[fill=black] ({2*cos(270+22.5)},{2*sin(270+22.5)}) circle (2pt);
   \draw[fill=black] ({2*cos(315+22.5)},{2*sin(315+22.5)}) circle (2pt);
   
   \end{scope}
\end{tikzpicture}
\caption{(a) The configuration $\Sigma$ of 24 $(-2)$-curves in the boundary of $w_Q^{-1}(t)$. (b) The configuration $\Sigma$ of 20 $(-2)$-curves in the boundary of $w_P^{-1}(t)$. }
\label{fig!minus-2-curves}
\end{center}
\end{figure}

Returning to our original Landau--Ginzburg model $w_Q\colon U\to C$, we see that the fibres are a family of affine K3 surfaces $w_Q^{-1}(t) = S_t\setminus \Sigma$, given by the complement of this configuration of lines $\Sigma$ in $S_t$. These can be represented by as affine quartics
\[ (t + 5)x_1x_2x_3 = (1 + x_1 + x_2)(1 + x_2 + x_3)\left(1 + x_1 + x_2 + x_3 + x_1x_3\right) \]
and there are precisely four degenerate fibres. Two reducible fibres corresponding to the values $E_Q=w_Q^{-1}(-5)$ and $D_Q=w_Q^{-1}(\infty)$, and two values $\lambda_\pm = 24(1 \pm \sqrt{2})$ for which the Picard rank jumps to 20. We get a picture as in the 2-dimensional case.
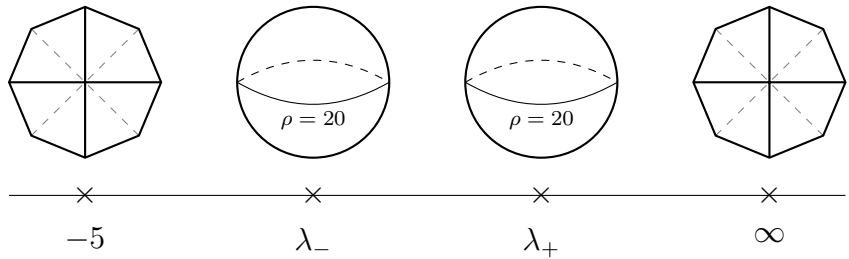
\begin{center}\begin{tikzpicture}
   \begin{scope}
   \draw[gray,dashed] ({cos(45)},{2+sin(45)}) -- ({cos(225)},{2+sin(225)})  ({cos(135)},{2+sin(135)}) -- ({cos(-45)},{2+sin(-45)});
   \draw[thick] ({cos(0)},{2+sin(0)}) -- ({cos(45)},{2+sin(45)}) -- ({cos(90)},{2+sin(90)}) -- ({cos(135)},{2+sin(135)}) -- ({cos(180)},{2+sin(180)}) -- ({cos(225)},{2+sin(225)}) -- ({cos(270)},{2+sin(270)}) -- ({cos(315)},{2+sin(315)}) -- cycle;
   \draw[thick] ({cos(0)},{2+sin(0)}) -- ({cos(180)},{2+sin(180)}) ({cos(90)},{2+sin(90)}) -- ({cos(270)},{2+sin(270)});
   \end{scope}
   
   \begin{scope}[xshift=-1cm]
   \draw[thick] (4,2) circle (1cm);
   \draw[dashed] (3,2) to [in = 150, out = 30] (5,2);
   \draw (3,2) to [in = -150, out = -30] (5,2);
   \end{scope}
   
   \begin{scope}[xshift=2cm]
   \draw[thick] (4,2) circle (1cm);
   \draw[dashed] (3,2) to [in = 150, out = 30] (5,2);
   \draw (3,2) to [in = -150, out = -30] (5,2);
   \end{scope}
   
   \begin{scope}[xshift=9cm]
   \draw[gray,dashed] ({cos(45)},{2+sin(45)}) -- ({cos(225)},{2+sin(225)})  ({cos(135)},{2+sin(135)}) -- ({cos(-45)},{2+sin(-45)});
   \draw[thick] ({cos(0)},{2+sin(0)}) -- ({cos(45)},{2+sin(45)}) -- ({cos(90)},{2+sin(90)}) -- ({cos(135)},{2+sin(135)}) -- ({cos(180)},{2+sin(180)}) -- ({cos(225)},{2+sin(225)}) -- ({cos(270)},{2+sin(270)}) -- ({cos(315)},{2+sin(315)}) -- cycle;
   \draw[thick] ({cos(0)},{2+sin(0)}) -- ({cos(180)},{2+sin(180)}) ({cos(90)},{2+sin(90)}) -- ({cos(270)},{2+sin(270)});
   \end{scope}

   \draw (-1,0.5) -- (10,0.5);
   \node at (0,0.5) [label={below:$-5$}] {$\times$};
   \node at (3,0.5) [label={below:$\lambda_-$}] {$\times$};
   \node at (6,0.5) [label={below:$\lambda_+$}] {$\times$};
   \node at (9,0.5) [label={below:$\infty$}] {$\times$};
   \node at (3,1.5) {\scriptsize $\rho=20$};
   \node at (6,1.5) {\scriptsize $\rho=20$};
\end{tikzpicture}\end{center}

\paragraph{The potential $w_P$.} 
We can rerun this entire analysis with the potential $w_P=x_1+\ldots+x_8$ in place of $w_Q$. In summary, the period of $w_P$ is
\[ \pi_{w_P}(t) = 1  + 24t^2 + 192t^3 + 2904t^4 + 40320t^5 + 611520t^6 + 9515520t^7 + \cdots \]
which agrees with the regularised quantum period $\widehat{G}_X(t)$ for the Fano 3-fold $X$ of type $V_{16}$, as listed in \cite[\S15]{ccgk}. There is a factorisation
\[ (w_P + 4)|_{\TT_{123}} = \frac{(1 + x_1)(1 + x_2)(1 + x_3)(1 + x_1 + x_2 + x_3)}{x_1x_2x_3} \]
which is another $\sigma_3$-invariant Laurent polynomial, and corresponds to a reducible fibre given by $E_P := w_P^{-1}(-4) = \bigcup_{i=1}^4F_{i,i+4}$ in the notation of Proposition~\ref{prop!div}. When extended to the compactification $X_P$, the boundary divisor $D_P$ meets $E_P$ in a collection of 12 lines, given by the $\Dih_8$-orbits of $\VV(x_1+x_2+x_3)\subset D_{123}$ and $\VV(x_1+x_3)\subset D_{1357}$. These are displayed (tropically) in Figure~\ref{fig!D-and-E}(b) and pass through the same eight nodes in the boundary of $X_P$ as the previous case. The general member in the anticanonical pencil $|D_P,E_P|$, corresponding to the fibres of $w_Q$, is a K3 surface $S_t = w_Q^{-1}(t)$ which passes through the twelve lines and is smooth apart from eight ordinary double points at the eight nodes of $X_P$. The minimal resolution of $S_t$ is a K3 surface containing a configuration $\Sigma$ of 20 $(-2)$-curves generating a lattice of Picard rank 19, and which correspond to the dual intersection diagram of Figure~\ref{fig!minus-2-curves}(b). The fibres of the original Landau--Ginzburg model $w_P\colon U\to \CC$ are affine K3 surfaces of the form $S_t\setminus \Sigma$. There are exactly four degenerate fibres: the two reducible fibres $D_P=w_Q^{-1}(\infty)$ and $E_P=w_Q^{-1}(-4)$, and a further two values where the Picard rank jumps to 20.

\subsubsection{Other Fano 3-folds}

We do not try to attempt to classify all reflexive polytopes in $N_U$, although it would be interesting to know how many there are in comparison to the 4319 reflexive polytopes in $\RR^3$. Nevertheless, we can consider all of the polytopes obtained from removing vertices from the polytope $Q$ from \S\ref{sec!LG-V12-V16}.

\begin{thm}
Consider the Laurent polynomial $w_0=w|_{\TT_{123}}\in\CC[x_1^{\pm1},x_2^{\pm1},x_3^{\pm1}]$ obtained by restricting the Landau--Ginzburg potential $w = \varepsilon_1x_1 + \ldots + \varepsilon_8x_8+\varepsilon_9q_1+\varepsilon_{10}q_2$ with coefficients $\varepsilon_i\in\{0,1\}$ for $i=1,\ldots,10$.
\begin{enumerate}
\item Of the 1024 possibilities for $w_0$, 705 have 3-dimensional Fano Newton polytopes, i.e.\ $\Newt(w_0)$ has primitive vertices and $0\in \Int(\Newt(w_0))$.
\item These 705 Laurent polynomials give rise to 46 distinct periods.
\item Of these 46 periods, 20 are equal to the regularised quantum period of a smooth Fano 3-fold, as described in Table~\ref{tab!LPmirrors}.
\end{enumerate} 
\end{thm}

\begin{table}[ht]
\caption{Mirror Landau--Ginzburg potentials on $U$ for 20 smooth Fano 3-folds.}
\begin{center}
\renewcommand{\arraystretch}{1.2}
\begin{tabular}{|c|l|} \hline
Fano 3-fold & Mirror Landau--Ginzburg potential $w$ \\ \hline
$V_{12}$ & $x_1+x_2+x_3+x_4+x_5+x_6+x_7+x_8+q_1+q_2$ \\
$V_{14}$ & $x_1+x_2+x_3+x_4+x_5+x_6+x_7+x_8+q_1$ \\ 
$V_{16}$ & $x_1+x_2+x_3+x_4+x_5+x_6+x_7+x_8$ \\
$V_{18}$ & $x_1+x_2+x_3+x_4+x_5+x_6+x_7$ \\
$V_{22}$ & $x_1+x_2+x_3+x_4+x_5+x_6$ \\ \hline
$\MM_{2,9}$ & $x_1 + x_2 + x_3 + x_6 + q_1 + q_2$ \\
$\MM_{2,12}$ & $x_1 + x_2 + x_3 + x_5 + x_6 + x_7$ \\
$\MM_{2,13}$ & $x_1 + x_2 + x_3 + x_4 + x_6 + x_7$ \\
$\MM_{2,14}$ & $x_1 + x_2 + x_3 + x_4 + x_5 + x_7 + q_1$ \\
$\MM_{2,16}$ & $x_1 + x_2 + x_3 + x_6 + q_1$ \\
$\MM_{2,17}$ & $x_1 + x_2 + x_3 + x_5 + x_6$ \\
$\MM_{2,20}$ & $x_1 + x_2 + x_3 + x_4 + x_7$ \\
$\MM_{2,21}$ & $x_1 + x_2 + x_3 + x_5 + x_7$ \\
$\MM_{2,22}$ & $x_1 + x_2 + x_3 + x_6$ \\ \hline
$\MM_{3,7}$ & $x_1 + x_2 + x_4 + x_6 + x_7$ \\
$\MM_{3,10}$ & $x_1 + x_2 + x_3 + x_5 + x_7 + q_1$ \\
$\MM_{3,12}$ & $x_1 + x_3 + x_6 + x_7 + q_1$ \\
$\MM_{3,13}$ & $x_1 + x_5 + q_2$ \\
$\MM_{3,15}$ & $x_1 + x_4 + x_5 + x_7$ \\ 
$\MM_{3,20}$ & $x_1 + x_4 + x_6$ \\ \hline
\end{tabular}
\end{center}
\label{tab!LPmirrors}
\end{table}%

Table~\ref{tab!LPmirrors} only contains one representative Landau--Ginzburg potential in each case and, even after taking the $\Dih_8$ symmetry into account, the same period sequence can be obtained from several different potentials. For example, in addition to the entry in the Table~\ref{tab!LPmirrors}, the following five potentials also have the right period to be mirror to $V_{22}$:
\[ x_1+x_2+x_5+x_6+q_2, \quad x_1+x_2+x_3+x_4+x_5+q_2, \quad x_1+x_2+q_1+q_2, \]
\[ x_1+x_2+x_3+x_4+x_5+x_7 \quad \text{and} \quad x_1+x_2+x_3+x_4+x_6+q_2. \]
According to Conjecture~\ref{conj!my-conj}, the different choices of potential correspond to different degenerations of a Fano 3-fold of type $V_{22}$ to a pair $(X_Q,D_Q)$ where $Q$ is the dual polytope to $P=\Newt(w)\subset M_U$. It would be interesting to know whether the remaining 26 period sequences are all period sequences of Fano 3-folds with terminal singularities.

\subsubsection{A final example}

We end with a final example of mirror duality for a pair of exotic polytopes in $N_U$, which cannot be realised combinatorially as polytopes in Euclidean space. The intersection of the two halfspaces $(q_1)^{\geq-1}\cap (q_2)^{\geq-1}$ is a closed polytope $Q\subset N_U$ with eight vertices $x_ix_{i+1}$ for $i\in \ZZ/8\ZZ$. Therefore the dual polytope $P$ is given by $P = \bigcap_{i\in\ZZ/8\ZZ} (x_ix_{i+1})^{\geq-1}$. This gives a pair of dual polytopes in $N_U$ 
\[ P=\conv(q_1, q_2) \qquad \text{and} \qquad Q=P^\star=\conv(x_1x_2, x_2x_3, \ldots, x_8x_1). \]
We draw these two polytopes below, where the solid lines denote edges of the polytope and dashed lines are due to the bend in the affine structure of $N_U$. We note that $P$ has two vertices and eight faces (so topologically it looks like a beachball with eight stripes), whereas $Q$ has eight vertices and two faces. 

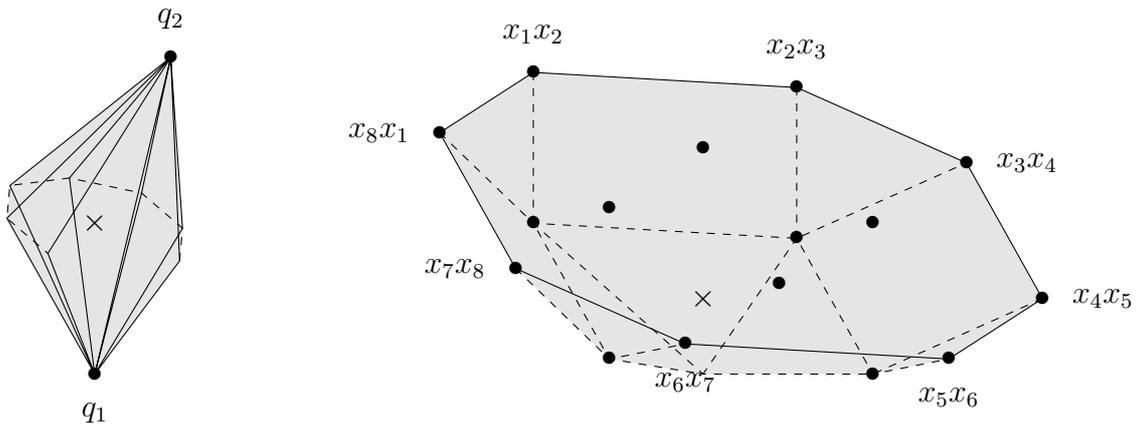
\begin{center}
\begin{tikzpicture}[scale=2]
   \begin{scope}[xshift=0cm, yshift=0.5cm]
   \node[inner sep = -2pt] (o) at (0,0) {};
   \node (x) at ({cos(330)+0.25},{sin(330)}) {};
   \node (x2) at ({(cos(330)+0.25)/2},{sin(330)/2}) {};
   \node (x3) at ({(cos(330)+0.25)/3},{sin(330)/3}) {};
   \node (y) at ({cos(90)},{sin(90)}) {};
   \node (z) at ({cos(210)+0.25},{sin(210)+0.1}) {};
   \node (z2) at ({(cos(210)+0.25)/2},{(sin(210)+0.1)/2}) {};
   \node (z3) at ({(cos(210)+0.25)/3},{(sin(210)+0.1)/3}) {};
   
   \fill[gray, opacity=0.2] ($(x)+(y)+(y)+(z)$) -- ($(x3)-(z3)$) -- ($(x2)$) -- ($(o)-(y)$) -- ($(z3)-(x3)$) -- ($(o)-(x2)$) -- cycle;
   
   \draw ($(x)+(y)+(y)+(z)$) -- ($(o)-(y)$)
         ($(x)+(y)+(y)+(z)$) -- ($(x2)$) -- ($(o)-(y)$)
         ($(x)+(y)+(y)+(z)$) -- ($(x3)-(z3)$) -- ($(o)-(y)$)
         ($(x)+(y)+(y)+(z)$) -- ($(o)-(z2)$) -- ($(o)-(y)$)
         ($(x)+(y)+(y)+(z)$) -- ($(o)-(x3)-(z3)$) -- ($(o)-(y)$)
         ($(x)+(y)+(y)+(z)$) -- ($(o)-(x2)$) -- ($(o)-(y)$)
         ($(x)+(y)+(y)+(z)$) -- ($(z3)-(x3)$) -- ($(o)-(y)$)
         ($(x)+(y)+(y)+(z)$) -- ($(z2)$) -- ($(o)-(y)$);
   \draw[dashed] ($(x2)$) -- ($(x3)-(z3)$) -- ($(o)-(z2)$) -- ($(o)-(x3)-(z3)$) -- ($(o)-(x2)$) -- ($(z3)-(x3)$) -- ($(z2)$);
   
   \node at ($(o)$) {$\times$};
   \node at ($(o)-(y)$) [label={below:$q_1$}] {$\bullet$};
   \node at ($(x)+(y)+(y)+(z)$) [label={above:$q_2$}] {$\bullet$};
   \end{scope}
   
   \begin{scope}[xshift=4cm]
   \node[inner sep = -2pt] (o) at (0,0) {};
   \node (x) at ({cos(330)+0.25},{sin(330)}) {};
   \node (y) at ({cos(90)},{sin(90)}) {};
   \node (y2) at ({cos(90)/2},{sin(90)/2}) {};
   \node (z) at ({cos(210)+0.25},{sin(210)+0.1}) {};
   
   \fill[gray, opacity=0.2] ($(z)$) -- ($(o)-(y2)$) -- ($(x)$) -- ($(x)+(x)+(y)+(z)$) -- ($(x)+(x)+(y)$) -- ($(x)+(y)-(z)$) -- ($(y)-(z)$) -- ($(y)-(x)$) -- ($(y)+(z)-(x)$) -- ($(y)+(z)+(z)$) -- cycle;
   
   \draw ($(y)-(x)$) -- ($(y)+(z)-(x)$) -- ($(y)+(z)+(z)$) -- ($(x)+(y)+(z)+(z)$) -- ($(x)+(x)+(y)+(z)$) -- ($(x)+(x)+(y)$) -- ($(x)+(y)-(z)$) -- ($(y)-(z)$) -- cycle;
   \draw[dashed] ($(y)-(x)$) -- ($(o)-(x)$) -- ($(z)$) -- ($(x)+(y)+(z)+(z)$)
                 ($(y)-(z)$) -- ($(o)-(z)$) -- ($(x)$) -- ($(x)+(x)+(y)+(z)$)
                 ($(y)+(z)-(x)$) -- ($(o)-(x)$) -- ($(o)-(z)$) -- ($(x)+(y)-(z)$)
                 ($(y)+(z)+(z)$) -- ($(z)$) -- ($(o)-(y2)$) -- ($(o)-(z)$)
                 ($(y)+(x)+(x)$) -- ($(x)$) -- ($(o)-(y2)$) -- ($(o)-(x)$);
   
   \node at ($(o)-(x)$) {$\bullet$};
   \node at ($(o)-(z)$) {$\bullet$};
   \node at ($(x)$) {$\bullet$};
   \node at ($(z)$) {$\bullet$};
   \node at ($(y)+(z)$) {$\bullet$};
   \node at ($(x)+(y)$) {$\bullet$};
   \node at ($(x)+(y)+(z)$) {$\bullet$};
   \node at ($(y)$) {$\bullet$};
   
   \node at ($(o)$) {$\times$};
   
   \node at ($(y)-(x)$) [label={above:$x_1x_2$}] {$\bullet$};
   \node at ($(y)+(z)-(x)$) [label={left:$x_8x_1$}] {$\bullet$};
   \node at ($(y)+(z)+(z)$) [label={left:$x_7x_8$}] {$\bullet$};
   \node at ($(x)+(y)+(z)+(z)$) [label={below:$x_6x_7$}] {$\bullet$};
   \node at ($(x)+(x)+(y)+(z)$) [label={below:$x_5x_6$}] {$\bullet$};
   \node at ($(x)+(x)+(y)$) [label={right:$x_4x_5$}] {$\bullet$};
   \node at ($(x)+(y)-(z)$) [label={right:$x_3x_4$}] {$\bullet$};
   \node at ($(y)-(z)$) [label={above:$x_2x_3$}]{$\bullet$};
   \end{scope}
\end{tikzpicture} 
\end{center}
We now describe how to set up the mirror correspondence between the (degenerate) Fano 3-fold $(X_P,D_P)$ (resp.\ $(X_Q,D_Q)$) and the Landau--Ginzburg model $(U,w_Q)$ (resp.\ $(U,w_P)$).

\begin{eg} \label{eg!last-one}
We start by describing the mirror correspondence for $(X_Q,D_Q)$ and $(U,w_P)$.

\paragraph{The Landau--Ginzburg model $(U,w_P)$.}
Using $P$ as the polytope to define a Landau--Ginzburg model $(U,w_P)$ is straightforward, since $P$ contains only two nonzero lattice points which are both vertices. By Remark~\ref{rmk!binomial-edge-coeffs} they both receive coefficient 1 and this uniquely determines the potential $w_P=q_1+q_2$. By restricting to the torus chart $\TT_{123}$ and computing the period sequence of the resulting Laurent polynomial, we find that
\[ \pi_{w_P}(t) = 1 + 8t^2 + 24t^3 + 240t^4 + 1440t^5 + 11960t^6 + 89040t^7 + \cdots \]
which is the regularised quantum period of the smooth Fano 3-fold $\MM_{2,21}$ \cite[\S38]{ccgk}.

\paragraph{The Fano 3-fold $(X_Q,D_Q)$.}
When considered in the integral affine structure of $M_U$, the two faces of $Q$ are two flat lattice octagons labelled with the following theta functions.
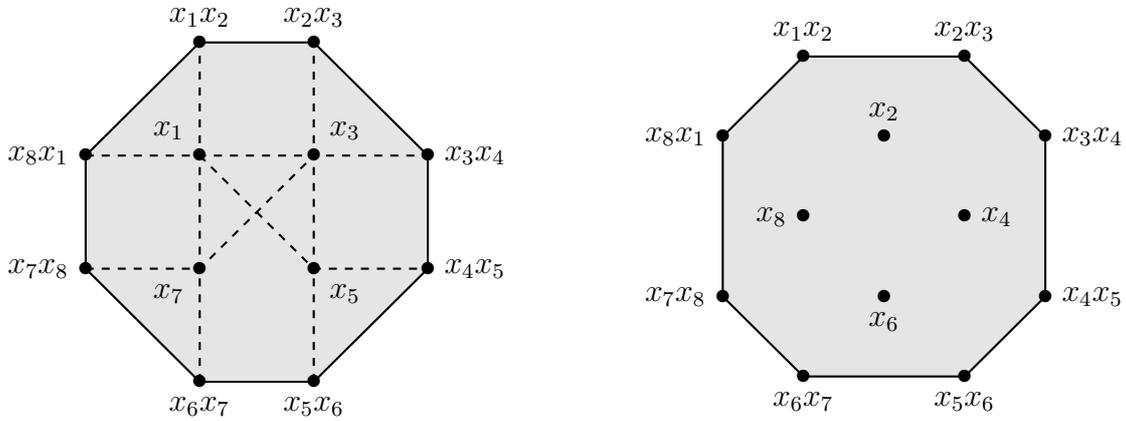
\begin{center}
\begin{tikzpicture}[scale=1.5]

    \draw[thick,fill=gray,fill opacity=0.2] (1,0) -- (2,0) -- (3,1) -- (3,2) -- (2,3) -- (1,3) -- (0,2) -- (0,1) -- cycle;
    \draw[thick,dashed] (0,2) -- (3,2) (1,0) -- (1,3) (2,0) -- (2,3) (0,1)  -- (1,1) -- (2,2) (1,2) -- (2,1) -- (3,1);
    \node at (1,0) [label={[label distance=-5pt]-90:$x_6x_7$}]{$\bullet$}; 
    \node at (2,0) [label={[label distance=-5pt]-90:$x_5x_6$}]{$\bullet$}; 
    \node at (3,1) [label={[label distance=-5pt]0:  $x_4x_5$}]{$\bullet$}; 
    \node at (3,2) [label={[label distance=-5pt]0:  $x_3x_4$}]{$\bullet$}; 
    \node at (2,3) [label={[label distance=-5pt]90: $x_2x_3$}]{$\bullet$}; 
    \node at (1,3) [label={[label distance=-5pt]90: $x_1x_2$}]{$\bullet$}; 
    \node at (0,2) [label={[label distance=-5pt]180:$x_8x_1$}]{$\bullet$}; 
    \node at (0,1) [label={[label distance=-5pt]180:$x_7x_8$}]{$\bullet$}; 
    
    \node at (1,1) [label={[label distance=-7pt]-135:$x_7$}]{$\bullet$}; 
    \node at (2,1) [label={[label distance=-7pt]-45: $x_5$}]{$\bullet$}; 
    \node at (2,2) [label={[label distance=-7pt]45:  $x_3$}]{$\bullet$}; 
    \node at (1,2) [label={[label distance=-7pt]135: $x_1$}]{$\bullet$};

    \begin{scope}[xshift = 7cm, yshift=-0.66cm, rotate=45]
    \draw[thick,fill=gray,fill opacity=0.2] (1,0) -- (2,0) -- (3,1) -- (3,2) -- (2,3) -- (1,3) -- (0,2) -- (0,1) -- cycle;
    
    \node at (1,0) [label={[label distance=-5pt]-90:$x_5x_6$}]{$\bullet$}; 
    \node at (2,0) [label={[label distance=-5pt]0:$x_4x_5$}]{$\bullet$}; 
    \node at (3,1) [label={[label distance=-5pt]0:  $x_3x_4$}]{$\bullet$}; 
    \node at (3,2) [label={[label distance=-5pt]90:  $x_2x_3$}]{$\bullet$}; 
    \node at (2,3) [label={[label distance=-5pt]90: $x_1x_2$}]{$\bullet$}; 
    \node at (1,3) [label={[label distance=-5pt]180: $x_8x_1$}]{$\bullet$}; 
    \node at (0,2) [label={[label distance=-5pt]180:$x_7x_8$}]{$\bullet$}; 
    \node at (0,1) [label={[label distance=-5pt]-90:$x_6x_7$}]{$\bullet$}; 
    
    \node at (1,1) [label={[label distance=-5pt]-90:$x_6$}]{$\bullet$}; 
    \node at (2,1) [label={[label distance=-5pt]0: $x_4$}]{$\bullet$}; 
    \node at (2,2) [label={[label distance=-5pt]90:  $x_2$}]{$\bullet$}; 
    \node at (1,2) [label={[label distance=-5pt]180: $x_8$}]{$\bullet$}; 
    \end{scope}
\end{tikzpicture} 
\end{center}
Thus to construct $(X_Q,D_Q)$ we consider the graded ring $R_Q$ which is generated in degree 1 by seventeen generators 
\[ x_0, \: x_1, \; \ldots,\; x_8,\; x_1x_2,\; \ldots,\; x_8x_1 \]
where $x_0$ is the homogenising variable. According to computer algebra $X_Q=\Proj R_Q\subset \PP^{16}$ is a Gorenstein Fano variety of Fano index 1 and degree $\deg X_Q=\vol Q=28$. This at least agrees with $X_Q$ being a (degeneration of a) Fano 3-fold of type $\MM_{2,21}$, as predicted by the period sequence of $w_P$. By Conjecture~\ref{conj!my-conj} we expect that $X_Q$ admits a $\QQ$-Gorenstein smoothing to a smooth Fano 3-fold of type $\MM_{2,21}$.

Moreover, the boundary divisor $D_Q=\VV(x_0)$ consists of two components $D=D_1\cup D_2$ which are toric surfaces (defined by the two octagons above) which glued together other along an octagon of rational curves. This realises $Q$ as the intersection complex of $D_Q$.
\end{eg}

\begin{eg}
The more interesting (and currently less clear) direction is to understand how to set up the mirror correspondence with the roles of $P$ and $Q$ reversed.

\paragraph{The Fano 3-fold $(X_P,D_P)$.}
The polytope $P$ gives degrees $2,2,\ldots,2,1,1$ to the variables $x_1,x_2,\ldots,x_8,q_1,q_2$ respectively. Taking the projective closure of $U$ with respect to this grading we get a Fano 3-fold $X_P=\Proj R_P\subset \PP(1^3,2^8)$ in weighted projective space. One of the equations defining $X_P$ is
\[ q_1q_2 = x_1 + x_2 + x_3 + x_4 + x_5 + x_6 + x_7 + x_8 + 4x_0^2 \]
which we could use to eliminate $x_1$ say, and reduce to a model $X_P\subset \PP(1^3,2^7)$. The Hilbert series of $X_P$ now agrees with \cite[\#14885]{grdb}, showing that $X_P$ has degree $\deg(X_P)=\vol(P)=4$, Fano index~1 and basket of quotient singularities $\{8\times\tfrac12(1,1,1)\}$. The anticanonical boundary divisor $D_P=\VV(x_0)$ has eight components:
\[ D_{12} = \VV(x_0,x_3,x_4,x_5,x_6,x_7,x_8,q_1q_2 - x_1 - x_2) \qquad \text{etc.} \]
which are all copies of $\PP(1,1,2)$ that are joined to their neighbours along conics meeting at the two coordinate points where $q_1\neq0$ and $q_2\neq0$. Thus $P$ is realised as the intersection complex for the boundary divisor $D_P$. One of the eight $\tfrac12(1,1,1)$ points of $X_P$ lies in the relative interior of $D_{12}$ at the point where $x_1+x_2\neq 0$ and all of the other coordinates vanish. Similarly for the other seven $\tfrac12(1,1,1)$ points.

We note that if $Y$ is a Fano 3-fold with Hilbert series \cite[\#14885]{grdb}, the $h^0(Y,{-K_Y})=3$ and hence $-K_Y$ is ample but not very ample, so there is no \emph{Gorenstein} toric degeneration of~$Y$. Therefore there is no reflexive polytope we can write down which will give us a mirror Laurent polynomial for $Y$.

\paragraph{The Landau--Ginzburg model $(U,w_Q)$.}
To cook up a Landau--Ginzburg potential $w_Q$ which is mirror to $(X_P,D_P)$, we need to label the lattice points of $Q$ with appropriate integer coefficients, as described in Remark~\ref{rmk!binomial-edge-coeffs}. Since we label all vertices of $Q$ with the coefficient 1, the only decisions we need to make are over the lattice points belonging to the relative interior of each face. Following \cite{cfp}, we expect that the right thing to do is label the lattice points of each face $F\subset Q$ with the coefficients of a \emph{0-mutable polynomial} supported on $F$. In our case this is one of the following three choices of coefficients, according to the three possible maximal Minkowski decompositions of the octagon. 
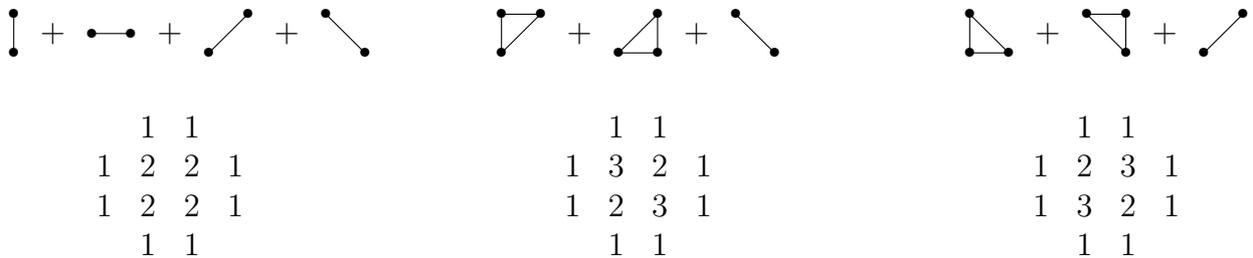
\begin{center}
\resizebox{\textwidth}{!}{\begin{tikzpicture} 
   \node at (0,0.25) {$\begin{matrix}&1&1&\\1&2&2&1\\1&2&2&1\\&1&1&\end{matrix}$};
   \node at (6,0.25) {$\begin{matrix}&1&1&\\1&3&2&1\\1&2&3&1\\&1&1&\end{matrix}$};
   \node at (12,0.25) {$\begin{matrix}&1&1&\\1&2&3&1\\1&3&2&1\\&1&1&\end{matrix}$}; 
   
   \draw (-2,2) -- (-2,2.5);
   \draw (-1,2.25) -- (-0.5,2.25);
   \draw (0.5,2) -- (1,2.5);
   \draw (2.5,2) -- (2,2.5);
   
   \node at (-2,2) {\scriptsize $\bullet$};
   \node at (-2,2.5) {\scriptsize $\bullet$};
   \node at (-1,2.25) {\scriptsize $\bullet$};
   \node at (-0.5,2.25) {\scriptsize $\bullet$};
   \node at (0.5,2) {\scriptsize $\bullet$};
   \node at (1,2.5) {\scriptsize $\bullet$};
   \node at (2.5,2) {\scriptsize $\bullet$};
   \node at (2,2.5) {\scriptsize $\bullet$};
   
   \node at (-1.5,2.25) {$+$};
   \node at (0,2.25) {$+$};
   \node at (1.5,2.25) {$+$};
   
   \begin{scope}[xshift = 6cm]
   \draw (-1.75,2) -- (-1.75,2.5) -- (-1.25,2.5) -- cycle;
   \draw (-0.25,2) -- (0.25,2) -- (0.25,2.5) -- cycle;
   \draw (1.75,2) -- (1.25,2.5) -- cycle;
   
   \node at (-1.75,2) {\scriptsize $\bullet$};
   \node at (-1.75,2.5) {\scriptsize $\bullet$};
   \node at (-1.25,2.5) {\scriptsize $\bullet$};
   \node at (-0.25,2) {\scriptsize $\bullet$};
   \node at ( 0.25,2) {\scriptsize $\bullet$};
   \node at ( 0.25,2.5) {\scriptsize $\bullet$};
   \node at ( 1.75,2) {\scriptsize $\bullet$};
   \node at ( 1.25,2.5) {\scriptsize $\bullet$};
   
   \node at (-0.75,2.25) {$+$};
   \node at (0.75,2.25) {$+$};
   \end{scope}
   
   \begin{scope}[xshift = 12cm]
   \draw (-1.75,2) -- (-1.75,2.5) -- (-1.25,2) -- cycle;
   \draw (-0.25,2.5) -- (0.25,2) -- (0.25,2.5) -- cycle;
   \draw (1.25,2) -- (1.75,2.5) -- cycle;
   
   \node at (-1.75,2) {\scriptsize $\bullet$};
   \node at (-1.75,2.5) {\scriptsize $\bullet$};
   \node at (-1.25,2) {\scriptsize $\bullet$};
   \node at (-0.25,2.5) {\scriptsize $\bullet$};
   \node at ( 0.25,2) {\scriptsize $\bullet$};
   \node at ( 0.25,2.5) {\scriptsize $\bullet$};
   \node at ( 1.25,2) {\scriptsize $\bullet$};
   \node at ( 1.75,2.5) {\scriptsize $\bullet$};
   
   \node at (-0.75,2.25) {$+$};
   \node at (0.75,2.25) {$+$};
   \end{scope}
\end{tikzpicture}}\end{center}
After taking into account the $\Dih_8$ symmetry, there are just three ways to decorate $Q$ with coefficients, yielding three Landau--Ginzburg potentials with distinct periods. As described in Remark~\ref{rmk!binomial-edge-coeffs}, according to Conjecture~\ref{conj!my-conj} these are conjecturally in one-to-one correspondence with three possible deformations of $X_P$. None of these three periods appear in the classification of quantum periods for smooth Fano 3-folds~\cite{ccgk}, which is unsurprising since $X_P$ has isolated terminal quotient singularities (which cannot be deformed away by any smoothing). However it suggests that $X_P$ may lie at the intersection of three different components of Fano 3-folds in its Hilbert scheme.
\end{eg}

\end{document}